\documentclass[11pt,a4paper]{article}
\usepackage{todonotes}

\newcommand{\vect}[1]{\boldsymbol{#1}}
\newcommand{\vectt}[1]{\boldsymbol{\mathbf{#1}}}
\newcommand{\tens}[1]{\pmb{\mathsf{#1}}}
\newcommand{\dx}{\mathrm{d}x}
\newcommand{\dy}{\mathrm{d}y}

\newcommand{\supp}{\mathrm{supp}}

\newcommand{\Hil}{\mathcal{H}}
\newcommand{\rvline}{\hspace*{-\arraycolsep}\vline\hspace*{-\arraycolsep}}

\newcommand{\eT}{\tilde{T}}
\newcommand{\eU}{\tilde{U}}
\newcommand{\Sk}{\vectt{S}^{\vect{I}}_{\vect{n}}}
\newcommand{\Skp}{\vectt{S}^{\vect{I},+}_{\vect{n}}}
\newcommand{\Skd}{\vectt{S}^{\vect{I},*}_{\vect{n}}}
\newcommand{\fdx}{\frac{\mathrm{d}}{\dx}}
\newcommand{\sgn}{\mathrm{sgn}}
\newcommand{\oFo}{{_1}F_1}
\newcommand{\mi}{\mathrm{i}}
\newcommand{\me}{\mathrm{e}}
\newcommand{\I}{\mathrm{i}}
\newcommand{\E}{\mathrm{e}}
\newcommand{\hT}{\hat{T}}
\newcommand{\hW}{\hat{W}}

\def\Xint#1{\mathchoice
{\XXint\displaystyle\textstyle{#1}}%
{\XXint\textstyle\scriptstyle{#1}}%
{\XXint\scriptstyle\scriptscriptstyle{#1}}%
{\XXint\scriptscriptstyle\scriptscriptstyle{#1}}%
\!\int}
\def\XXint#1#2#3{{\setbox0=\hbox{$#1{#2#3}{\int}$}
\vcenter{\hbox{$#2#3$}}\kern-.5\wd0}}

\def\dashint{\Xint-}

\usepackage{color}
\definecolor{deepblue}{rgb}{0,0,0.5}
\definecolor{deepred}{rgb}{0.6,0,0}
\definecolor{deepgreen}{rgb}{0,0.5,0}

\usepackage{amsmath,amssymb}
\usepackage{bm}
\usepackage{amsthm}
\usepackage{anyfontsize}
\usepackage{url}
\usepackage{graphicx}
\usepackage{array}
\usepackage{varwidth}
\usepackage{geometry}
\usepackage{ esint }
\definecolor{ao(english)}{rgb}{0.0, 0.5, 0.0}
\usepackage{listings}
\usepackage{fancyhdr}
\usepackage{enumitem}
\usepackage[
    backend=bibtex,
   bibstyle=ieee,
    citestyle=numeric-comp,
    natbib=true,
    sortlocale=en_GB,
    sorting = nyt,
    url=false,
    doi=true,
    arxiv=true,
    eprint=true
]{biblatex}
\usepackage{caption}
\usepackage{mathtools}
\usepackage{subfig}
\usepackage{appendix}
\usepackage{tikz-cd} 
\usepackage{easybmat}
\usepackage{hyperref}
\usepackage{xcolor}
\hypersetup{
    colorlinks,
    linkcolor={blue!50!black},
    citecolor={blue!50!black},
    urlcolor={blue!80!black}
}

\RequirePackage{algorithm}[0.1]

\RequirePackage[capitalize,nameinlink]{cleveref}[0.19]
\crefformat{equation}{\textup{#2(#1)#3}}
\crefrangeformat{equation}{\textup{#3(#1)#4--#5(#2)#6}}
\crefmultiformat{equation}{\textup{#2(#1)#3}}{ and \textup{#2(#1)#3}}
{, \textup{#2(#1)#3}}{, and \textup{#2(#1)#3}}
\crefrangemultiformat{equation}{\textup{#3(#1)#4--#5(#2)#6}}%
{ and \textup{#3(#1)#4--#5(#2)#6}}{, \textup{#3(#1)#4--#5(#2)#6}}{, and \textup{#3(#1)#4--#5(#2)#6}}

\Crefformat{equation}{#2Equation~\textup{(#1)}#3}
\Crefrangeformat{equation}{Equations~\textup{#3(#1)#4--#5(#2)#6}}
\Crefmultiformat{equation}{Equations~\textup{#2(#1)#3}}{ and \textup{#2(#1)#3}}
{, \textup{#2(#1)#3}}{, and \textup{#2(#1)#3}}
\Crefrangemultiformat{equation}{Equations~\textup{#3(#1)#4--#5(#2)#6}}%
{ and \textup{#3(#1)#4--#5(#2)#6}}{, \textup{#3(#1)#4--#5(#2)#6}}{, and \textup{#3(#1)#4--#5(#2)#6}}

\usepackage{algpseudocode}
\algnewcommand{\Initialize}[1]{
  \State \textbf{Initialize:}
  \Statex \hspace*{\algorithmicindent}\parbox[t]{.8\linewidth}{\raggedright #1}
}
\algnewcommand{\Indent}[2]{
  \State {#1}
  \vspace{-2mm}
  \Statex \hspace*{\algorithmicindent}\parbox[t]{.9\linewidth}{\raggedright #2}
}

\newtheorem{proposition}{Proposition}[section]
\newtheorem{theorem}{Theorem}[section]
\newtheorem{lemma}{Lemma}[section]
\newtheorem{definition}{Definition}[section]
\newtheorem{corollary}{Corollary}[section]
\newtheorem{remark}{Remark}[section]
\numberwithin{equation}{section}

\title{{A sparse spectral method for fractional differential equations in one-spatial dimension}}

\author{Ioannis P.~A.~Papadopoulos\thanks{Department of Mathematics, Imperial College London, UK. \newline \hspace*{5.5mm} \tt{ioannis.papadopoulos13@imperial.ac.uk},} \and Sheehan Olver\thanks{Department of Mathematics, Imperial College London, UK, {\tt s.olver@imperial.ac.uk}.  \newline
\hspace*{4.7mm} \textbf{Funding:} This work was completed with the support of the EPSRC grant EP/T022132/1 ``Spectral element methods for fractional differential equations, with applications in applied analysis and medical imaging" and the Leverhulme Trust Research Project Grant RPG-2019-144 ``Constructive approximation theory on and inside algebraic curves and surfaces". IP was also supported by the Deutsche Forschungsgemeinschaft (DFG, German Research Foundation) under Germany's Excellence Strategy -- The Berlin Mathematics Research Center MATH+ (EXC-2046/1, project ID: 390685689).}\footnotemark[2]}

\begin{document}
\maketitle
\date
\thispagestyle{empty}
\pagestyle{fancy}
\lhead{Ioannis Papadopoulos, Sheehan Olver}

\begin{abstract}
We develop a sparse spectral method for a class of fractional differential equations, posed on $\mathbb{R}$, in one dimension. These equations may include sqrt-Laplacian, Hilbert, derivative, and identity terms. The numerical method utilizes a basis consisting of weighted Chebyshev polynomials of the second kind in conjunction with their Hilbert transforms. The former functions are  supported on $[-1,1]$ whereas the latter have global support. The global approximation space may contain different affine transformations of the basis, mapping $[-1,1]$ to other intervals. Remarkably, not only are the induced linear systems sparse, but the  operator decouples across the different affine transformations. Hence, the solve reduces to solving $K$ independent sparse linear systems of size $\mathcal{O}(n)\times \mathcal{O}(n)$, with $\mathcal{O}(n)$ nonzero entries, where $K$ is the number of different intervals and $n$ is the highest polynomial degree contained in the sum space. This results in an $\mathcal{O}(n)$ complexity solve. Applications to fractional heat and wave equations are considered.
\end{abstract}

\section{Introduction}
\label{sec:introduction}

In this work, we develop a spectral method,  detailed in \cref{alg:spectral-method},  to solve equations of the form
\begin{align}
\mathcal{L}_{\lambda, \mu, \eta}[u] \coloneqq (\lambda \mathcal{I} + \mu \Hil + \eta \fdx+ (-\Delta)^{1/2})[u] = f,
\label{eq:fpde}
\end{align}
in one dimension where $\lambda, \mu, \eta \in \mathbb{R}$ are known constants, $\Hil$ is the Hilbert transform and $(-\Delta)^{1/2}$ is the sqrt-Laplacian, as defined in the next section. The domain is the whole real line $\mathbb{R}$. We seek a solution such that:
\begin{align}
\lim_{|x| \to \infty} u(x) = 0.
\end{align}

Although in this work we only consider the sqrt-Laplacian $(-\Delta)^{1/2}$, the method may generalize to fractional rational powers $(-\Delta)^s$, $s \in (0,1) \cap \mathbb{Q}$. This would be achieved via weighted Jacobi polynomials and their fractional Laplacian analogues \cite[Tab.~5, Row $1^*$]{Gutleb2023} combined with the techniques developed by Hale and Olver \cite{Hale2018} for fractional derivatives of a rational power of Riemann--Liouville and Caputo types. These techniques may further extend to two and three-dimensional problems via the formulae found in \cite[Tab.~7]{Gutleb2023}.

Equations of the form \cref{eq:fpde} arise in several contexts. Hilbert transforms occur in equations modelling interfacial turbulence of surface-tension driven Stokes flow \cite{Thess1995}.   The Hilbert transform is particularly useful in analyzing and processing signals, especially in the context of understanding their instantaneous frequency and phase relationships. By taking the derivative of the phase of the analytic signal with respect to time, one can obtain the instantaneous frequency, providing insight into how the frequency of the signal changes over time. The Hilbert transform has also been used in the Benjamin--Ono equation \cite{Benjamin1967}  as well as in the context of Riemann--Hilbert problems \cite{Trogdon2015}.  The sqrt-Laplacian $(-\Delta)^{1/2}$ is an operator that represents nonlocal or long-range interactions in a physical system. Examples include
\begin{itemize}
\itemsep=0pt
\item anomalous diffusion and random walks: unlike the classical Laplacian, the sqrt-Laplacian can capture more complex diffusion behaviour often characterized by power-law scaling and naturally arises as the continuous limit of discrete long jump random walks, cf.~\cite[Sec.~2.1.4]{Lischke2020}, \cite[Ch.~1]{pozrikidis2018fractional} and \cite{valdinoci2009long}.
\item nonlocal potential energy: the sqrt-Laplacian has appeared in the kinetic energy term of a Hamiltonian in the context of quantum mechanics where it models a nonlocal contribution to the kinetic energy of a particle \cite{laskin2000fractional,laskin2002fractional}.
\end{itemize}

When $\lambda = \eta = \mu =0$ and $\lambda = 1$, $\eta = \mu =0$ we recover the fractional Poisson and fractional (positive-definite) Helmholtz equations, respectively. Such equations have found many uses \cite{Antil2017, Du2019, Hatano1998, Benson2000, Carmichael2015}. In particular the fractional Helmholtz equation arises after a backward Euler time discretization of the fractional heat equation
\begin{align}
\partial_t u(x,t) +  (-\Delta)^{1/2}[u](x,t) = f(x,t).
\end{align} 
In this case, the constant $\lambda = (\Delta t)^{-1}$ is the inverse of the time step size and $\mu= \eta = 0$. It also arises in power law absorption \cite{Treeby2010} or during the Newton linearization for nonlocal Burgers-type equations \cite{Baker1996, Cordoba2006}, which have applications for quasi-geostrophic equations \cite{Chae2005}, or fractional porous medium flow \cite{Caffarelli2016}.

Our approach is to approximate the function $u$ with a space of functions consisting of weighted Chebyshev polynomials of the second kind and their Hilbert transforms, the span of which we refer to as a {\it sum space}: it is a direct sum of the span of two bases. This approach is inspired by a similar method utilized by Hale and Olver for fractional integral and differential equations \cite{Hale2018}. As the Hilbert transform is an anti-involution, $\Hil[\Hil[u]] = -u$, the operator $(\lambda \mathcal{I} + \mathcal{H})$ maps the sum space to itself and the operator representation of $(\lambda \mathcal{I} + \mathcal{H})$ is sparse. This method naturally extends to combinations of affine transformations of the sum space where the Hilbert transform decouples across the affine transformations. By constructing a dual sum space that contains the range of the derivative of the sum space, we exploit a relationship between the sqrt-Laplacian, derivative and Hilbert transform (see \cref{eq:hilbert-fL}) to solve \cref{eq:fpde}.

There exist many methods for tackling problems involving fractional Laplacian terms. Finite element methods (FEM) are popular, for instance see \cite{Bonito2018, Lischke2020} and the references therein. Typically during the construction of FEM, the boundary must be bounded. This implies that special care must be taken for the imposed boundary conditions as the different definitions of the fractional Laplacian are no longer equivalent. Moreover, the solutions themselves are often heavy-tailed, i.e.~they decay at the rate of a Cauchy distribution rather than a Gaussian (see \cref{sec:frac-heat-example}) \cite[Sec.~2]{Vazquez2018}. Hence, FEM na\"ively applied to a (large) truncation of an unbounded domain will likely exhibit numerical artefacts due to the nonlocal nature of the operators \cite[Sec.~9]{D2020}.

Spectral methods for tackling Helmholtz-like problems posed on $\mathbb{R}^d$ can also be found in the literature \cite{Chen2018, Tang2018, Tang2020, Mao2017, Sheng2020, Li2021}. For a summary see \cite{D2020}. Mao and Shen \cite{Mao2017} developed a Hermite spectral method to solve the fractional Helmholtz problem ($\lambda >0$, $\eta=\mu=0$) on unbounded domains, $d \in \{1,2\}$, for an arbitrary fractional Laplacian $(-\Delta)^s$, $s \in (0,1)$.  Each solve relied on four discrete Hermite transforms. They also developed a Galerkin method where the stiffness matrix entries can be computed by explicit integrals, however, the stiffness matrix is dense. Li et al.~\cite{Li2021} also develop a spectral-Galerkin method based on Hermite polynomials and their stiffness matrix is also dense. They are also able to define a truncated fractional Laplacian. Tang et al.~\cite{Tang2018} developed a Hermite collocation method where they derived explicit recurrence relationships for the entries in the stiffness matrix. They applied their method for $d \in \{1,2\}$ as well as to nonlinear problems. Similarly, as they use a collocation method, their linear systems are dense. Then, Tang et al.~\cite{Tang2020} extended their collocation method to modified mapped Gegenbauer functions which are better suited to approximating solutions with algebraic decay rates. Sheng et al.~\cite{Sheng2020} note that for dimensions two or more, the previous methods can become expensive. Sheng et al.~developed a Chebyshev spectral Galerkin method for the fractional positive-definite Helmholtz equation on an unbounded domain using the Dunford--Taylor formulation of the fractional Laplacian \cite[Th.~4.1]{Bonito2019}. They construct so-called mapped Chebyshev functions defined on $\mathbb{R}$. Their method works for arbitrary $s \in (0,1)$ and they considered one, two, and three-dimensional problems. The initialization of their algorithm requires the solution of an eigenvalue problem with worst case complexity $\mathcal{O}(n^3)$ in one dimension. However, they remark this can be improved to $\mathcal{O}(n^2)$ with an optimal solver. Each solve requires an FFT and the complexity of the solve is $\mathcal{O}(n (\log n)^d)$, where $d$ is the dimension.  Recently, in one dimension, a series of specialized algorithms were developed by de la Hoz and coworkers that utilize a so-called $L\cot(s)$ transformation to convert the unbounded domain to a bounded one \cite{cayama2021pseudospectral, cayama2020numerical,cayama2022fast,cuesta2023numerical}. They converge at a spectral rate when there is sufficient smoothness in the problem, accurately capture behaviour of the solution as $|x|\to \infty$, and may be applied to nonlinear problems. The main disadvantages are the heavy reliance on the choice of $L$ for accuracy (although heuristics exist), the evaluation of infinite sums or expensive evaluations of $_{2}F_1$ functions (which may require higher precision than double precision), and the ill-conditioning of the matrices that represent the action of the fractional Laplacian which introduces numerical error when the truncation degree is large. Out of the four papers mentioned, \cite{cuesta2023numerical} is a specialized algorithm for $s=1/2$ and accelerates the computation of the sqrt-Laplacian via an FFT resulting in a fast algorithm. However, depending on the behaviour of the solution, the algorithm requires different choices for the setup (for instance the specific choice of the transformation to the bounded domain) which induces some subtleties in the implementation. We emphasize their algorithms are effective, however, we believe our spectral method is better suited for certain classes of nonsmooth right-hand sides (see \cref{sec:examples:nonsmooth}), solutions with singularities that naturally arise in the case of the sqrt-Laplacian, and is better conditioned. We prove that the linear system to compute the solution may be preconditioned with a diagonal preconditioner such that the condition number is $\mathcal{O}(1)$, independent of the truncation degree and the number of discretization intervals. Moreover, the expansion of the right-hand side relies on the use of frames for which an $\epsilon$-truncated SVD factorization has been proven to mitigate the perceived ill-conditioning (see \cref{sec:implementation}). 

We reiterate some advantages of our method. We can handle the terms $\fdx$ and $\Hil$ as well as $(-\Delta)^{1/2}$ and $\mathcal{I}$, the former of which are not included in the aforementioned works. Moreover, the ability to contain multiple affine transformations of the reference sum space allows one to better handle regions of discontinuity in the data resulting in better approximations. The resulting linear systems are banded and sparse. Each solve is divided into independent  well-conditioned  solves, that can be computed in parallel, involving only the functions contained in the same affine transformed sum space. This results in quick solve times.  Moreover, there exists a diagonal preconditioner such that the preconditioned linear system has an $\mathcal{O}(1)$ condition number, independent of the truncation degree or interval choice.   The  evaluation of the solution  requires the integral of four independent special functions per differently scaled sum space, and if all sum spaces have the same scaling (i.e.~they are just translations), we require only four integrals in total. The special functions are not dependent on the right-hand side of \cref{eq:fpde}, however, they are dependent on $\lambda$, $\mu$, and $\eta$. 

Given that the different affine transformed sum spaces have overlapping support, the fact that the operator decouples across the sum spaces seems implausible at first glance. There must be communication between the sum spaces for continuity purposes. The communication occurs during the expansion of the right-hand side: a process that cannot be decoupled across different affine transformed sum spaces. Hence, the computational expense shifts from the solve of the equation to an interpolation of a known function \cite{Adcock2019, Adcock2020}.

\begin{remark}
As the sum spaces centred at multiple intervals are combined together to form the overall approximation space, it is tempting to call this a spectral \emph{element} method. We avoid using the terminology ``element" as it may be confusing. The support of a sum space function centred at an interval is not necessarily contained within the interval. In fact the support of the approximation space is always $\mathbb{R}$. Moreover, the expansion of the right-hand side is slightly more difficult than a typical finite/spectral element method as the sum space functions centred at different intervals all interact. However, we emphasize that the operator $L_{\lambda,\mu,\eta}$ decouples across sum space functions centred at different intervals. Hence, given an expansion of the right-hand side, the solve complexity only increases linearly with the number of intervals used and the solve over different intervals can be done in parallel.
\end{remark}

The paper is structured as follows. In \cref{sec:math-setup} we rigorously define the problem we are solving and prove relationships between the sqrt-Laplacian, the Hilbert transform, and the derivative operator. In \cref{sec:sum-space}, we introduce the approximation space of our spectral method which we call the ``sum space". We prove the regularity of the functions contained in the sum space and derive the action of the fractional operator $\mathcal{L}_{\mu,\lambda,\eta}$ on the functions. In \cref{sec:spectral-method}, we outline the algorithm in \cref{alg:spectral-method} and discuss implementation details in \cref{sec:implementation}. We analyze the conditioning of the solve in \cref{sec:conditioning} and prove the existence of a diagonal preconditioner that induces a condition number independent of the truncation degree and number of intervals. In \cref{sec:analysis} we prove a convergence result for the spectral method and consider various examples in \cref{sec:examples} which include the fractional heat equation and a fractional wave propagation problem. In \cref{sec:conclusions} we give our conclusions. In the appendices, we explain how to utilize the FFT to approximate the continuous inverse Fourier transforms of the four special functions, required for the evaluation of the solution, and also discuss special cases for the parameters $\lambda$, $\mu$, and $\eta$.

\section{Mathematical setup}
\label{sec:math-setup}
Let $W^{s,p}(\mathbb{R})$ denote the (possibly fractional) Sobolev space \cite{Adams2003,Di2012} and $H^{s}(\mathbb{R}) \coloneqq W^{s,2}(\mathbb{R})$. We denote the Lebesgue space by $L^s(\mathbb{R})$, $s >0$. We seek a solution $u \in H^{1/2}(\mathbb{R})$ for \cref{eq:fpde}. Moreover, if $\eta \neq 0$, then we require the stronger assumption that $u \in H^1(\mathbb{R})$. We note that $H^1(\mathbb{R}) \subset H^{1/2}(\mathbb{R})$ which follows from the (Fourier) definition of $H^1(\mathbb{R})$ \cite[Sec.~7.62]{Adams2003} and a small extension to \cite[Prop.~3.6]{Di2012}. Similarly, the right-hand side $f$ must have sufficient regularity so that \cref{eq:fpde} is well-posed. Let $H^{-s}(\mathbb{R})$ denote the dual of $H^s(\mathbb{R})$, then we require $f \in H^{-1/2}(\mathbb{R})$.

$\mathcal{I}$ denotes the identity operator and $(-\Delta)^{1/2}$ denotes the fractional sqrt-Laplacian. Its definition is somewhat subtle depending on the domain and range considered. Let $\mathcal{S}$ denote the space of Schwartz functions. For any $s \in (0,1)$, we define $(-\Delta)^s : \mathcal{S} \to L^2(\mathbb{R})$ as \cite[Sec.~3]{Di2012}:
\begin{align}
(-\Delta)^{s} u(x) &\coloneqq c_{s} \,\,  \dashint_{\mathbb{R}} \frac{u(x) - u(y)}{|x-y|^{1+2s}} \, \dy, \; \text{for a.e.} \; x \in \mathbb{R}, \;\;
c_s  \coloneqq \frac{4^s \Gamma(1/2+s)}{\pi^{1/2} |\Gamma(-s)|}. \label{def:fracLap}
\end{align} 
Here $\dashint_{\mathbb{R}} \cdot$ denotes the Cauchy principal value integral \cite[Ch.~2.4]{King2009a} and $\Gamma(\cdot)$ denotes the Gamma function. For any $u \in H^s(\mathbb{R})$ it can be shown that $(-\Delta)^{s/2}u \in L^2(\mathbb{R})$ \cite[Prop.~3.6]{Di2012} where $(-\Delta)^{s/2}$ is defined as in \cref{def:fracLap} for $s \in (0,1)$. This also holds for $s=1$ by a small extension. Thus we consider the weak form reformulation of the fractional Laplacian $(-\Delta)^s: H^s(\mathbb{R}) \to H^{-s}(\mathbb{R})$:
\begin{align}
\langle (-\Delta)^s u, v \rangle_{H^{-s}(\mathbb{R}),H^s(\mathbb{R})} &= \langle (-\Delta)^{s/2} u,(-\Delta)^{s/2}  v \rangle_{L^2(\mathbb{R})} \;\; \text{for all} \; v \in H^s(\mathbb{R}). \label{eq:weak}
\end{align}
The right-hand side of \cref{eq:weak} is often referred to as the \emph{quadratic} form of the fractional Laplacian. The quadratic form is best suited to prove existence of solutions. Suppose $\lambda > 0$ and $\eta \in \mathbb{R}$. Then the following equation has a weak solution $u_* \in H^{1/2}(\mathbb{R})$ (see \cref{th:existence}), 
\begin{align}
\begin{split}
&\lambda \langle u, v \rangle_{L^2(\mathbb{R})} + \eta \langle \Hil u, v \rangle_{L^2(\mathbb{R})}\\
& \indent + \langle (-\Delta)^{1/4} u,(-\Delta)^{1/4}  v \rangle_{L^2(\mathbb{R})} = \langle f, v \rangle _{H^{-1/2}(\mathbb{R}),H^{1/2}(\mathbb{R})}. 
\end{split} \label{eq:weak-form}
\end{align}
By definition of $H^{1/2}(\mathbb{R})$ we have that $u_* \in L^2(\mathbb{R})$. Moreover, if $f \in L^2(\mathbb{R})$ (and thus $(-\Delta)^{1/2} u_* \in L^2(\mathbb{R})$) then the weak solution $u_*$ also satisfies \cite[Th.~1.1]{Kwasnicki2017}
\begin{align}
(\lambda \mathcal{I} + \eta \Hil + (-\Delta)^{1/2}) u_* = f \; \text{a.e.~in} \; \mathbb{R}. \label{eq:strong}
\end{align}

The fractional Laplacian can also be defined via the Fourier transform. In this work we use the following conventions for the Fourier transform and its inverse, $\mathcal{F}:L^1(\mathbb{R}) \to C(\mathbb{R})$:
\begin{align}
\mathcal{F}[f](\omega) = \int_{\mathbb{R}} f(x) \me^{-\mi \omega x} \dx, \;\;\; 
\mathcal{F}^{-1}[F](x) = \frac{1}{2\pi} \int_{\mathbb{R}} F(\omega) \me^{\mi \omega x} \mathrm{d}\omega. \label{def:fourier}
\end{align}
The Fourier transform is an automorphism for the space $\mathcal{S}$, $\mathcal{F} : \mathcal{S} \to \mathcal{S}$ \cite[Cor.~9.1.8]{Bogachev2020}.  By duality, the Fourier transform (and its inverse) may be defined for the space of tempered distributions, $\mathcal{F}:\mathcal{S}^* \to \mathcal{S}^*$, where the space of tempered solutions $\mathcal{S}^*$ is the dual space of $\mathcal{S}$ \cite[Sec.~9.3]{Bogachev2020}. As the space of Schwartz functions is dense in $L^p(\mathbb{R})$, for $p \in [1,\infty)$, the Fourier transform may be extended to all $f \in L^p(\mathbb{R})$. However, the Fourier transform may only be a tempered distribution for $p > 2$. The Hausdorff--Young inequality implies that, for $p \in [1,2]$, $\mathcal{F} : L^p(\mathbb{R}) \to L^q(\mathbb{R})$ where $1/p+1/q=1$ \cite[Ch.~2.3]{Rauch1991}. In general, the Fourier transform of a function $f \in L^p(\mathbb{R}) \backslash L^1(\mathbb{R})$, $p \in (1,\infty)$, cannot be computed via the formula in \cref{def:fourier} \cite[Sec.~9.2]{Bogachev2020}.

If both $u, (-\Delta)^s u \in L^p(\mathbb{R})$ for some $p \in [1,\infty)$, then the operator $(-\Delta)^s$ may be equivalently interpreted as a Fourier multiplier, i.e.~the following equation holds in the sense of distributions \cite[Th.~1.1]{Kwasnicki2017}:
\begin{align}
\mathcal{F}[(-\Delta)^s [u]](\omega) = |\omega|^{2s} \mathcal{F}[u](\omega).
\label{eq:fourier-frac-laplace}
\end{align}
If $p \in [1,2]$, then \cref{eq:fourier-frac-laplace} holds a.e.~in $\mathbb{R}$ \cite[Th.~1.1]{Kwasnicki2017}.

$\mathcal{H}: L^p(\mathbb{R}) \to L^p(\mathbb{R})$, $p \in (1,\infty)$, denotes the Hilbert transform which is a bounded linear operator. It can be defined via the following convolution:
\begin{align}
\mathcal{H}[u](x) \coloneqq \frac{1}{\pi} \dashint_{\mathbb{R}} \frac{u(y)}{x-y}  \mathrm{d}y \;\; \text{for a.e.} \; x \in \mathbb{R}. 
\end{align}
Equivalently, the Hilbert transform can be seen as a Fourier multiplier such that \cite[Sec.~2.2]{Johansson1999}
\begin{align}
\mathcal{F}[\mathcal{H}[u]](\omega) = -\mi \, \mathrm{sgn}(\omega) \mathcal{F}[u](\omega),
\label{eq:fourier-hilbert}
\end{align}
where $\mathrm{sgn}(\cdot)$ denotes the sign function. A useful property of the Hilbert transform is that it is anti-self adjoint \cite[Th.~102]{Titchmarsh1948}.
\begin{lemma}[Anti-self adjointness of the Hilbert transform]
\label{lem:existence2}
The Hilbert transform is anti-self adjoint, i.e.~for $u \in L^p(\mathbb{R})$, $v \in L^q(\mathbb{R})$ such that $1 <p, q < \infty$, $1/p+1/q = 1$,
\begin{align}
\langle \mathcal{H}[u], v \rangle_{L^p(\mathbb{R}), L^q(\mathbb{R})} = -  \langle u, \mathcal{H}[v] \rangle_{L^p(\mathbb{R}), L^q(\mathbb{R})}.
\end{align}
\end{lemma}
The Hilbert transform has found uses in aerofoil theory, crack formation, elasticity, and potential theory \cite{King2009a}. Of particular interest here is the connection with the sqrt-Laplacian.

\begin{theorem}
\label{th:fdx-Hil}
Consider any $u \in L^2(\mathbb{R})$ such that $(-\Delta)^{1/2} u \in \mathcal{S}^*$. Then, the following equation holds in the sense of distributions:
\begin{align}
\fdx \mathcal{H}[u] = (-\Delta)^{1/2}[u]. \label{eq:hilbert-fL}
\end{align}
\end{theorem}
\begin{proof}
For any $u \in L^2(\mathbb{R})$, $(-\Delta)^{1/2} u$ is defined as follows: for any $\phi \in \mathcal{S}$ \cite{Kwasnicki2017},
\begin{align}
\int_{\mathbb{R}} (-\Delta)^{1/2}[u] \phi \, \dx = \int_{\mathbb{R}} u (-\Delta)^{1/2}[\phi] \, \dx,
\end{align}
where $(-\Delta)^{1/2}[\phi]$ is understood via \cref{def:fracLap}. Similarly,
\begin{align}
\int_{\mathbb{R}} \fdx \Hil[u] \phi \, \dx = - \int_{\mathbb{R}} \Hil[u] \fdx \phi \, \dx = \int_{\mathbb{R}} u \Hil \fdx \phi \, \dx,
\end{align}
where the first equality is the definition of the distributional derivative and the second holds by \cref{lem:existence2} since $u, \fdx \phi \in L^2(\mathbb{R})$. It remains to show that $(-\Delta)^{1/2} \phi = \Hil \fdx \phi$ for any $\phi \in \mathcal{S}$. This follows by comparing Fourier coefficients. 

Since $\phi \in \mathcal{S}$ then $\phi \in L^2(\mathbb{R})$ and by the singular operator definition of $(-\Delta)^{1/2}$, $(-\Delta)^{1/2} \phi \in L^2(\mathbb{R}$). Hence, by a result of Kwasnicki \cite[Th.~1.1]{Kwasnicki2017}, $\mathcal{F}[(-\Delta)^{1/2}\phi](\omega) = |\omega| \mathcal{F}[\phi](\omega)$ for a.e.~$\omega \in \mathbb{R}$. Moreover, since $\phi \in \mathcal{S}$, then $\fdx \phi \in L^2(\mathbb{R})$ which implies that $\Hil[\fdx \phi] \in L^2(\mathbb{R})$ and, therefore, $\mathcal{F}[\Hil [\fdx  \phi]] \in L^2(\mathbb{R})$. In particular,
\begin{align}
\mathcal{F}[\Hil [\fdx  \phi]](\omega) = - \mi \, \sgn(\omega) \mathcal{F}[\fdx \phi](\omega) = |\omega| \mathcal{F}[\phi](\omega) \; \text{for a.e.} \; \omega \in \mathbb{R}. 
\end{align}
\end{proof}

For $f \in L^2(\mathbb{R})$, the same $u_*$ that solves \cref{eq:weak-form} also satisfies
\begin{align}
u_*(x) = \mathcal{F}^{-1} [(\lambda - \mi \, \eta \, \mathrm{sgn}(\omega) +  |\omega|)^{-1} \mathcal{F} [f](\omega)](x), \label{eq:strong-fourier}
\end{align}
whenever the right-hand side of \cref{eq:strong-fourier} is well-defined (which is not automatic). For the case where $\lambda < 0$, $\eta = 0$, there may exist \emph{standing wave} solutions to the homogeneous equation ($f = 0$). These standing waves do not decay and thus do not live in $H^{1/2}(\mathbb{R})$. However, the Fourier transform definition does not automatically select the solution that lives in $H^{1/2}(\mathbb{R})$. Thus a solution that satisfies \cref{eq:strong-fourier} might not satisfy \cref{eq:weak}. In \cref{sec:special-cases} we encounter an example where $f \in L^1(\mathbb{R}) \backslash L^2(\mathbb{R})$ (and thus $\mathcal{F}[f]$ is well-defined) such that \cref{eq:strong-fourier} has a solution and the solution does not live in $H^{1/2}(\mathbb{R})$, and thus cannot be a solution of \cref{eq:weak}. 

\begin{theorem}[Existence]
\label{th:existence}
Suppose that $\mu \in \mathbb{R}$, $\eta = 0$, $\lambda > 0$, and $f \in H^{-1/2}(\mathbb{R})$. Then, there exists a unique weak solution $u_* \in H^{1/2}(\mathbb{R})$ that satisfies \cref{eq:weak-form}.
\end{theorem}
\begin{proof}
The result follows by utilizing the anti-self adjointness and continuity of $\Hil$ (see \cref{lem:existence2} and \cite{Pichorides1972}) as well as an application of the Lax--Milgram theorem (see \cite[Sec.~2.2]{Mao2017}). 
\end{proof}

\section{Sum space}
\label{sec:sum-space}
The goal of this section is to define the four families of functions found in \cref{fig:relations}. The first row in \cref{fig:relations} form the primal set of functions that we approximate our solution with. The action of $(-\Delta)^{1/2}$ maps the top left and right families to the bottom left and right families of functions, respectively. The map induces a diagonal matrix. By contrast, we construct the identity map from the top left family to the bottom right, and from the top right to the bottom left. These identity mappings only have two diagonals in their induced matrices that are non-zero entries. Throughout this work, we use the convention $\mathbb{N} \coloneqq \{1,2,3,\dots\}$ and $\mathbb{N}_0 \coloneqq \{0,1,2,3,\dots\}$.

\begin{remark} 
We note that there does not exist an identity operator from top left to bottom left and top right to bottom right. This is precisely why we require both families in the top and bottom rows, rather than just one. By combining the first row of families of functions and the bottom row of families of functions, we can find the sparse map induced by $\lambda \mathcal{I}+(-\Delta)^{1/2}$. The operators $\Hil$ and $\fdx$ are treated similarly. 
\end{remark}

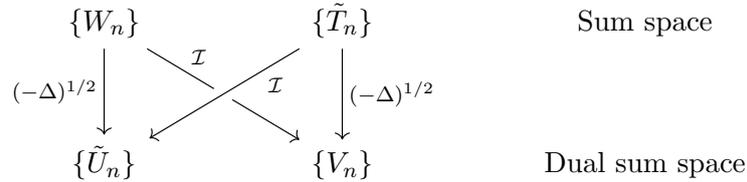
\begin{figure}[h]
\centering
\begin{tikzcd}[column sep=5em, row sep=3em]
\{W_n\} \arrow[swap]{d}{(-\Delta)^{1/2}} \arrow[near start]{dr}{\mathcal{I}} &\{\eT_n\}  \arrow{d}{(-\Delta)^{1/2}}   \arrow[crossing over,near start]{dl}{\mathcal{I}} & \text{Sum space} \\%
\{\eU_{n} \}& \{V_n\}& \text{Dual sum space} 
\end{tikzcd} 
\caption{The domain and range of the operators $(-\Delta)^{1/2}$ and $\mathcal{I}$ for our approximating spaces.}
\label{fig:relations}
\end{figure}

Consider the {\it sum space}:
\begin{align}
S \coloneqq \mathrm{span}( \{\eT_n,  W_n : n \in \mathbb{N}_0 \}),
\end{align}
where $W_n(x) = (1-x^2)^{1/2}_+ U_n(x)$, $U_n$ are Chebyshev polynomials of the second kind, and $(1-x^2)^{1/2}_+ = 0$ if $|x|\geq 1$.  Recall that \cite[Sec.~18.3]{OlverNIST}, $U_n(x)$ are orthogonal to the weight $(1-x^2)^{1/2}_+$, in the domain $(-1,1)$, and
\begin{align}
U_0(x) = 1, \;\; U_1(x) = 2x,\;\; U_{n+1}(x) = 2xU_n(x) - U_{n-1}(x), \; n \in \mathbb{N}.
\end{align} 
Hence, $W_n(x)$ are \emph{weighted} Chebyshev polynomials of the second kind extended to $\mathbb{R}$ by zero. Moreover,
\begin{align}
\eT_n(x) \coloneqq
\begin{cases}
T_n(x) & \text{for} \; |x|\leq 1,\\
(x-\mathrm{sgn}(x)\sqrt{x^2 - 1})^n & \text{for} \; |x| > 1,
\end{cases}
\end{align}
where $T_n(x)$ are Chebyshev polynomials of the first kind.  Recall that \cite[Sec.~18.3]{OlverNIST} $T_n(x)$ are orthogonal to the weight $(1-x^2)^{-1/2}_+$ in $(-1,1)$ and
\begin{align}
T_0(x) = 1, \;\; T_1(x) = x,\;\; T_{n+1}(x) = 2xT_n(x) - T_{n-1}(x), \; n \in \mathbb{N}.
\end{align} 
We call $\eT_n(x)$ \emph{extended} Chebyshev functions of the first kind. Under the change of variables, $x = \cos \theta$, for $|x|\leq 1$, the sum space is equivalent to the Fourier series on the half-period $0 \leq \theta \leq \pi$. Thus expanding a function $f(x)$ in this space is equivalent to the Fourier extension problem \cite{Boyd2002,Bruno2007,Huybrechs2010}. We emphasize that we must take the span of both families of functions otherwise we are unable to construct a sparse representation of the fractional operator in \cref{eq:fpde}.

The extended Chebyshev functions $\eT_n$ satisfy the same recurrence relationship as $T_n$.
\begin{proposition}[Recurrence relation]
\label{prop:recurrence}
For all $x \in \mathbb{R}$,  $n \in \mathbb{N}$, 
\begin{align}
2x \eT_n(x) = \eT_{n+1}(x) +  \eT_{n-1}(x). \label{eq:recurrence1}
\end{align}
\end{proposition}
\begin{proof}
The case where $|x| \leq 1$ is the three-term recurrence for Chebyshev polynomials of the first kind. For $|x| > 1$, the result follows by utilizing the substitution $x = \cosh(y)$ if $x>1$ and $x=-\cosh(y)$ if $x<-1$ into \cref{eq:recurrence1}. 
\end{proof}

The following result is a key observation for the construction of the spectral method. The result can be found in \cite[Cor.~5.7]{Trogdon2015}.
\begin{proposition}[Hilbert transform]
\label{prop:Hilbert-W}
For $x \in \mathbb{R}$, $n \in \mathbb{N}_0$, we have that
\begin{align}
\Hil[W_n](x) = \eT_{n+1}(x).
\label{eq:HwU}
\end{align}
Moreover, for a function $u \in L^p(\mathbb{R})$, $p \in (1,\infty)$, the Hilbert transform is an anti-involution, i.e.~$\Hil^2[u] = -u$. Hence,
\begin{align}
\Hil[\eT_{n+1}](x) = -W_n(x).
\end{align}
\end{proposition}
\begin{remark}
Here $\eT_0(x) = 1$ for all $x \in \mathbb{R}$. It can be shown that $\Hil[\eT_0](x) = 0$. Therefore, $\Hil^2[\eT_0](x) = 0$. This is not a contradiction as $\eT_0 \not \in L^p(\mathbb{R})$ for any $p \in[1,\infty)$. 
\end{remark}

\begin{lemma}
\label{lem:sum-space-regularity}
$W_n \in L^s(\mathbb{R}) \cap W^{1,p}(\mathbb{R})$ for $n \geq 0$, $s \in [1,\infty]$, $p \in [1,2)$, and $\eT_n \in L^s(\mathbb{R})  \cap W^{1,p}(\mathbb{R})$ for $n \geq 1$, $s \in (1,\infty)$, $p \in (1,2)$. 
\end{lemma}
\begin{proof}
Since $W_n(x)$ are essentially bounded and supported on $[-1,1]$, it follows immediately that $W_n \in L^p(\mathbb{R})$ for all $n \geq 0$, $p \in [1,\infty]$. A direct check reveals that $W_n \in W^{1,p}(\mathbb{R})$ for $p \in [1,2)$. Moreover, for any function $f \in L^s(\mathbb{R})$, $s \in (1,\infty)$, we have that \cite{Pichorides1972}
\begin{align}
\| \Hil f \|_{L^s(\mathbb{R})} \leq C \| f \|_{L^s(\mathbb{R})},
\label{eq:frame2}
\end{align}
for a constant $C$ that only depends on $s$. Hence, for $ n \geq 1$,
\begin{align}
\| \eT_n \|_{L^s(\mathbb{R})} = \| \Hil W_{n-1} \|_{L^s(\mathbb{R})} \leq C \| W_{n-1} \|_{L^s(\mathbb{R})} < \infty.
\end{align}
As $\fdx$ and $\Hil$ commute for any function in $W^{1,q}(\mathbb{R})$, $q \in (1,\infty)$ \cite[Th.~3.2]{Johansson1999}, we have that $\Hil \fdx \eT_n(x) = \fdx W_{n-1}(x)$ a.e.~Thus, for any $p \in (1,2)$, $n \geq 1$,
\begin{align}
\| \fdx \eT_n \|_{L^p(\mathbb{R})} = \| \Hil \fdx W_{n-1} \|_{L^p(\mathbb{R})} \leq C \| \fdx W_{n-1} \|_{L^p(\mathbb{R})} < \infty.
\end{align}
Therefore, $\eT_n \in L^s(\mathbb{R})  \cap W^{1,p}(\mathbb{R})$ for $n \geq 1$, $s \in (1,\infty)$, $p \in (1,2)$.
\end{proof}

Consider the expansion, for $x \in \mathbb{R}$,
\begin{align}
u(x) &= \tilde{u}_0 \eT_0(x) +  \sum^\infty_{n=1}[ \tilde{u}_n \eT_n(x) + u_n W_{n-1}(x)], \label{eq:fsum-space}
\end{align}
where $\{\tilde{u}_n\}, \{u_n\} \subset \mathbb{R}$ are constants. Then, a calculation shows that, for $x \in \mathbb{R}$,
\begin{align}
(\lambda \mathcal{I} + \mu \Hil)[u](x) &=  \lambda \tilde{u}_0 \eT_0(x) + \sum^\infty_{n=1}[ ( \lambda \tilde{u}_n+ \mu u_n) \eT_{n}(x) + (\lambda u_n- \mu \tilde{u}_n) W_{n-1}(x) ], \label{eq:Hsum-space}
\end{align}
i.e.~the $(\lambda \mathcal{I} + \mu \Hil)$ operator maps the expansion to itself. Consider the equation $(\lambda \mathcal{I} + \mu \mathcal{H})[u] = f$. If one has an expansion for the right-hand side term, $f$, in the sum space, then it is possible to quickly compute the corresponding solution, $u$, via \cref{eq:Hsum-space}.

\begin{remark}
\label{rem:affine-transform}
We note that the Hilbert transform relationship \cref{eq:HwU} also holds for scaled and shifted Chebyshev polynomials. This provides the bedrock for the decoupling of the Hilbert transform problem across affine transformed sum spaces.
\end{remark}

We note that the sum space is a so-called \emph{frame} for $x \in [-1,1]$ \cite{Adcock2019}. 

\begin{definition}
Consider a Hilbert space $H$. An indexed family of functions $\{\phi_n\} \in H$ is called a \emph{frame} for $H$ if there exist constants $0 < c \leq C < \infty$ such that
\begin{align}
c\| f \|^2_H \leq \sum_n | \langle f, \phi_n \rangle_H |^2 \leq C \|f \|^2_H \;\; \text{for all} \; \; f \in H.
\label{ineq:frame-condition}
\end{align}
\end{definition}

\begin{definition}[Weighted Lebesgue space]
Let $w : [a,b] \to [c,\infty)$, $c > 0$, be a measurable function. Then, the weighted Lebesgue space $L^p_w(a,b)$ is equipped with the norm $\|f\|_{L^p_w(a,b)} \coloneqq (\int_{a}^b |f|^p w \, \dx)^{1/p}$. If $p=2$, then $L^2_w(a,b)$ is a Hilbert space equipped with the inner product $( f,g)_{L^2_w(a,b)} \coloneqq \int_{a}^b f g w \, \dx$.
\end{definition}

The following proposition is equivalent to a similar result by Huybrechs \cite[Cor.~3.2]{Huybrechs2010} concerning a frame for Fourier extensions, as our space restricted to $[-1,1]$ is equivalent to the Fourier basis functions restricted to $[-1,1]$.
\begin{proposition}[Frame]
\label{prop:frame}
Let $w(x) \coloneqq (1-x^2)_+^{-1/2}$. Then, the sum space $S$, for $x \in (-1,1)$, is a frame on the  weighted Lebesgue space $L^2_w(-1,1)$. 
\end{proposition}
\begin{proof}
The families of Chebyshev polynomials $T_n(x)$ and $U_n(x)$ are orthogonal with respect to the weights $(1-x^2)^{-1/2}_+$ and $(1-x^2)^{1/2}_+$, respectively on $L^2(-1,1)$. Consider any $f \in L^2_w(-1,1)$. By noting $\eT_n(x) = T_n(x)$ for $x \in [-1,1]$ and two applications of Parseval's identity, we see that 
\begin{align}
\begin{split}
&\sum_{n=0}^\infty \left[| \langle f, \eT_n \rangle_{L^2_w(-1,1)}|^2 + | \langle f, W_n \rangle_{L^2_w(-1,1)} |^2 \right]\\
& \indent = \sum_{n=0}^\infty \left[| \langle f, w T_n \rangle_{L^2(-1,1)}|^2 + | \langle f w, W_n \rangle_{L^2(-1,1)} |^2 \right]\\
& \indent \indent = \| f  \|^2_{L^2_w(-1,1)} + \| f w \|^2_{L^2_{1/w}(-1,1)}
= 2\| f \|^2_{L^2_w(-1,1)} .
\end{split}\label{eq:frame1}
\end{align}
Thus the lower and upper bound constants in \cref{ineq:frame-condition} are $c = C=2$.
\end{proof}

\begin{remark}
Suppose that the support of a function is contained in the union of user-chosen intervals. Then the function can be represented to arbitrary precision by the frame induced by taking the union of all the sum space functions centred on all the affine transformed sum spaces (see \cref{sec:multiple-intervals} and \cref{prop:mul-frame}).
\end{remark}

\subsection{Dual sum space}
In order to exploit the relationship \cref{eq:hilbert-fL}, we require the action of the derivative on the sum space. Let $V_n(x) \coloneqq (1-x^2)^{-1/2}_+T_n(x)$, $x \in \mathbb{R}$, denote the weighted Chebyshev polynomial of the first kind of order $n$ extended to $\mathbb{R}$ by zero. Define, $(1-x^2)^{-1/2}_+ = 0$ if $|x| \geq 1$. Consider the following functions:
\begin{align}
\eU_{-2}(x) & \coloneqq 
\begin{cases}
0 & |x| \leq 1,\\
-\frac{|x|}{\sqrt{x^2-1}} & |x| > 1,
\end{cases} 
\quad \quad
\eU_{-1}(x)  \coloneqq
\begin{cases}
0 & |x| \leq 1,\\
-\frac{\mathrm{sgn}(x)}{\sqrt{x^2-1}} & |x| > 1.
\end{cases}
\end{align}
Furthermore, for $x \in \mathbb{R}$, we define 
\begin{align}
\eU_0(x) \coloneqq \eT_0(x) + \eU_{-2}(x), \label{def:eU0}
\end{align} 
and, for $n \in \mathbb{N}$, we recursively define the extended Chebyshev functions of the second kind as follows: 
\begin{align}
\eU_n(x) \coloneqq 2\eT_n(x) + \eU_{n-2}(x). \label{def:eUk}
\end{align}
The dual sum space is given by
\begin{align}
S^* \coloneqq \mathrm{span}(\{ \eU_m, V_n: m \in \mathbb{N}_0 \cup \{-2,-1\}, n \in \mathbb{N}_0\}). 
\end{align}

\begin{remark}
Note that the functions in the dual sum space $S^*$ do not live in $H^{1/2}(\mathbb{R})$. However, it is not the solution that is expanded in the dual sum space but rather the right-hand side which is only required to live in $H^{-1/2}(\mathbb{R})$. 
\end{remark}

The weighted Chebyshev polynomials and extended Chebyshev functions have been centred on the interval $[-1,1]$. 
\begin{definition}[Affine transformation]
Consider an interval $I = [a,b] \subset \mathbb{R}$, $a<b$ and the affine transformation $y=2/(b-a) (x-(a+b)/2)$. We define the affine transformed extended and weighted Chebyshev functions and polynomials, centred at $I$, as $\eT_n^{I}(x) = \eT_n(y)$, $\eU_n^{I}(x) = \eU_n(y)$, $W^{I}_n(x) = W_n(y)$, and $V^{I}_n(x) = V_n(y)$. 
\end{definition}
We note that $W^{I}_n(x)$ and $V^{I}_n(x)$ are supported on the interval $I$, whereas $\eT_n^{I}(x)$\ and $\eU_n^{I}(x)$ have global support extending beyond $I$. 

In the following proposition we show some relationships between the sum space and the dual sum space.
\begin{proposition}[Equality and derivatives]
\label{prop:maps}
Consider the interval $I = [a,b] \subset \mathbb{R}$, $a<b$. Then, for any $n \in \mathbb{N}$, $x \in \mathbb{R} \backslash \{a,b\}$,
\begin{align}
\fdx \eT^I_n(x) = \frac{2n}{b-a} \eU^I_{n-1}(x),
\label{eq:prop:maps1}
\end{align} 
and, for $n \in \mathbb{N}_0$,
\begin{align}
\fdx  W^I_n(x) = -\frac{2(n+1)}{b-a} V^I_{n+1}(x).
\label{eq:prop:maps2}
\end{align}
Moreover, for any $x \in \mathbb{R}$,
\begin{align}
W^I_n(x) = \frac{1}{2}[V^I_n(x) - V^I_{n+2}(x)]. 
\label{eq:prop:maps3}
\end{align}
\end{proposition}
\begin{proof}
All three identities \cref{eq:prop:maps1}--\cref{eq:prop:maps3} follow from classical identities between Chebyshev polynomials of the first and second kinds, their recurrence relationships, and an induction argument. We prove the first identity \cref{eq:prop:maps1} and note that the other two are found in the literature \cite[Sec.~18.9.10 \& 18.9.22]{OlverNIST}.

We consider \cref{eq:prop:maps1} and initially examine the case where $I=[-1,1]$. Consider the region $x \in [-1,1]$. From the definition of $\eU_n(x)$, for $n \geq 0$, in \cref{def:eU0} and \cref{def:eUk}, and a classical identity between Chebyshev polynomials of the first and seconds kinds \cite[Sec.~18.9.9]{OlverNIST}, it can be shown that
\begin{align}
\eU_n(x) = U_n(x), \;  n \in \mathbb{N}_0,  \;\; \text{for} \;\; x \in [-1,1],
\end{align}
where $U_n(x)$ are the Chebyshev polynomials of the second kind. Hence, another classical identity \cite[Sec.~18.9.21]{OlverNIST} reveals that 
\begin{align}
\fdx \eT_n(x) = n \eU_{n-1}(x) \;\; \text{for} \; x \in (-1,1). \label{eq:maps1}
\end{align}
Consider $|x|>1$. By a direct calculation we see that
\begin{align}
\fdx \eT_1(x) = \eU_0(x) \;\; \text{and} \;\; \fdx \eT_2(x) = 2\eU_1(x).
\end{align}
Moreover, another direct calculation shows that, for $|x|>1$ and $n \geq 2$,
\begin{align}
\fdx \eT_n(x) = -\sgn(x) \frac{n (x-\sgn(x)\sqrt{x^2-1})^n}{\sqrt{x^2-1}}.
\end{align}
Suppose that \cref{eq:prop:maps1} holds up to and including $\eT_{n-1}(x)$. Now,
\begin{align}
n \eU_{n-1}(x) = 2n \eT_{n-1}(x) + n \eU_{n-3}(x) = 2n \eT_{n-1}(x) + \frac{n}{n-2} \fdx \eT_{n-2}(x), \label{eq:maps2}
\end{align}
where the first equality follows by definition of $\eU_n(x)$ and the second equality follows from the induction argument. Substituting in the explicit definitions of $\eT_{n-1}(x)$ and $\fdx \eT_{n-2}(x)$, $|x|>1$, we find that the right-hand side of \cref{eq:maps2} is equal to \cref{eq:maps1}. Hence, \cref{eq:prop:maps1} holds for all $n \in \mathbb{N}$, $x \in \mathbb{R} \backslash \{-1,1\}$, when $I=[-1,1]$, by induction. It follows by applying the chain rule that \cref{eq:prop:maps1} holds for arbitrary $I = [a,b]$.  
\end{proof}

\begin{proposition}[Sqrt-Laplacian]
\label{prop:halfLaplacian}
Consider the interval $I=[a,b] \subset \mathbb{R}$, $a<b$. Then, for any $n \in \mathbb{N}$, $x \in \mathbb{R} \backslash \{a,b\}$, 
\begin{align}
(-\Delta)^{1/2}[\eT^I_n](x) = \frac{2n}{b-a} V^I_n(x),
\end{align}
and, for $n \in \mathbb{N}_0$, $x \in \mathbb{R} \backslash \{a,b\}$, 
\begin{align}
 (-\Delta)^{1/2}[W^I_n](x) = \frac{2(n+1)}{b-a} \eU^I_n(x). 
\end{align}
\end{proposition}
\begin{proof}
By \cref{lem:sum-space-regularity}, $\eT_n, W_m \in L^2(\mathbb{R}) \cap W^{1,s}(\mathbb{R})$, $s \in (1,2)$, $n \geq1$, $m\geq 0$. Thus the conditions of \cref{th:fdx-Hil} hold and $(-\Delta)^{1/2}[\eT_n] = \fdx \Hil[\eT_n]$ and $(-\Delta)^{1/2}[W_m] = \fdx \Hil[W_m]$ in the sense of distributions. The result then follows by applying \cref{prop:Hilbert-W} and \cref{prop:maps}.
\end{proof}

\begin{corollary}[Hilbert transforms of $V^I_n$ and $\eU^I_n$]
Consider the interval $I=[a,b] \subset \mathbb{R}$, $a<b$. Then, for $n \geq 0$,
\begin{align}
\Hil[\eU^I_n](x) = V^I_{n+1}(x) \;\; \text{and} \;\; \Hil[V^I_{n+1}](x) = -\eU^I_n(x). 
\end{align}
\end{corollary}
\begin{proof}
We note that
\begin{align}
\Hil[\eU^I_n] = \Hil \fdx \left[\frac{b-a}{2(n+1)} \eT_{n+1}^I \right] = -\fdx  \frac{b-a}{2(n+1)} W^I_n = V^I_{n+1}, 
\end{align}
where the first and third equalities follow from \cref{prop:maps} and the second equality holds as $\eT_{n+1} \in W^{1,p}(\mathbb{R})$ for any $p \in (1,2)$ \cite[Th.~3.2]{Johansson1999}. The second result follows as $\Hil$ is an anti-involution.
\end{proof}

\begin{corollary}[Regularity of $(-\Delta)^{1/2}\eT^I_n$ and $(-\Delta)^{1/2}W^I_n$]
Consider the interval $I=[a,b] \subset \mathbb{R}$, $a<b$. Then, $(-\Delta)^{1/2}\eT^I_n, (-\Delta)^{1/2}W^I_m \in L^p(\mathbb{R})$ for any $p \in (1,2)$, $n \in \mathbb{N}$, $m \in \mathbb{N}_0$. 
\end{corollary}
\begin{proof}
A direct check reveals that $V^I_m(x) \in L^p(\mathbb{R})$ for any $p \in [1,2)$, $m \in \mathbb{N}_0$. Hence, by \cref{prop:halfLaplacian}, $(-\Delta)^{1/2}\eT^I_m \in L^p(\mathbb{R})$ for any $p \in (1,2)$. Moreover, for any $p \in (1,2)$,
\begin{align}
\begin{split}
&\| (-\Delta)^{1/2}W^I_n \|_{L^p(\mathbb{R})} 
=  \frac{2(n+1)}{b-a} \| \eU^I_n \|_{L^p(\mathbb{R})} \\
& \indent =  \frac{2(n+1)}{b-a} \| \Hil V^I_{n+1} \|_{L^p(\mathbb{R})} 
\leq \frac{2C(n+1)}{b-a} \| V^I_{n+1} \|_{L^p(\mathbb{R})} < \infty.
\end{split}
\end{align}
\end{proof}

The Hilbert space $H^{1/2}(\mathbb{R})$ is equipped with the inner-product
\begin{align}
\langle u, v \rangle_{H^{1/2}(\mathbb{R})} \coloneqq \langle u, v \rangle_{L^2(\mathbb{R})} + \langle (-\Delta)^{1/4} u,  (-\Delta)^{1/4} v\rangle_{L^2(\mathbb{R})}.
\end{align}
The following proposition shows that $\eT_n(x)$ and $W_n$ have partial orthogonality with respect to the $H^{1/2}(\mathbb{R})$ inner product. 
\begin{proposition}[Partial orthogonality in $H^{1/2}(\mathbb{R})$]
$\eT_n$, $W_m$, $n\in \mathbb{N}$, $m \in \mathbb{N}_0$ satisfy
\begin{align}
\langle (-\Delta)^{1/4} \eT_j, (-\Delta)^{1/4} \eT_l \rangle_{L^2(\mathbb{R})} = l \delta_{jl}, \; \langle (-\Delta)^{1/4} W_j, (-\Delta)^{1/4} W_l \rangle_{L^2(\mathbb{R})} = (l+1)\delta_{jl},
\end{align}
where $\delta_{jl}$ is the Kronecker delta. 
\end{proposition}
\begin{proof}
We note that
\begin{align}
\begin{split}
\langle (-\Delta)^{1/4} \eT_j, (-\Delta)^{1/4} \eT_l \rangle_{L^2(\mathbb{R})} 
&= \langle \eT_j, (-\Delta)^{1/2} \eT_l \rangle_{L^2(\mathbb{R})} = l \langle \eT_j, V_l \rangle_{L^2(\mathbb{R})}\\
&=l \langle T_j, V_l \rangle_{L^2(-1,1)} = l \delta_{jl}. 
\end{split}
\end{align}
Similarly,
 \begin{align}
\begin{split}
\langle (-\Delta)^{1/4} W_j, (-\Delta)^{1/4} W_l \rangle_{L^2(\mathbb{R})} 
&= \langle W_j, (-\Delta)^{1/2} W_l \rangle_{L^2(\mathbb{R})} = (l+1) \langle W_j, \eU_l \rangle_{L^2(\mathbb{R})}\\
&=(l+1) \langle W_j, U_l \rangle_{L^2(-1,1)} = (l+1) \delta_{jl}. 
\end{split}
\end{align}
\end{proof}

\section{Spectral method}
\label{sec:spectral-method}
  The goal of this section is to define the various functions, vectors, and matrices that appear in the spectral method described in \cref{alg:spectral-method}. 

\begin{algorithm}[ht]

\caption{ Spectral method for solving \cref{eq:fpde}}
\label{alg:spectral-method}
{\bf Input:} Values for $\lambda$, $\mu$ and $\eta$, the right-hand side $f$, the number of intervals $K$, a tuple of truncation degrees $\vect{n} = (n_1, n_2, \dots, n_K)$, and a tuple of intervals $\vect{I} = (I_1, I_2, \dots, I_K)$.

{\bf Output:} The coefficient vector ${\bf u}^{\vect{I}, +}_{\vect{n}}$ such that $u(x) \approx \Skp(x) {\bf u}^{\vect{I}, +}_{\vect{n}}$.

{\bf Solve for ${\bf u}^{\vect{I}, +}_{\vect{n}}$:} 
\begin{algorithmic}[1]
\State{\texttt{EXPAND.} Via an $\epsilon$-truncated SVD factorization, solve the least squares problem $\min_{\vectt{f}^{\vect{I}, *}_{\vect{n}}} \| \tens{G} \vectt{f}^{\vect{I}, *}_{\vect{n}} - {\bf b}\|_{\ell^2}$,  such that $f(x) \approx \Skd(x) \vectt{f}^{\vect{I}, *}_{\vect{n}}$ (cf.~\cref{sec:expansion}).}
\State{\texttt{ASSEMBLE.} Assemble the matrix $\tens{L}^{\vect{I}, +}_{\vect{n}}$, cf.~\cref{eq:Ap}, \cref{eq:Lmatrix1}, and  \cref{eq:Lmatrix2}.}
\State{\texttt{SOLVE.} Solve the linear system $\tens{L}^{\vect{I}, +}_{\vect{n}} {\bf u}^{\vect{I}, +}_{\vect{n}} = {\bf f}^{\vect{I}, *}_{\vect{n}}$.}
\end{algorithmic}

{\bf Evaluate $\Skp(x) {\bf u}^{\vect{I}, +}_{\vect{n}}$:} 
\begin{algorithmic}[1]
\State{\texttt{INTEGRATE.} Compute the $4K$ continuous inverse Fourier transforms to find the appended sum space functions as described in \cref{sec:single-interval}. Note that if the intervals are translations (and not scalings), this reduces to four continuous inverse Fourier transforms as shown in \cref{sec:app:appended-sum-space}.}
\State{\texttt{EVALUATE.} Evaluate the sum $u(x_i) \approx \Skp(x_i) {\bf u}^{\vect{I}, +}_{\vect{n}}$ at the grid points $x_i$, $i \in \{1,2,\dots,M\}$.}
\end{algorithmic}
\end{algorithm}

\subsection{Operators}
For any $x \in \mathbb{R}$, we denote the $n$-th sum space and dual sum space quasimatrices, centred at an interval $I = [a,b] \subset \mathbb{R}$, by
\begin{align}
 {\bf S}^I_n(x)  &\coloneqq 
\begin{pmatrix}
\eT_0(x) \; \vline & W^I_0(x) & \eT^I_1(x) \; \vline & \cdots \;\;\; \vline & W^I_n(x) & \eT^I_{n+1}(x)
\end{pmatrix}  \\
 {\bf S}^{I,*}_n(x)  &\coloneqq 
\begin{pmatrix}
\eU^I_{-2}(x) \; \vline & V^I_0(x) & \eU^I_{-1}(x) \; \vline & \cdots \;\;\; \vline & V^I_{n+2}(x) & \eU^I_{n+1}(x)
\end{pmatrix}. 
\end{align}
 Note that ${\bf S}^I_n(x) \in L^\infty(\mathbb{R}) \times H^{1/2}(\mathbb{R})^{2n+2}$ and ${\bf S}^{I,*}_n(x) \in H^{-1/2}(\mathbb{R})^{2n+7}$.  A quasimatrix is a matrix whose ``columns'' are functions defined on $\mathbb{R}$ \cite[Lec.~5]{Stewart1998}. In  ${\bf S}^I_n(x)$  the first column contains the constant function. Thereafter, we block together the columns $W^I_k(x)$ and $\eT^I_{k+1}(x)$ for $k = 0,\dots,n$. Similarly in  ${\bf S}^{I,*}_n(x)$,  the first column contains $\eU^I_{-2}(x)$.  Thereafter, we block together the columns $V^I_{k}(x)$ and $\eU^I_{k-1}(x)$ for $k = 0,\dots,n+2$. We note that  ${\bf S}^{I,*}_n(x)$  has four more columns than  ${\bf S}^I_n(x)$.  These are required to represent the identity map between the two sum spaces exactly.

\begin{proposition}[Quasimatrix operators]
\label{prop:quasimatrix-operators}
Consider the interval $I=[a,b] \subset \mathbb{R}$, $a<b$. Then, for $x \in \mathbb{R}$, 
\begin{align}
 {\bf S}^I_n(x)  =  {\bf S}^{I,*}_n(x) \tens{E}_{n}, \;\;\; \Hil [ {\bf S}^I_n ](x) =  {\bf S}^{I,*}_n(x) \tens{H}_n, 
\end{align}
and for $x \in \mathbb{R} \backslash \{a,b\}$,
\begin{align}
\fdx  [ {\bf S}^I_n](x)  =  {\bf S}^{I,*}_n(x) \tens{D}^I_n,  \; \text{and} \; (-\Delta)^{1/2}   [{\bf S}^I_n](x)  =  {\bf S}^{I,*}_n(x)  \tens{A}^I_n. 
\end{align}
Let
\begin{align}
\tens{N} =
\begin{pmatrix}
-1 & 0 \\
0 & 1
\end{pmatrix}, \;\;
\tens{M} =
\begin{pmatrix}
0 & 1 \\
1 & 0
\end{pmatrix},\;\;
\tens{Z} = 
\begin{pmatrix}
0 & 0 \\
0 & 0
\end{pmatrix},\;\;
\vectt{0} =\begin{pmatrix}
0\\
0
\end{pmatrix}, \;\;
\vectt{e}_2 =\begin{pmatrix}
0\\
1
\end{pmatrix}.
\end{align}
Then  $\tens{E}_{ n }, \tens{H}_{ n }, \tens{D}^I_{ n }, \tens{A}^I_{ n }  \in \mathbb{R}^{(2n+7) \times (2n+3)}$  are defined by
\begin{align}
\begin{split}
    \tens{E}_n   & \coloneqq
   \begin{pmatrix}
   -1        &   \vectt{0}^\top    &     \hdots          &    \vectt{0}^\top      \\
   \vectt{0} &- \tens{N} /2     &               &        \\
   \vectt{e}_2      &  \tens{Z}       & \ddots     &          \\
    \vectt{0}  & \tens{N} /2     &  \ddots    &  - \tens{N} /2    \\
    \vdots   &       &  \ddots     &    \tens{Z}      \\
       \vectt{0}       &       &               &    \tens{N} /2
   \end{pmatrix}, \quad
    \tens{D}^I_n    \coloneqq
   \begin{pmatrix}
   0        &  \vectt{0}^\top        &      \cdots         &     \vectt{0}^\top        \\
  \vectt{0}            & \tens{Z}       &               &        \\
   \vectt{0}          &\frac{2}{b-a} \tens{N}      &             &          \\
        \vdots    &        &              \ddots &   \\
            \vdots        &        &            &             \frac{2(n+1)}{b-a} \tens{N}    \\
  \vectt{0}              &        &            &              \tens{Z}   \\
   \end{pmatrix},\\
    \tens{H}_n    &\coloneqq
   \begin{pmatrix}
   0        &     \vectt{0}^\top     & \cdots              &       \vectt{0}^\top     \\
   \vectt{0}           &- \tens{M} /2     &               &        \\
    \vectt{0}             & \tens{Z}       & \ddots     &          \\
      \vectt{0}           & \tens{M} /2     &  \ddots    &  - \tens{M} /2    \\
   \vdots         &       &  \ddots     &    \tens{Z}      \\
     \vectt{0}            &       &               &    \tens{M} /2
   \end{pmatrix}, 
    \quad
    \tens{A}^I_n    \coloneqq
   \begin{pmatrix}
   0        &      \vectt{0}^\top       &  \cdots    &     \vectt{0}^\top  \\
   \vectt{0}           & \tens{Z}       &                    &  \\
     \vectt{0}            &\frac{2}{b-a} \tens{M}       &             &         & \\
  \vdots          &        &             \ddots &   \\
        \vdots   &        &                       & \frac{2(n+1)}{b-a} \tens{M}    \\
           \vectt{0}      &        &                       &  \tens{Z}   \\
   \end{pmatrix}.
   \end{split}
\end{align}
\end{proposition}
\begin{proof}
The result follows from Propositions \labelcref{prop:Hilbert-W}, \labelcref{prop:maps}, and \labelcref{prop:halfLaplacian}.
\end{proof}

\begin{proposition}[Jacobi matrix]
Consider the interval $I=[a,b] \subset \mathbb{R}$, $a<b$. Let $ {\bf S}^I_\infty$  denote the quasimatrix of the entire sum space, i.e.~$ {\bf S}_n^I$  with $n = \infty$. Moreover, let $ {\bf S}^{\circ,I}_\infty$  denote $ {\bf S}^I_\infty$  but with the first column $(\eT_0)$ removed. For any $x \in \mathbb{R}$, let $y=2/(b-a)(x-(a+b)/2)$. Then,
\begin{align}
y  {\bf S}^{\circ,I}_\infty(x) =  {\bf S}^I_\infty(x)   \tens{J}  
\;\;
\text{where}
\;\; 
 \tens{J}  \coloneqq 
\begin{pmatrix}
0 & 1/2 & 0&\\
0 & 0 & 1/2&\\
0 & 0 & 0&\ddots\\
1/2 & 0& 0&\\
0 & 1/2 & 0& \\
& & \ddots & 
\end{pmatrix}. 
\end{align}
\end{proposition}
\begin{proof}
The Jacobi matrix of $W_n$ is the same as the Jacobi matrix for $U_n$. The entries of the columns associated with $\eT_n$ follow from \cref{prop:recurrence}.
\end{proof} 

\begin{remark}
The matrices  $\tens{E}_n$, $\tens{H}_n$, $\tens{D}^I_n$,  $\tens{A}^I_n$, and $\tens{J}$  are sparse. 
\end{remark}

\subsection{Single interval}
\label{sec:single-interval}
We now collect the final ingredients for the spectral method centred on a single interval at $[-1,1]$ to solve \cref{eq:fpde} for $x \in \mathbb{R}$. Consider the expansion of the right-hand side $f(x)$ in the dual sum space and truncate at degree $n$,
\begin{align}
f(x) \approx {{\bf S}}^*_n(x)  \vectt{f}_{n}^*, \;\; \vectt{f}_{n}^* 
 \in \mathbb{R}^{2n+7}.
\end{align}
Consider the matrix $ \tens{L}_n  \coloneqq \lambda \tens{E}_n + \mu \tens{H}_n + \eta \tens{D}^I_n + \tens{A}^I_n \in  \mathbb{R}^{(2n+7)\times(2n+3)}$.  We note that  $\tens{L}_n$  is a linear operator mapping from the sum space ${{\bf S}}_n(x)$ to the dual sum space ${{\bf S}}^*_n(x)$. However, $ \tens{L}_n  \in \mathbb{R}^{(2n+7)\times(2n+3)}$ is a rectangular matrix and the linear system
\begin{align}
 \tens{L}_n  \vectt{u}_{ n} =  \vectt{f}^*_{ n },  \;\;  \vectt{u}_n \in \mathbb{R}^{2n+3}, 
\end{align}
to find the coefficients of $u(x)$ in the truncated sum space is overdetermined. We wish to construct a method where the expansion of $u(x)$ is determined exactly i.e.~we want to avoid a least squares solution for $\vectt{u}_{ n }$. Let ${{\bf S}}^\circ_n(x)$ denote the sum space ${{\bf S}}_n(x)$ without the first block, i.e.~we drop the first function $\eT_0(x)$. We construct the following \emph{appended sum space}
\begin{align}
{{\bf S}}^+_n(x) \coloneqq
\begin{pmatrix}
\eT_0(x) \; \vline & v_0(x) & \tilde{u}_{-1}(x) \; \vline & v_1(x) & \tilde{u}_{0}(x) \; \vline & {{\bf S}}^\circ_n(x)
\end{pmatrix},
\end{align}
 where ${{\bf S}}^+_n \in L^{\infty}(\mathbb{R}) \times H^{1/2}(\mathbb{R})^{2n+6}$.  Here, the functions $v_n(x)$ and $\tilde{u}_{n}(x)$ satisfy
\begin{alignat}{2}
\mathcal{L}_{\lambda, \mu, \eta} \tilde{u}_{n}(x) &=  \eU_{n}(x), \label{eq:columns1}  \\ 
\mathcal{L}_{\lambda, \mu, \eta}  v_n(x) &=  V_{n}(x).
\label{eq:columns2} 
\end{alignat}

Let $ \tens{L}^\circ_n  \in \mathbb{R}^{(2n+7) \times (2n+2)}$ denote the matrix $ \tens{L}_n $ without the first column. By considering the map induced by $\mathcal{L}_{\lambda,\mu,\eta}$ from the appended sum space ${{\bf S}}^+_n(x)$ to the dual sum space ${{\bf S}}^*_n(x)$, we see that the solution of \cref{eq:fpde} can be approximated by solving
\begin{align}
 \tens{L}^+_n  \vectt{u}^+_{ n } =  \vectt{f}^*_{n}, \;\;  \vectt{u}^+_{ n } \in \mathbb{R}^{2n+7},  \label{eq:append2dual}
\end{align}
where $ \tens{L}^+_n  \in \mathbb{R}^{(2n+7)\times(2n+7)}$ is the following square matrix:
\begin{align}
 \tens{L}^+_n  \coloneqq 
\scriptsize
\setcounter{MaxMatrixCols}{50}
\begin{pmatrix}
-\lambda & 0 & 0 & 0 & 0 & \rvline& \\
0& 1 &0 &0 &0 & \rvline& \\
0 & 0 &1 &0 &0 & \rvline& \\
0&0 & 0 &1 & 0 & \rvline &\\
\lambda &0 &0 &0 &1& \rvline& \text{\huge$ \tens{L}^\circ_n $}  \\
0& 0 & 0 & 0 & 0 & \rvline& \\
\vdots &\vdots&\vdots&\vdots&\vdots& \rvline& \\
0& 0 & 0 & 0 & 0& \rvline& \\
\end{pmatrix}  \in \mathbb{R}^{(2n+7)\times(2n+7)}, \;\; \tens{L}^\circ_n  \in \mathbb{R}^{(2n+7) \times (2n+2)}. 
 \label{eq:Ap}
\end{align}
$ \tens{L}^+_n $ is constructed by pre-appending $ \tens{L}_n    \in \mathbb{R}^{(2n+7) \times (2n+3)} $ with four columns containing one nonzero entry each and commuting the column associated with $\eT_0(x)$ to be first. The entries in the new columns have a value of $1$ and are in the positions $(2,2)$, $(3,3)$, $(4,4)$, $(5,5)$. By solving the linear system \cref{eq:append2dual}, we conclude that $u(x) \approx {{\bf S}}^+_n(x) \vectt{u}^+_{ n }$. In order to realize the values of ${{\bf S}}^+_n(x) \vectt{u}^+_{ n }$, we require approximations for the solutions in equations  \cref{eq:columns1} and \cref{eq:columns2}. These can be found using a fast Fourier transform (FFT) \cite{Cooley1965}, specialized quadrature formulas or explicit expressions. We give more details below. 

\begin{remark}
The solutions $\tilde{u}_n(x)$ and $v_n(x)$ are not dependent on the right-hand side $f$ of \cref{eq:fpde}, although they are dependent on the constants $\lambda$, $\mu$, and $\eta$. 
\end{remark}

\begin{remark}
\label{rem:stabiliseL}
The choice of $v_0(x),  \tilde{u}_{-1}(x), v_1(x)$ and $\tilde{u}_{0}(x)$ as the additional functions is not the only option. Indeed, if the goal was to improve the conditioning of  $\tens{L}^+_n$,  then a better choice would be $\tilde{u}_{n+1}(x)$ and $v_{n+2}(x)$ rather than $\tilde{u}_{0}(x)$ and $v_1(x)$.  In this case, the conditioning of  $\tens{L}^+_n$  only grows linearly with $n$, is robust for large parameter ranges of $\lambda, \mu$ and $\eta$, and there exists a diagonal preconditioner such that the condition number is independent of $n$ as described in \cref{sec:conditioning}.  With the presented choice, the conditioning degrades if both $\mu \to 0$ and $\lambda \to 0$. The disadvantage of  $\tilde{u}_{n+1}(x)$ and $v_{n+2}(x)$ is that if multiple solves are required with different degrees $n$, then additional work is required to compute the required additional functions. Furthermore, in numerical experiments, we observed that the approximate identity mapping from the appended sum space ${{\bf S}}^+_n(x)$ to the dual sum space ${{\bf S}}^*_n(x)$ is unstable when using  $\tilde{u}_{n+1}(x)$ and $v_{n+2}(x)$ as the additional functions. Hence, for discretized time-dependent problems where we are required to map the current solution iterate expanded in the appended sum space to the expansion in the dual sum space (for the right-hand side), this poses a distinct issue. The choice of the additional functions is context dependent. 
\end{remark}

\begin{proposition}[Fourier transforms of $V_n(x)$]
\label{prop:FourierTransformsV}
The weighted Chebyshev polynomials, $V_n(x)$, $n\geq 0$, have the Fourier transforms
\begin{align}
\mathcal{F}[V_n](\omega) &= (-\mi)^n \pi J_n(\omega), \label{eq:FourierVk}
\end{align}
where $J_n$, $n \in \mathbb{N}_0$, denote the Bessel functions of the first kind \cite[Sec.~10.2]{OlverNIST}.  
\end{proposition}
\begin{proof}
This is a known result and follows from an application of Parseval's integral and recurrence relationships between Bessel functions \cite[Ch.~13]{Watson1922}.
\end{proof}

\begin{proposition}[Fourier transforms of $\eU_n(x)$]
\label{prop:FourierTransformsU}
The extended Chebyshev functions $\eU_{-1}(x)$, $\eU_0(x)$ have the following Fourier transforms,
\begin{align}
\mathcal{F}[\eU_{-1}](\omega) &= \mi  \pi \mathrm{sgn}(\omega) J_0(\omega), \;\;  \;\; \mathcal{F}[\eU_0](\omega) = \pi J_1(|\omega|).
\end{align}
\end{proposition}
By taking the Fourier transform on both sides in  \cref{eq:columns1} and  \cref{eq:columns2}, utilizing Propositions \labelcref{prop:FourierTransformsV} and \labelcref{prop:FourierTransformsU} as well as \cref{eq:fourier-frac-laplace}, \cref{eq:fourier-hilbert}, and $\mathcal{F}[\fdx u](\omega) = \mi \omega \mathcal{F}[u](\omega)$, we observe that
\begin{align}
\mathcal{F}[\tilde{u}_{-1}](\omega) &=  \mi \pi(\lambda - \mi \mu \, \mathrm{sgn}(\omega) + \mi \eta\omega + |\omega|)^{-1} |\omega|^{-1} \omega  J_0(\omega), \label{eq:FeU-1} \\ 
\mathcal{F}[\tilde{u}_{0}](\omega) &=  \pi (\lambda - \mi \mu \, \mathrm{sgn}(\omega) + \mi \eta\omega + |\omega|)^{-1} J_1(|\omega|),\\
\mathcal{F}[v_n](\omega) &=  (-\mi)^n \pi(\lambda - \mi \mu \, \mathrm{sgn}(\omega) + \mi \eta\omega + |\omega|)^{-1} J_n(\omega). \label{eq:FVk}
\end{align}
The continuous inverse Fourier transforms to compute the solutions $\tilde{u}_{-1}(x)$, $\tilde{u}_{0}(x)$, $v_0(x)$ and $v_1(x)$ can be approximated via specialized quadrature rules, an FFT, or explicit expressions for certain ranges of $x$.  The FFT as an approximation of the continuous inverse Fourier transform is equivalent to the trapezoid rule applied to the Fourier transform. Particularly since the integrands are oscillatory, if the trapeziums are too wide (equivalent to an FFT with too few samples), the approximation of the continuous inverse Fourier transform is poor, even if the FFT itself is computed to a high precision. The FFT approximation is discussed in more detail in \cref{sec:FFT}. If the trapezoid rule is not sufficiently accurate then we opt to utilize Mathematica's \texttt{NIntegrate} which uses quadrature rules that are very robust for oscillatory integrands which (at the time of writing) is the routine that best balanced the speed of the quadrature with the accuracy obtained, outperforming Julia in-house alternatives for this purpose. For more details of how to use a Julia wrapper for \texttt{NIntegrate} routine in this context, we refer the reader to \texttt{src/mathematica.jl} in \cite{SumSpacesMathLink}. Alternatively one could simply increase the sample size of the FFT or use a more accurate FFT-accelerated approximation of the continuous inverse Fourier transform such as those described in \cite{anderson2020high}. 

We do not have exact expressions for the Fourier transforms of $\eU_n(x)$ for $n \geq 1$. However, this can be overcome by first applying a Hilbert transform to the equations in \cref{eq:columns1} to obtain, for $n \geq -1$,
\begin{align}
\left(\lambda \Hil - \mu \mathcal{I} + \eta (-\Delta)^{1/2} - \fdx \right)[\tilde{u}_{n}](x) &= \Hil[\eU_n](x)  = V_{n+1}(x). \label{eq:uUk}
\end{align}
By taking the Fourier transform of \cref{eq:uUk}, we find that
\begin{align}
\mathcal{F}[\tilde{u}_{n}](\omega) &=  (-\mi)^{n+1} \pi (-\mi \lambda \, \mathrm{sgn}(\omega) - \mu + \eta |\omega| - \mi \omega)^{-1} J_{n+1}(\omega). \label{eq:uUk2}
\end{align}

In \cref{sec:special-cases} we discuss edge cases of \cref{eq:fpde} and how they are treated in our numerical method. In summary, when $\lambda = \mu = 0$ and $\eta \in \mathbb{R}$, the additional functions are no longer required and their associated columns (and related rows) are removed from \cref{eq:Ap}. In particular the columns associated with $\eT_0(x)$, $v_0(x)$, $\tilde{u}_{-1}(x)$, $v_1(x)$, and $\tilde{u}_{0}(x)$, and the rows associated with $\eU_{-2}(x)$, $V_0(x)$, $\eU_{-1}(x)$, $V_{n+2}(x)$, and $\eU_{n+1}(x)$ are removed reducing $ \tens{L}^+_n $ from a $(2n+7)\times(2n+7)$ matrix to a $(2n+2)\times(2n+2)$ matrix. The exact mappings are still conserved. When $\lambda = 0$ but $|\mu|>0$ and $\eta \in \mathbb{R}$, we keep the additional functions, although we advise using the additional functions $v_{n+2}(x)$ and $\tilde{u}_{n+1}(x)$ instead of $v_1(x)$ and $\tilde{u}_{0}(x)$ to alleviate ill-conditioning.

\subsection{Multiple intervals}
\label{sec:multiple-intervals}

Although the sum space is dense on the interval it is centred on, it is not dense in $\mathbb{R}$ and, therefore, cannot approximate general solutions. By combining intervals together we enlarge the subset of density in $\mathbb{R}$. Moreover, in some applications, it can be helpful to decompose the domain into multiple intervals. This can be in regions where (spatially varying) coefficients or right-hand sides $f$ have discontinuities best accounted for by approximating them with sum spaces centred over two or more intervals. Experimentally, we also found that the tails of a function $f$ outside the intervals were better approximated when using the combination of sum spaces centred at multiple intervals. Examples of low-order multiple-interval sum space functions are given in \cref{fig:multiple-interval-spaces}.  

In this section, we will discuss how to combine sum spaces centred at multiple intervals together. We will show that the operator $\mathcal{L}_{\mu, \lambda, \eta}$ decouples across the different affine transformed sum spaces and this results in a linear system with a block diagonal matrix. Hence, each block can be solved separately resulting in a highly parallelizable method. 

\begin{figure}[h]
\centering
\subfloat[{$\eT_1(x)$ centred at $[-3,-1]$, $[-1,1]$ and $[1,3]$.}]{\includegraphics[width =0.49 \textwidth]{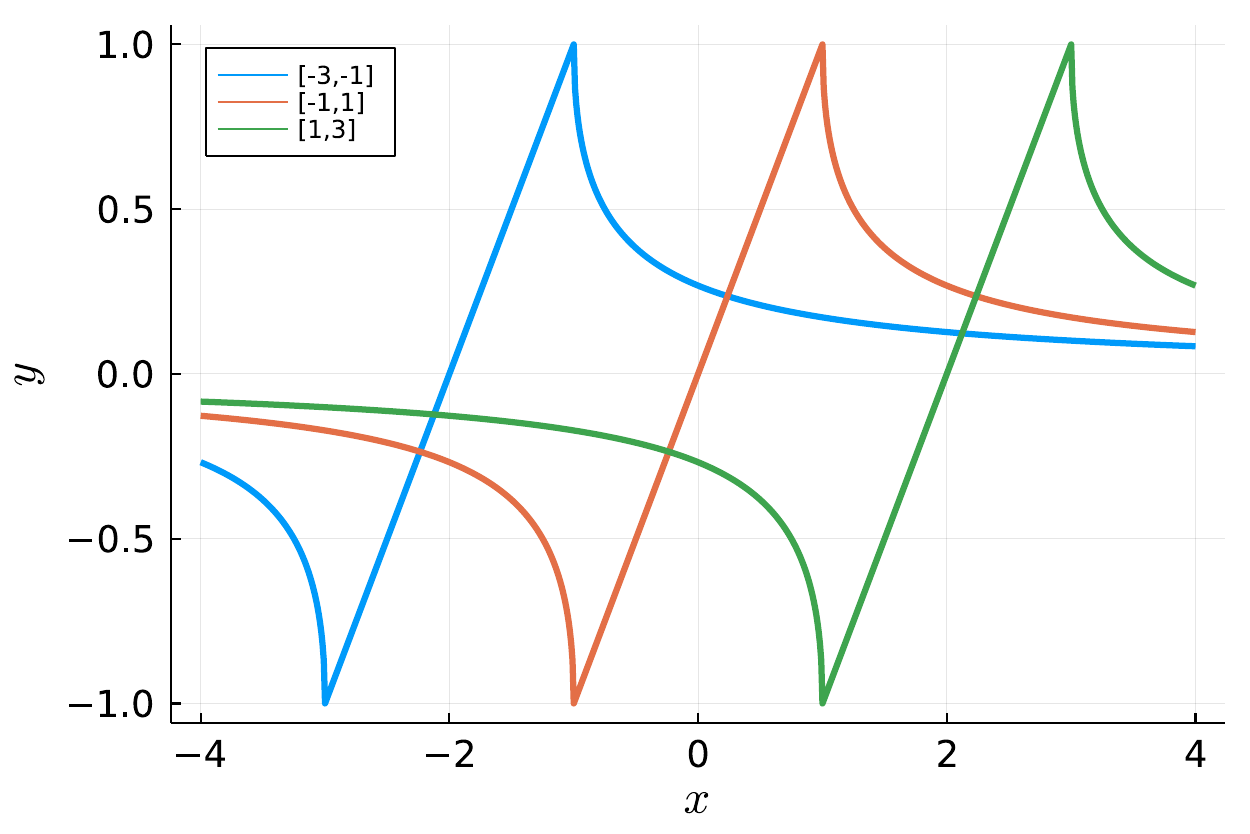}}
\subfloat[{$W_0(x)$ centred at $[-3,-1]$, $[-1,1]$ and $[1,3]$.}]{\includegraphics[width =0.49 \textwidth]{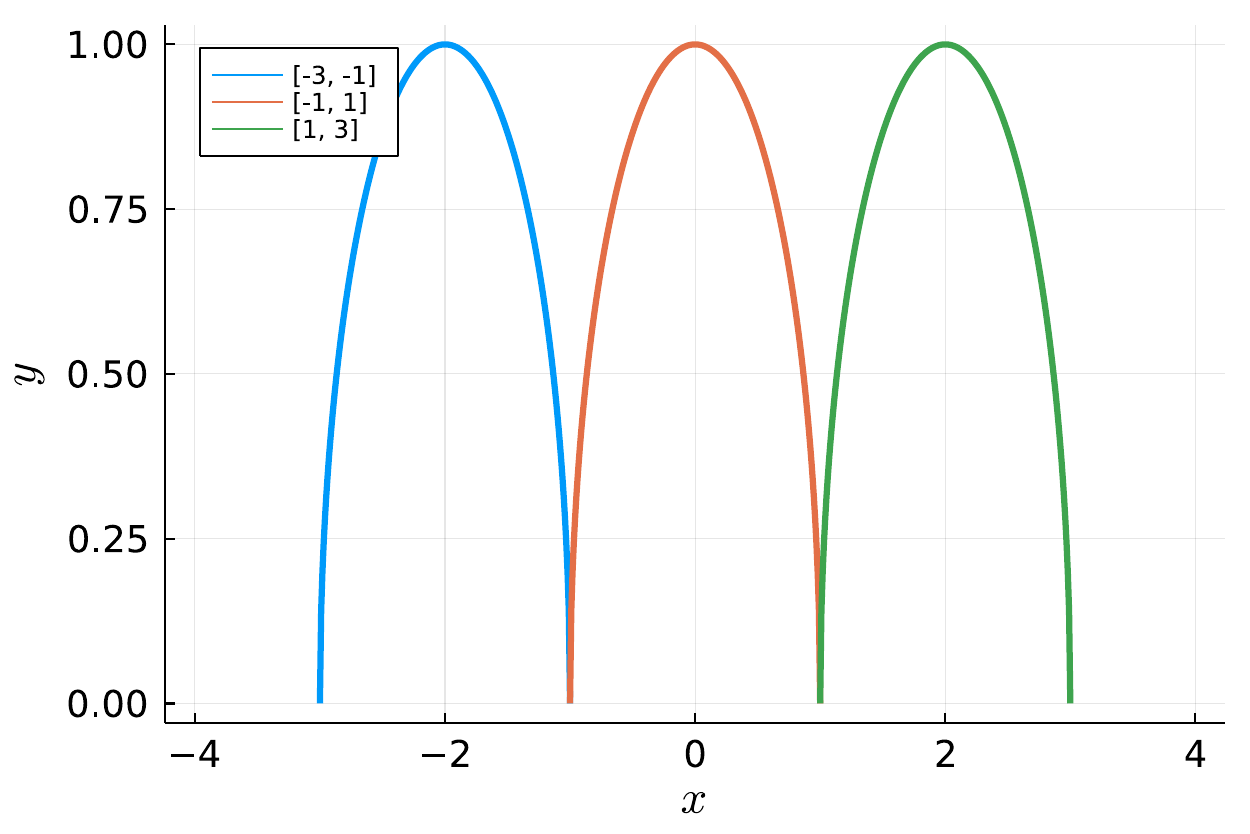}}
\caption{Plots of $\eT_1(x)$ and $W_0(x)$ centred at the three different intervals $[-3,-1]$, $[-1,1]$ and $[1,3]$. We note that $\eT_1(x)$ has global support and those centred on different intervals always overlap. Whereas, $W_0(x)$ has compact support on the interval it is centred on.}\label{fig:multiple-interval-spaces}
\end{figure}

We fix the closed and bounded intervals $I_k = [a_k, b_k]$, $k=1,\dots,K$, which may be connected but not overlap (except at a point). Consider a truncation of the expansion of a function, $g(x)$, $x \in \mathbb{R}$, in the appended multiple-interval sum space:
\begin{align}
\begin{split}
g(x) &\approx \tilde{c}_0 \eT_0(x)\\
&\indent + \left.\sum_{k=1}^K \right(
c_{v^{I_k}_{0}}v^{I_k}_0(x) 
+c_{\tilde{u}^{I_k}_{-1}}\tilde{u}^{I_k}_{-1}(x) 
+ c_{v^{I_k}_{1}}v^{I_k}_1(x)
+ c_{\tilde{u}^{I_k}_{0}}\tilde{u}^{I_k}_{0}(x) \\
&\indent \indent+\left. \sum_{j=1}^{n_k} [\tilde{c}^{I_k}_j \eT^{I_k}_j(x) + c^{I_k}_j W^{I_k}_{j-1}(x)] \right).
\end{split}
\label{eq:f-element-expansion}
\end{align}
Let  ${\bf S}_{n_k}^{\circ,I_k,+}(x) \in H^{1/2}(\mathbb{R})^{2n_k+6}$  denote the appended sum space centred on the interval $I_k$ not including the constant term $\eT^{I_k}_0(x)$. Let $\vect{n} = (n_1,\dots, n_K)$, $N = \sum_{k=1}^K n_k$, and $\vect{I} = (I_1,\dots,I_K)$. Then, in quasimatrix notation, \cref{eq:f-element-expansion} can be rewritten as $g(x) \approx \Skp(x) \vectt{c}^{\vect{I},+}_{\vect{n}}$ where 
\begin{align}
\setcounter{MaxMatrixCols}{50}
\Skp(x) &= 
(
\eT_0(x) \; | \;  \underbrace{{\bf S}_{n_1}^{\circ,I_1,+}(x)}_{ \in H^{1/2}(\mathbb{R})^{2n_1 + 6}, }  \; | \; \cdots \;\;\; |  \;  \underbrace{ {\bf S}_{n_K}^{\circ,I_K,+}(x)}_{ \in H^{1/2}(\mathbb{R})^{2n_{K} + 6}, }
),\\
 {\bf S}_{n_k}^{\circ,I_k,+}(x)  &=
\begin{pmatrix}
v^{I_k}_0(x) & \tilde{u}^{I_k}_{-1}(x) & v^{I_k}_1(x) & \tilde{u}^{I_k}_{0}(x)  \; \vline &  {\bf S}_{n_k}^{\circ,I_k}(x)
\end{pmatrix}, \label{eq:SL1I+}\\
 {\bf S}_{n_k}^{\circ,I_k}(x)  &=
\begin{pmatrix}
W^{I_k}_{0}(x) & \eT^{I_k}_1(x) \; \vline & \cdots \;\;\; \vline & W^{I_k}_{n_k}(x) & \eT^{I_k}_{n_k+1}(x)
\end{pmatrix} \\
\vectt{c}^{\vect{I},+}_{\vect{n}} &=
\begin{pmatrix}
\tilde{c}_0  \; \vline & (\vectt{c}_{n_1}^{\circ,I_1,+})^\top \; \vline & \cdots \;\;\; \vline  & (\vectt{c}_{n_K}^{\circ,I_K,+})^\top
\end{pmatrix}^\top  \in \mathbb{R}^{2N + 6K + 1},\\
\vectt{c}_{n_k}^{\circ,I_k,+} &= 
\begin{pmatrix}
c_{v^{I_k}_{0}} & c_{\tilde{u}^{I_k}_{-1}} & c_{v^{I_k}_{1}} & c_{\tilde{u}^{I_k}_{0}}  \; \vline & (\vectt{c}_{n_k}^{\circ,I_k})^\top 
\end{pmatrix}^\top  \in \mathbb{R}^{2n_k + 6}, \\
\vectt{c}_{n_k}^{\circ,I_k}  &=
\begin{pmatrix}
c^{I_k}_{1} & \tilde{c}^{I_k}_1 \; \vline & \cdots \;\;\; \vline & c^{I_k}_{n_k+1} & \tilde{c}^{I_k}_{n_k+1}
\end{pmatrix}^\top  \in \mathbb{R}^{2n_k + 2}.  
\end{align}

Unlike in ${{\bf S}}^+_n(x)$, the sum spaces centred on each interval are ordered sequentially in $\Skp(x)$ rather than interleaving the Chebyshev polynomials/functions of the same degree across all the intervals. This is because each block in $\Skp(x)$ directly correlates to a block diagonal in the matrix after discretization, allowing for a clear decoupling of the problem. We denote the multiple-interval sum space by $\Sk$ which is constructed as $\Skp$ but without the initial four terms $v^{I_k}_0(x), \tilde{u}^{I_k}_{-1}(x), v^{I_k}_1(x), \tilde{u}^{I_k}_{0}(x)$ at the start of each block in \cref{eq:SL1I+}.

Similarly, we define the dual multiple-interval sum space $\Skd(x)$ as
\begin{align}
\Skd(x) &= 
(
\eU_{-2}^{I_1}(x) \; | \; \underbrace{{ \bf S }_{n_1}^{\circ,I_1,*}(x)}_{ H^{-1/2}(\mathbb{R})^{2n_1+6} } \; |\; \cdots \;\;\; |\;   \underbrace{{ \bf S }_{n_K}^{\circ,I_K,*}(x)}_{ H^{-1/2}(\mathbb{R})^{2n_K+6} }
),\\
{ \bf S }_{n_k}^{\circ,I_k,*}(x)  &=
\begin{pmatrix}
V_0^{I_k}(x) & \eU^{I_k}_{-1}(x) \; \vline & \cdots \;\;\; \vline & V^{I_k}_{n_k+2}(x) & \eU^{I_k}_{n_k+1}(x)
\end{pmatrix}, 
\end{align}
 such that  ${\bf S}_{n_k}^{\circ,I_k,*}(x)  \in H^{-1/2}(\mathbb{R})^{2n_k+6}$. 

From \cref{prop:quasimatrix-operators}, we note that the operator $\mathcal{L}_{\mu,\lambda,\eta}$ maps a sum space centred on an interval $I$, to the dual sum space centred on the same interval. This observation leads to the following proposition. 
\begin{proposition}[Multiple-interval quasimatrix operators]
\label{prop:element-mappings}
\begin{align}
\begin{split}
\Sk(x) &= \Skd(x)  \tens{E}_{\vect{n}},  \; \Hil [\Sk](x) = \Skd(x)  \tens{H}_{\vect{n}}, \\
\; \fdx [\Sk](x) &= \Skd(x)  \tens{D}^{\vect{I}}_{\vect{n}},  \; \text{and} \; (-\Delta)^{1/2} [\Sk](x) = \Skd(x)  \tens{A}^{\vect{I}}_{\vect{n}}, 
\end{split}
\end{align}
where  $\tens{E}_{\vect{n}}, \tens{H}_{\vect{n}}, \tens{D}^{\vect{I}}_{\vect{n}}, \tens{A}^{\vect{I}}_{\vect{n}} \in \mathbb{R}^{(1+2N+6K) \times (1+2N+2K)}$  are the following block diagonal matrices: 
\begin{alignat}{2}
\tens{E}_{\vect{n}} &=
\begin{pmatrix}
\tens{E}_{n_1} &       &            &\\
   & \tens{E}_{n_2}^{\dagger} &            &\\
   &       & \ddots  &\\
   &       &            & \tens{E}_{n_K}^{\dagger}\\
\end{pmatrix}, \;\;
&&\tens{H}_{\vect{n}}=
\begin{pmatrix}
\tens{H}_{n_1} &       &            &\\
   & \tens{H}_{n_2} ^{\dagger} &            &\\
   &       & \ddots  &\\
   &       &            & \tens{H}_{n_K} ^{\dagger}\\
\end{pmatrix} \label{eq:element-E-H}\\
\tens{D}^{\vect{I}}_{\vect{n}} &=
\begin{pmatrix}
\tens{D}^{I_1}_{n_1} &       &            &\\
   & \tens{D}^{I_2,\dagger}_{n_2}  &            &\\
   &       & \ddots  &\\
   &       &            & \tens{D}^{I_K,\dagger}_{n_K} \\
\end{pmatrix}, \;\;
&&\tens{A}^{\vect{I}}_{\vect{n}} =
\begin{pmatrix}
\tens{A}^{I_1}_{n_1} &       &            &\\
   & \tens{A}^{I_2,\dagger}_{n_2}&            &\\
   &       & \ddots  &\\
   &       &            & \tens{A}^{I_K,\dagger}_{n_K} \\
\end{pmatrix}, 
\end{alignat} 
where the matrices  $\tens{E}_{n_k}^\dagger, \tens{H}_{n_k}^\dagger, \tens{D}_{n_k}^{I_k, \dagger}$, and $\tens{A}_{n_k}^{I_k, \dagger}$, $k=2,\dots,K$,  are the matrices  $\tens{E}_{n_k}, \tens{H}_{n_k}, \tens{D}_{n_k}^{I_k}$, and $\tens{A}_{n_k}^{I_k}$,  respectively, as defined in \cref{prop:quasimatrix-operators}, but with the first row and column removed. 
\end{proposition}

Consider the truncated expansion of the right-hand side $f$ in the dual multiple-interval sum space, $f(x) \approx \Skd  \vectt{f}_{\vect{n}}^*$.  As in the single interval case, the matrix  $\tens{L}^{\vect{I}}_{\vect{n}} = \lambda \tens{E}_{\vect{n}}  + \mu \tens{H}_{\vect{n}}  + \eta \tens{D}^{\vect{I}}_{\vect{n}}  + \tens{A}^{\vect{I}}_{\vect{n}}$  is rectangular and the linear system to find the coefficients, $\vectt{u}_{ \vect{n} }$, of $u(x)$ in the multiple-interval sum space,
\begin{align}
 \tens{L}^{\vect{I}}_{\vect{n}}  \vectt{u}_{ \vect{n} } =  \vectt{f}^*_{ \vect{n} },  \;\;  \tens{L}^{\vect{I}}_{\vect{n}}  \in \mathbb{R}^{(1+2N+6K)\times(1+2N+2K)}
\end{align}
is overdetermined. This issue is overcome by appending four extra columns to  $\tens{L}^{\vect{I}}_{\vect{n}}$  per interval.  $\tens{L}^{\vect{I}}_{\vect{n}}$  has the block diagonal structure: 
\begin{align}
\tens{L}^{\vect{I}}_{\vect{n}} =
\begin{pmatrix}
\tens{L}^{I_1}_{n_1} &       &            &\\
   & \tens{L}^{I_2,\dagger}_{n_2} &            &\\
   &       & \ddots  &\\
   &       &            & \tens{L}^{I_K,\dagger}_{n_K}\\
\end{pmatrix},
\end{align} 
where each block has four more rows than columns. Hence, we define the matrix  $\tens{L}^{\vect{I}, +}_{\vect{n}}$  as 
\begin{align}
\tens{L}^{\vect{I}, +}_{\vect{n}} \coloneqq
\begin{pmatrix}
\tens{L}^{I_1, +}_{n_1}  &       &            &\\
   & \tens{L}^{I_2, \dagger, +}_{n_2}  &            &\\
   &       & \ddots  &\\
   &       &            & \tens{L}^{I_K, \dagger, +}_{n_K}  \\
\end{pmatrix},
\label{eq:Lmatrix1}
\end{align} 
where, for $k=2,\dots,K$,
\begin{align}
 \tens{L}^{I_1, +}_{n_1}   =
\scriptsize
\setcounter{MaxMatrixCols}{50}
\begin{pmatrix}
-\lambda & 0 & 0 & 0 & 0 & \rvline& \\
0 & 1 &0 &0 &0 & \rvline& \\
0 & 0 & 1 & 0 & 0 & \rvline& \\
0 & 0 & 0 &1 & 0 & \rvline & \\
\lambda & 0 &0 &0 &1& \rvline&  \text{\huge$ \tens{L}^{I_1,\circ}_{n_1}$}\\
0 & 0 & 0 & 0 & 0& \rvline& \\
\vdots & \vdots&\vdots&\vdots&\vdots& \rvline& \\
0 & 0 & 0 & 0 & 0& \rvline& \\
\end{pmatrix}, \;\;
 \tens{L}^{I_k,\dagger,+}_{n_k}   =
\scriptsize
\setcounter{MaxMatrixCols}{50}
\begin{pmatrix}
1 & 0 & 0 & 0 & \rvline& \\
0 &1 &0 &0 & \rvline& \\
0 &0 &1 &0 & \rvline& \\
0 & 0 &0 & 1 & \rvline &\\
0 &0 &0 &0& \rvline&  \text{\huge$ \tens{L}^{I_k,\dagger}_{n_k}$}  \\
\vdots&\vdots&\vdots&\vdots& \rvline& \\
0 & 0 & 0 & 0& \rvline& \\
\end{pmatrix},
\label{eq:Lmatrix2}
\end{align}
where the nonzero entries are placed in the rows corresponding to $V^{I_k}_0, \eU^{I_k}_{-1}, V^{I_k}_1$, and $\eU^{I_k}_{0}$, respectively. We reiterate that symbol $^\circ$ denotes the matrix without its first column and $^\dagger$ denotes the matrix without its first column and row. The matrix  $\tens{L}^{\vect{I}, +}  \in \mathbb{R}^{(1+2N+6K)\times(1+2N+6K)}$  is square and we find the coefficients $\vectt{u}^+_{ \vect{n} }$, in the appended multiple-interval sum space, of the solution $u(x)$ to \cref{eq:fpde} by solving the following system:
\begin{align}
 \tens{L}^{\vect{I}, +}_{\vect{n}}  \vectt{u}^+_{ \vect{n} } =  \vectt{f}^*_{\vect{n}}.  
\label{eq:multipe-interval-ls}
\end{align}
Since $ \tens{L}^{\vect{I}, +}_{\vect{n}} $ is block diagonal, the solve can be reduced to solving $K$ linear systems for each (much smaller) block, $k=2,\dots,K$,
\begin{align}
 \tens{L}^{I_1, +}_{n_1}  \vectt{u}^{I_1, +}_{ n_1 } =  \vectt{f}^{I_1, *}_{n_1},  \;\;  \tens{L}^{I_k, +}_{n_k}  \vectt{u}^{I_k, +}_{ n_k } =  \vectt{f}^{I_k, *}_{n_k}. \label{eq:block-decompose}
\end{align}
Moreover, the linear systems are independent and can be solved in parallel.

\section{Implementation notes}
\label{sec:implementation}

\subsection{Expansion}
\label{sec:expansion}

As the family of approximating functions is not orthogonal, the expansion of a known function, $f(x)$, in the sum space or its dual is nontrivial. In particular, the expansion of a general function need not be unique. For our purposes, we desire an expansion that approximates our known function to the required tolerance and the coefficients of the expansion are (relatively) small in magnitude.

Aside from the known function $f$ being an obvious composition of the functions in $\Sk(x)$, there are two cases to consider:
\begin{enumerate}
\itemsep=0pt
\item $f$ is compactly supported on the real line and the intervals $I_k = [a_k,b_k]$, $k =1,\dots,K$, are chosen such that either $f(x) \to 0$ or $|f(x)| \to \infty$ as $x \to a_k$ and $x \to b_k$;
\item Any function that does not fit under the first case. 
\end{enumerate}
In the first case, if $f(x) \to 0$ at the endpoints of the interval $I_k$ for some $k\in \{1,\dots,K\}$, then it might be well approximated in that interval by an expansion in $W^{I_k}_n(x)$. Similarly, if $f$ blows up at the interval endpoints it might be well represented by an expansion in $V^{I_k}_n(x)$. The advantage of only expanding $f(x)$ in $W^{I_k}_n(x)$ and $V^{I_k}_n(x)$ is that the coefficients of the expansion can be quickly computed to any given tolerance with an adaptive algorithm based on the discrete cosine transform that takes $\mathcal{O}(K n \log n)$ for $n$ coefficients and $K$ intervals \cite{Aurentz2017, Driscoll2014}. If we partially expand $f(x)$ in $V^{I_k}_n(x)$ for some $k$, but we require the expansion in the dual sum space, we utilize the identity mappings as defined in Propositions \labelcref{prop:quasimatrix-operators} and \labelcref{prop:element-mappings}. 

For functions that are not contained in the first case, we turn to the techniques used in \emph{frame} theory \cite{Adcock2019}. Essentially, we find the coefficients of the expansion that optimally interpolate the values of $f(x)$ (or a linear operator applied to $f$) at a set of collocation points in a least squares sense. Consider the collocation points $\vectt{x} = (x_1,\dots,x_M)$. We note that the set of collocation points should include points outside the intervals in order to ensure the tails also match. The least-squares matrix for the dual sum space is given by:
\begin{align}
\vectt{G}_{ij} = [l_i[\Skd]]_j, \;\; i=1\,\dots,M, \; j=1,\dots,1+2N+6K), 
\end{align}
where $\{l_i\}$ are a set of linear operators. Common choices include the identity, e.g.~ $l_i[\Skd] = \Skd(x_i)$ or linear operators designed to emulate Riemann sums $l_i[\Skd] = (x_{i+1}-x_{i-1})\Skd(x_i)$. Similarly we compute $\vectt{b}_i = l_i[f]$ and we solve the following least-squares problem for the expansion coefficients  $\vectt{f}^*_{\vect{n}}$: 
\begin{align}
\min_{\vectt{f}} \|\vectt{G} \vectt{f}^*_{\vect{n}} - \vectt{b}\|_2,
\label{eq:LS}
\end{align}
so that $f(x) \approx \Skd(x)  \vectt{f}^*_{\vect{n}}$.  The same technique can be used to find expansions in $\Sk(x)$.  It is well known that the least-squares matrix  in \cref{eq:LS}  is ill-conditioned for increasing $M$, $K$ and $n$. However, we still recover suitable least squares solutions if we sufficiently oversample the collocation points and an  $\epsilon$-truncated  SVD solver is used,  cf.~\cite[Sec.~5]{Papadopoulos2023}.  We refer to the work of Adcock and Huybrechs for further details \cite{Adcock2019,Adcock2020}.

\subsection{Approximate identity map from ${ \bf S }_n^+(x)$ to ${ \bf S }_n^*(x)$}
\label{sec:Id:Sp:Sd}

In our method, we find the approximate solution expanded in the appended sum space. A requirement in some problems is to reuse the computed solution as part of the right-hand side of the next solve. Hence, after each solve, we must map the appended sum space expansion of our current iterate to the expansion in the dual sum space. Propositions \labelcref{prop:quasimatrix-operators} and \labelcref{prop:element-mappings} provide identity mappings for the sum space. Hence, it remains to map the coefficients of the functions found in the appended sum space that are not in the sum space.

We outline the approach when utilizing a single interval at $[-1,1]$ in the solver and briefly mention the extension to multiple intervals. Given $\vectt{u}^+_{ n }$, the goal is to find the coefficient vector, $\vectt{u}^*_{ n}$, such that
\begin{align}
u(x) \approx { \bf S }_n^+(x)\vectt{u}^+_{ n } = { \bf S }_n^*(x)\vectt{u}^*_{ n }.
\end{align}
The first step is to expand the four functions $v_0(x),  \tilde{u}_{-1}(x), v_1(x)$ and $\tilde{u}_{0}(x)$ in the sum space and collect the coefficients in the vectors $\vectt{v}_{0}, \tilde{\vectt{u}}_{-1}, \vectt{v}_{1}$, and $\tilde{\vectt{u}}_{0}$, respectively. Next we form the identity mapping from ${{\bf S}}^+_n(x)$ to ${ \bf S }_n(x)$, denoted $ \tens{R}_{n} $, as follows:
\begin{align}
 \tens{R}_{n}   \coloneqq
\begin{pmatrix}
1 & \vline & \vline & \vline & \vline& 0 & \cdots  & 0\\
0 & \vectt{v}_{0}& \tilde{\vectt{u}}_{-1}& \vectt{v}_{1}& \tilde{\vectt{u}}_{0} & 1 & & \\
\vdots & \vline & \vline&\vline &\vline & & \ddots&\\
0 & \vline & \vline&\vline &\vline & & & 1
\end{pmatrix}  \in \mathbb{R}^{(2n+3)\times(2n+7)}.
\end{align}
Essentially $ \tens{R}_{n}  $ is constructed by assembling the $(2n+3)\times (2n+3)$ identity matrix and adding the four coefficient vectors, of the additional functions, in $ \tens{R}_{n}  $ between the first and second columns. The approximate identity operator $ \tens{B}_{n}  \in \mathbb{R}^{(2n+7) \times (2n+7)}$ such that ${ \bf S}_n^+(x)  \approx {{\bf S}}^*_n(x)  \tens{B}_{n}  $ is defined as
\begin{align}
 \tens{B}_{n}   \coloneqq  \tens{E}_{n}   \tens{R}_{n} \in \mathbb{R}^{(2n+7) \times (2n+7)}  ,
\end{align} 
where $ \tens{E}_{n}  $ is the identity operator from ${ \bf S}_n(x)$ to ${{\bf S}}^*_n(x)$ as defined in \cref{prop:quasimatrix-operators}. 

For multiple intervals, we denote the identity mapping as $ \tens{R}^{\vect{I}}_{\vect{n}}  \in \mathbb{R}^{(1+2N+2K)\times(1+2N+6K)}$, $ \tens{R}^{\vect{I}}_{\vect{n}} : \Skp(x) \to \Sk(x)$, where $K$ is the number of intervals and consists of the identity matrix with an additional $4K$ dense columns in the positions of the additional functions in the appended sum space. As before, we define $ \tens{B}^{\vect{I}}_{\vect{n}} \coloneqq  \tens{E}_{\vect{n}} \tens{R}^{\vect{I}}_{\vect{n}}$  so that $\Skp(x) \approx \Skd(x)  \tens{B}^{\vect{I}}_{\vect{n}}$.

\subsection{Numerical conditioning}
\label{sec:conditioning}

Two potential sources of numerical ill-conditioning in \cref{alg:spectral-method} are:
\begin{enumerate}
\itemsep=0pt
\item The expansion of the right-hand side $f$ to find a coefficient vector ${\bf f}^{\vect{I}, *}_{\vect{n}}$;
\item Inverting the matrix $\tens{L}_{\vect{n}}^{\vect{I}, +}$ to find ${\bf u}^{\vect{I}, +}_{\vect{n}}$.
\end{enumerate}
In \cref{cor:dual-frame}, we prove that the dual sum space is a frame. Leveraging recent analysis on frames (which are overdetermined approximation spaces) it is known that an $\epsilon$-truncated SVD factorization alleviates the ill-conditioning that one might expect in solving such a least-squares problem. Hence, the expansion of the right-hand side is an understood and well-conditioned problem, resolving (1). 

For (2), we note that the 2-norm condition number $\kappa(\tens{L}_{\vect{n}}^{\vect{I}, +})$ only grows linearly with $n_{\text{max}} = \max(n_1, \dots, n_K)$ and is independent on the number of intervals in the approximation space $K$ due the block diagonal structure of $\tens{L}_{\vect{n}}^{\vect{I}, +}$. Moreover, there exists a \emph{diagonal} right preconditioner $\tens{P}_{\vect{n}}^{\vect{I}}$ such that $\kappa(\tens{L}_{\vect{n}}^{\vect{I}, +} (\tens{P}_{\vect{n}}^{\vect{I}})^{-1}) = C$ where $C$ is independent of $\vect{n}$. This is typical of a coefficient-based spectral method, cf.~\cite[Sec.~4.1]{Olver2013}.

\begin{proposition}
\label{prop:condition-number} 
Consider the interval $I=[-1,1]$, $\lambda = \mu = \eta = 1$, and choose the appended sum space functions $\tilde{u}_{-1}$, $\tilde{u}_{n+1}$, $v_0$, and $v_{n+2}$ as defined in \cref{sec:single-interval}. Then
\begin{align}
\kappa(\tens{L}^+_{n}) = \mathcal{O}(n),
\end{align}
where $\kappa(\cdot)$ denotes the 2-norm condition number of a matrix. Moreover, consider the diagonal matrix
\begin{align}
\tens{P}_{n} = \begin{pmatrix}
\tens{I}_{5} &\\
& \tens{D}_n
\end{pmatrix},\;\; \text{where} \;\;
\tens{D}_n =
\begin{pmatrix}
1 &&&&&&\\
& 1 &&&&&\\
&& 2 &&&&\\
&&&2 &&&\\
&&&&\ddots &&\\
&&&&&n+1 &\\
&&&&&&n+1
\end{pmatrix} \in \mathbb{R}^{(2n+2) \times (2n+2)},
\end{align}
and $\tens{I}_5$ is the $5 \times 5$ identity matrix. Then, there exists a constant $C>0$, such that $C$ is independent of $n$ and
\begin{align}
\kappa(\tens{L}^+_{n} (\tens{P}_{n})^{-1}) = C.
\label{eq:condition1}
\end{align}
\end{proposition}
\begin{proof}
The result follows by showing that $\tens{L}^+_{n} (\tens{P}_{n})^{-1} = \tens{I} + \tens{K}_n$ where $\tens{K}_n$ is a compact operator and a very similar result is proven in \cite[Lem.~4.3]{Olver2013}.
\end{proof}
\begin{remark}
Numerically we found that $C \approx 3.1$ where $C$ is the constant in \cref{eq:condition1}.
\end{remark}

A diagonal preconditioner may be constructed for the multiple-interval case as follows:
\begin{align}
\tens{P}_{\vect{n}} = 
\begin{pmatrix}
\tens{P}_{n_1} & & &\\
& \tens{P}^\dagger_{n_2} &&\\
&&  \ddots & \\
& & &\tens{P}^{\dagger}_{n_K}
\end{pmatrix},
\end{align}
where $\tens{P}^{\dagger}_n$ denotes $\tens{P}_n$ with the first row and column removed. It is a direct corollary of \cref{prop:condition-number} that there exists a $C>0$, independent of $\vect{n}$, such that $\kappa( \tens{L}_{\vect{n}}^{\vect{I}, +}  (\tens{P}_{\vect{n}})^{-1}) = C$.

\section{Numerical analysis}
\label{sec:analysis}
The goal of this section is to prove that the dual sum space is a frame for the right-hand side $f$ and, consequently, derive a convergence result for the discretization. These results motivate the favourable approximation properties that we observe in our numerical experiments in the next section.

\begin{definition}[Hilbert space $H^s_w(\mathbb{R})$]
\label{def:Hsw}
Let $H^s_w(\mathbb{R})$ denote the Hilbert space 
\begin{align}
H^s_w(\mathbb{R}) \coloneqq
\{ u \in L^2_w(\mathbb{R}) : \supp(u) \subseteq \supp(w), \; (-\Delta)^{s/2} u \in L^2(\mathbb{R})\},
\end{align}
equipped with the inner-product $( u, v )_{H^s_w(\mathbb{R})} \coloneqq ( u, v )_{L^2_w(\mathbb{R})} +  ((-\Delta)^{s/2} u,  (-\Delta)^{s/2} v )_{L^2(\mathbb{R})}$. Here $\supp$ denotes the support of a function.
\end{definition}

\begin{lemma}[Orthogonality, Lem.~5.1 in \cite{Papadopoulos2023}]
Let $w = (1-x^2)^{-1/2}_+$. Then the family of extended Chebyshev functions $\{\eT_n\}_{n \in \mathbb{N}}$ is orthogonal with respect to the $H^{1/2}_w(\mathbb{R})$ inner product. This also holds for  $\{W_n\}_{n \in \mathbb{N}_0}$.
\end{lemma}

\begin{definition}
Let $w = (1-x^2)^{-1/2}_+$. Define $\{\hT_n\}_{n \in \mathbb{N}}$ and $\{\hW_n\}_{n \in \mathbb{N}_0}$ as the orthonormalized families of functions $\{\eT_n\}_{n \in \mathbb{N}}$ and $\{W_n\}_{n \in \mathbb{N}_0}$, respectively, with respect to the $H^{1/2}_w(\mathbb{R})$-inner product. Moreover, let $\hat{V}_n(x) \coloneqq (-\Delta)^{1/2} \hT_n(x)$, $n \in \mathbb{N}$ and  $\hat{U}_n(x) \coloneqq (-\Delta)^{1/2} \hW_n(x)$, $n \in \mathbb{N}_0$.
\end{definition}

Let $\hat{S}$ and $\hat{S}^{*}$ denote the orthonormalized sum and dual sum spaces:
\begin{align}
\hat{S} &\coloneqq \mathrm{span}(\{\hT_n, \hW_m : n \in \mathbb{N}, m \in \mathbb{N}_0\}),\\
\hat{S}^{*} & \coloneqq \mathrm{span}(\{\eU_{-2}, \eU_{-1}, V_0, \hat{U}_n, \hat{V}_m: n \in \mathbb{N}_0, m \in \mathbb{N}\}).
\end{align}

\begin{theorem}[Frame on $H^{1/2}_w(\mathbb{R}) \cap C(\mathbb{R})$]
\label{prop:mul-frame}
Let $w(x) = (1-x^2)^{-1/2}_+$. Then the $H^{1/2}_w(\mathbb{R})$-orthonormalized sum space $\hat{S}$ is a frame on $H^{1/2}_w(\mathbb{R}) \cap C(\mathbb{R})$.
\end{theorem}
\begin{proof}
The result follows, line-by-line, from the proof of \cite[Th.~5.2]{Papadopoulos2023}.
\end{proof}

\begin{corollary}
\label{cor:dual-frame}
Let  $w(x) = (1-x^2)^{-1/2}_+$. Then the dual space $\hat{S}^{*}$ is a frame on $(H^{1/2}_w(\mathbb{R}) \cap C(\mathbb{R}))^*$.
\end{corollary}

\begin{proof}
From \cref{prop:mul-frame} and an application of Theorem 5.1 in \cite{Papadopoulos2023}, we may deduce that $\{\hat{V}_n\}_{n=1}^\infty \cup \{\hat{U}_n\}_{n=0}^\infty$ is a frame on $(H^{1/2}_w(\mathbb{R}) \cap C(\mathbb{R}))^*$. 

It remains to prove the frame bound condition for the remaining dual sum space functions $\eU_{-2}(x)$ and $\eU_{-1}(x)$, and $V_0(x)$. The lower bound follows trivially. Moreover, it may be shown that $\eU_{-2}, \eU_{-1}, V_0 \in (H^{1/2}_w(\mathbb{R}))^*$. Hence, the upper bound also follows.
\end{proof}

\begin{theorem}
\label{th:convergence}
Consider the operator $\mathcal{L}_{\mu,\lambda,\eta}$ where $\mu, \lambda >0$ and $\eta = 0$. Let the weight $w(x) = (1-x^2)^{-1/2}_+$. Suppose that $f \in H^*  = (H^{1/2}_w(\mathbb{R}) \cap C(\mathbb{R}))^*$. Consider the degree $n$. Suppose that an $\epsilon$-truncated SVD factorization is utilized to discover a coefficient vector ${\bf f}^{*}_{n}$ such that $f(x) \approx \hat{{\bf S}}^*_n(x) {\bf f}^{*}_{n}$ and fix $\hat{\tens{L}}^{+}_{n} {\bf u}^{+}_{n} = {\bf f}^{*}_{n}$. Here $\hat{\tens{L}}^{+}_{n}$ is the scaled matrix ${\tens{L}}^{+}_{n}$ with respect to the orthonormalized sum space. Then, there exist constants $C, \kappa^\epsilon_{M,n}, \theta^{\epsilon}_{M,n}>0$ such that
\begin{align}
\begin{split}
&\|u - \hat{{\bf S}}^+_n {\bf u}^{+}_{n} \|_{H^{1/2}(\mathbb{R})} \\
& \leq C \inf_{{\bf v} \in \mathbb{R}^{n+7}} 
\left\{
\| f -  \hat{{\bf S}}^*_n {\bf v} \|_{H^*} + \kappa_{M, \vect{n}}^\epsilon \| f - \hat{{\bf S}}^*_n {\bf v}  \|_M + \epsilon \theta^\epsilon_{M,\vect{n}} \| {\bf v} \|_{\ell^2}
 \right\},
\end{split}
\end{align}
where $\| u \|_M \coloneqq \sum_{i=1}^M |l_i[u]|^2$ and $l_i$, $i=1,\dots,M$ are defined in \cref{sec:expansion}.
\end{theorem}
\begin{proof}
Note that
\begin{align}
\begin{split}
&\|u - \hat{{\bf S}}^+_n {\bf u}^{+}_{n} \|_{H^{1/2}(\mathbb{R})} 
= \|\mathcal{L}_{\mu, \lambda, 0}^{-1} \mathcal{L}_{\mu, \lambda, 0} (u - \hat{{\bf S}}^+_n {\bf u}^{+}_{n}) \|_{H^{1/2}(\mathbb{R})}\\
&\indent= \|\mathcal{L}_{\mu, \lambda, 0}^{-1} (f- \hat{{\bf S}}^*_n \hat{\tens{L}}^{+}_{n} {\bf u}^{+}_{n}) \|_{H^{1/2}(\mathbb{R})}\\
& \indent= \|\mathcal{L}_{\mu, \lambda, 0}^{-1} (f- \hat{{\bf S}}^*_n  {\bf f}^{*}_{n} ) \|_{H^{1/2}(\mathbb{R})}\\
&\indent \leq \|\mathcal{L}_{\mu, \lambda, 0}^{-1} \|_{\mathcal{B}(H^*, H^{1/2}(\mathbb{R}))} \|f- \hat{{\bf S}}^*_n {\bf f}^{*}_{n} \|_{H^*}.
\end{split}
\end{align}
Since $H^* \subset H^{-1/2}(\mathbb{R})$, we have that there exists a $C>0$ such that $\|\mathcal{L}_{\mu, \lambda, 0}^{-1} \|_{\mathcal{B}(H^*, H^{1/2}(\mathbb{R}))} \leq C$. \cref{cor:dual-frame}  and an application of Theorem 3.7 in \cite{Adcock2020} achieves the result.
\end{proof}

\begin{remark}
The constants $\theta_{M, n}^\epsilon$ and $\kappa_{M, n}^\epsilon$ are problem and frame dependent. Ideally their magnitude is $\mathcal{O}(1)$, in which case we are guaranteed to reach an accuracy of $\epsilon$ if $M$ and $n$ are taken sufficiently large. For a thorough discussion of the constants and their behaviour, we refer the reader to \cite{Adcock2020}.
\end{remark}

A full convergence analysis for $\lambda, \mu, \eta \in \mathbb{R}$, $f \in H^{-1/2}(\mathbb{R})$, and sum spaces consisting of multiple intervals is nontrivial and beyond the scope of this work.

\section{Numerical examples}
\label{sec:examples}
In this section we provide several numerical examples. During numerical experiments, we found that approximating with a sum space centred at a single interval was normally insufficient for satisfying the condition $u(x) \to \infty$ as $|x| \to \infty$. However, this was resolved when utilizing a combination of sum spaces centred at multiple intervals. In particular, we found that five intervals worked well in all examples. Unless the initial condition or right-hand side are represented exactly by the (dual) sum space, we find the coefficient vectors via the truncated least-squares matrix as explained in \cref{sec:expansion}.

\textbf{Code availability:} For reproducibility, an implementation of the spectral method as well as scripts to generate the plots and solutions can be found at \url{https://github.com/ioannisPApapadopoulos/SumSpaces.jl}. The version of the software used in this paper is archived on Zenodo \cite{SumSpacesMathLink, SumSpaces, SumSpacesv2}.  Any timings provided were achieved on computer with a 2.90GHz Intel(R) Core(TM) i7-10700 CPU and 16GB of RAM. 

\subsection{Manufactured solutions}
In order to test the spectral convergence of our method, our first example is constructed via the method of manufactured solutions. We fix the exact solution $u(x) = \me^{-x^2}$. Note that $\fdx \me^{-x^2} = -2x\me^{-x^2}$. Moreover \cite[Prop.~4.2]{Sheng2020},
\begin{align}
(-\Delta)^{1/2} u(x) = \frac{2}{\sqrt{\pi}} \oFo(1;1/2;-x^2), \label{eq:ex2-1}
\end{align}
where $\oFo$ denotes the Kummer confluent hypergeometric function \cite[Sec.~16.2]{OlverNIST}.  Note that, by appealing to the Maclaurin series expansions of $_{1}F_1$ and $\mathrm{erf}$, it can be shown that  $\frac{2}{\sqrt{\pi}} {_{1}F_1}(1;1/2;-x^2) = \frac{2}{\sqrt{\pi}} + 2 \I \E^{-x^2} |x| \mathrm{erf}(\I|x|)$ where $\mathrm{erf}(z) \coloneqq \frac{2}{\sqrt{\pi}} \int_0^z \me^{-t^2} \mathrm{d}t$. By utilizing the identity
\begin{align}
\fdx [- \I \, \mathrm{erf}(\I|x|) ]= \frac{2}{\sqrt{\pi} |x|} x \E^{x^2},
\end{align}
we deduce that
\begin{align}
\fdx \left[-\frac{\I}{x} \left(\E^{-x^2} |x| \, \mathrm{erf}(\I |x|) \right) \right] = 2 + 2 \I \E^{-x^2} |x| \mathrm{erf}(\I|x|).
\end{align}
Hence
\begin{align}
\Hil[u](x) = -\frac{\mi}{x} \left(\me^{-x^2} |x| \, \mathrm{erf}(\mi |x|) \right).
\end{align} 
In the first case we set $\lambda = 1$, $\mu = \eta = 0$:
\begin{align}
(\mathcal{I} + (-\Delta)^{1/2})u(x) = \me^{-x^2} +  \frac{2}{\sqrt{\pi}} \oFo(1;1/2;-x^2).  \label{eq:ex2:case1}
\end{align}
In the second case we have $\lambda = \mu = \eta = 1$:
\begin{align}
\begin{split}
&(\mathcal{I} + \fdx + \Hil + (-\Delta)^{1/2})u(x) \\
& \indent = (1-2x)\me^{-x^2} -\frac{\mi}{x} \left(\me^{-x^2} |x| \, \mathrm{erf}(\mi |x|) \right) 
+ \frac{2}{\sqrt{\pi}}\oFo(1;1/2;-x^2).  \label{eq:ex2:case2}
\end{split}
\end{align}
For this example we fix five intervals at $[-5,-3]$, $[-3,-1]$, $[-1,1]$, $[1,3]$, and $[3,5]$. In the least-squares expansion we choose 6001 equally spaced points on each interval as well as 6001 equally spaced points between $[-25,-5]$ and $[5,25]$. This results in 42,001 unique collocation points. We choose the appended sum space functions $v_0(x)$, $\tilde{u}_{-1}(x)$, $v_{n+2}(x)$, $\tilde{u}_{n+1}(x)$ centred on each interval. As discussed in \cref{rem:stabiliseL}  and proven in \cref{prop:condition-number},  this improves the conditioning of the induced linear systems and helps to ensure that the convergence rate is not polluted by numerical instabilities due to ill-conditioning.  The additional functions are computed by approximating the continuous inverse Fourier transforms via Mathematica's \texttt{NIntegrate} routine \cite{mathematica}.  We evaluate the $_{1}F_1(1;1/2;-x^2)$ function via a parameter change $_{1}F_1(a,b,-x^2) = \mathrm{e}^{-x^2} {}_{1}F_1(b-a, b, x^2)$ and then an evaluation of its Maclaurin series \cite[Sec.~13.2.2]{OlverNIST} as implemented in \cite[v.~0.3.10]{HypergeometricFunctions.jl}.

In \cref{fig:ex2-2}, we provide spy plots of the induced subblocks of $\vectt{L}^+$.  We see that the matrices are sparse and almost banded. A semi-log plot of the convergence for both cases is depicted in \cref{fig:ex2-1}. The error is measured in the $l^\infty$-norm as measured on a 1001-point equally spaced grid at $[-5,5]$, i.e.
\begin{align}
\max_{x \in \{-5,-4.99,\cdots,4.99,5\}} | u(x) - \Sk(x) \vectt{u}|.
\label{eq:man-error}
\end{align}
 When $n=100$ on each interval and the right-hand side \cref{eq:ex2:case2}, the least-squares matrix $\tens{G} \in \mathbb{R}^{42,001 \times 1031}$ in line 1 of \cref{alg:spectral-method} was assembled in 8.16 seconds and the least-squares solve was computed in 8.66 seconds, via an $\epsilon$-truncated SVD factorization to find $\vect{f}^{\vect{I},*}_{\vect{n}} \in \mathbb{R}^{1031}$. The matrix $\tens{L}^{\vect{I}, +}_{\vect{n}}$ was assembled in 0.0944 seconds. The linear system solve to find $\vect{u}^{\vect{I}, +}_{\vect{n}}$ took 0.0128 seconds.  Note that $n=100$ on each interval is excessive since the convergence is spectral. the error stagnates around $10^{-14}$ in \cref{fig:ex2-1a} and $10^{-13}$ in \cref{fig:ex2-1b} by $n=21$ as the error of the approximation of the right-hand side stagnates at the same magnitude. 
\begin{figure}[h!]
\centering
\includegraphics[width =0.49 \textwidth]{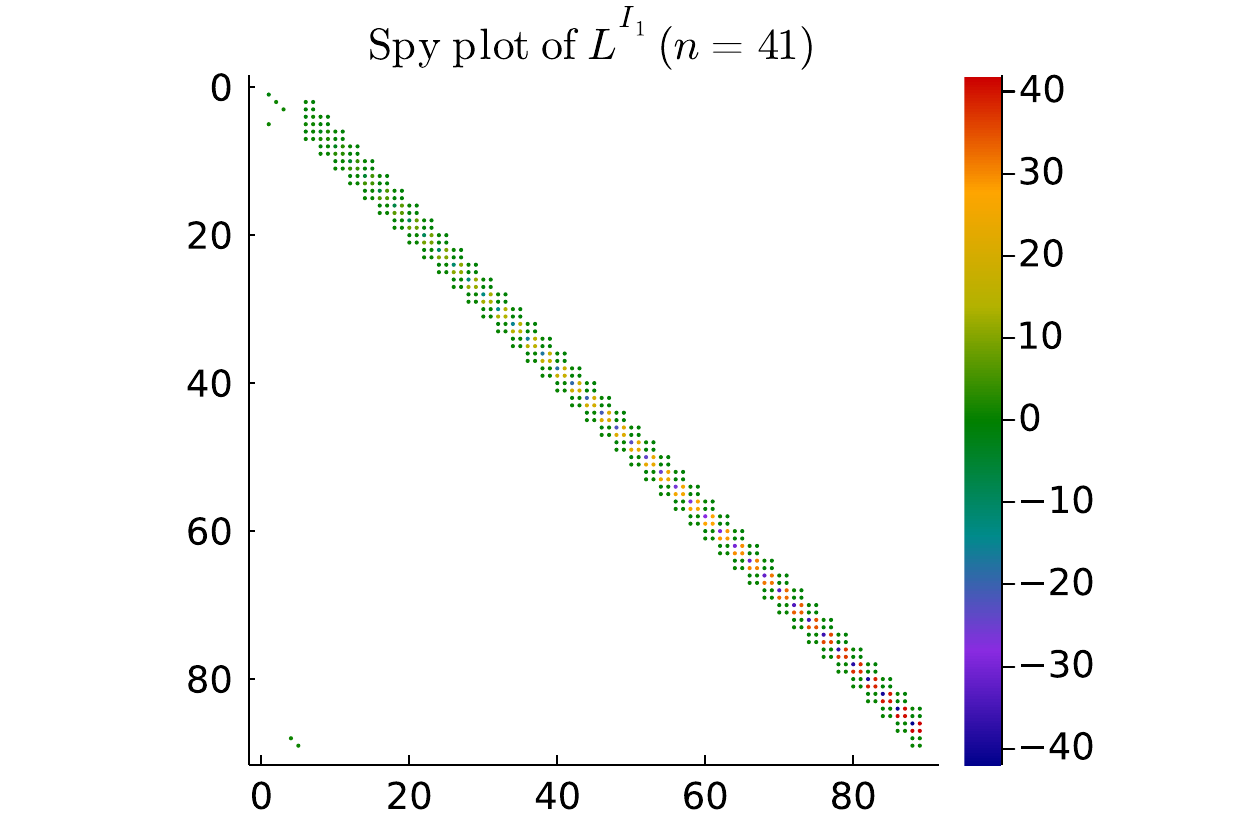}
\includegraphics[width =0.49 \textwidth]{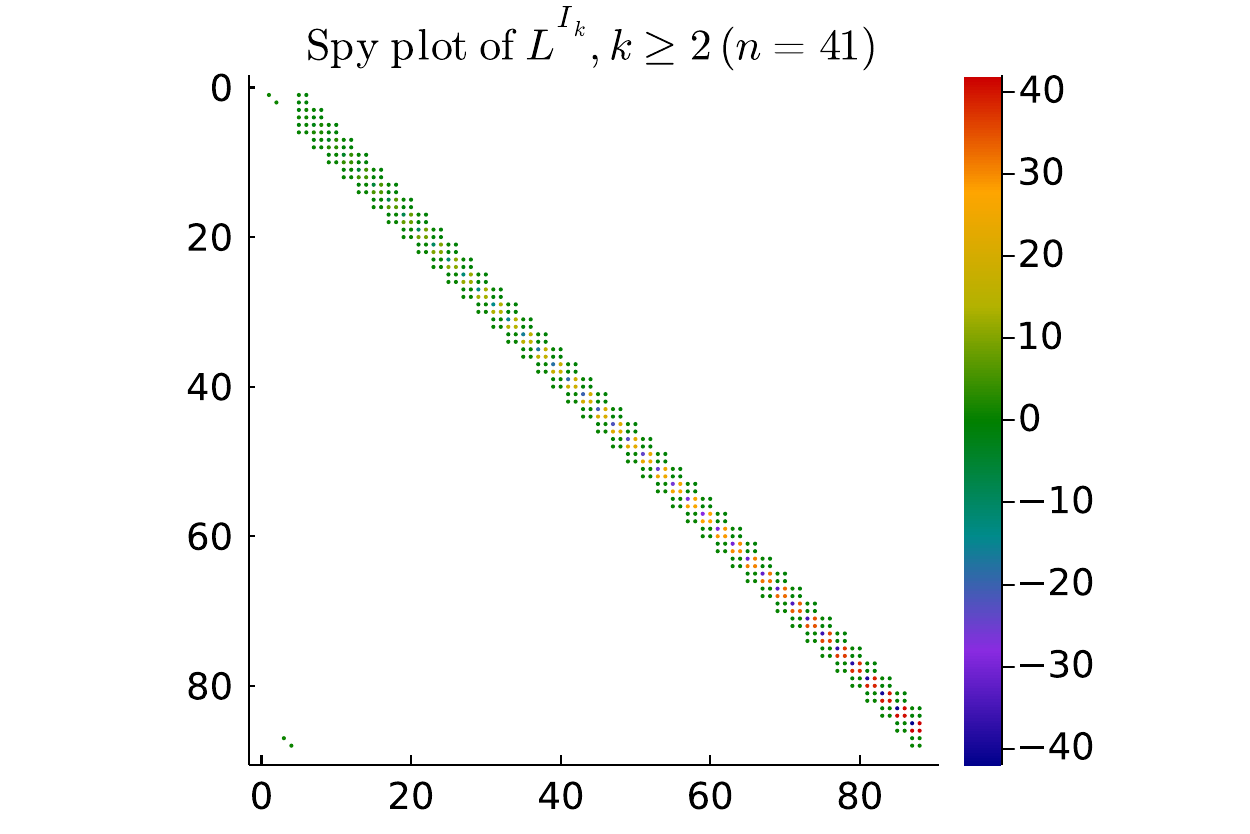}
\caption{\texttt{Spy} plots of the matrices in the linear systems after discretizing and decoupling \cref{eq:ex2:case2} interval-wise. Here $\lambda = \eta = \mu = 1$ and we use the additional functions $v_0$, $\tilde{u}_{-1}$,  $v_{n+1}$, $\tilde{u}_{n+1}$. The matrices are sparse and almost banded.}\label{fig:ex2-2}
\end{figure}
\begin{figure}[h!]
\centering
\subfloat[Error for \cref{eq:ex2:case1}]{\includegraphics[width =0.49 \textwidth]{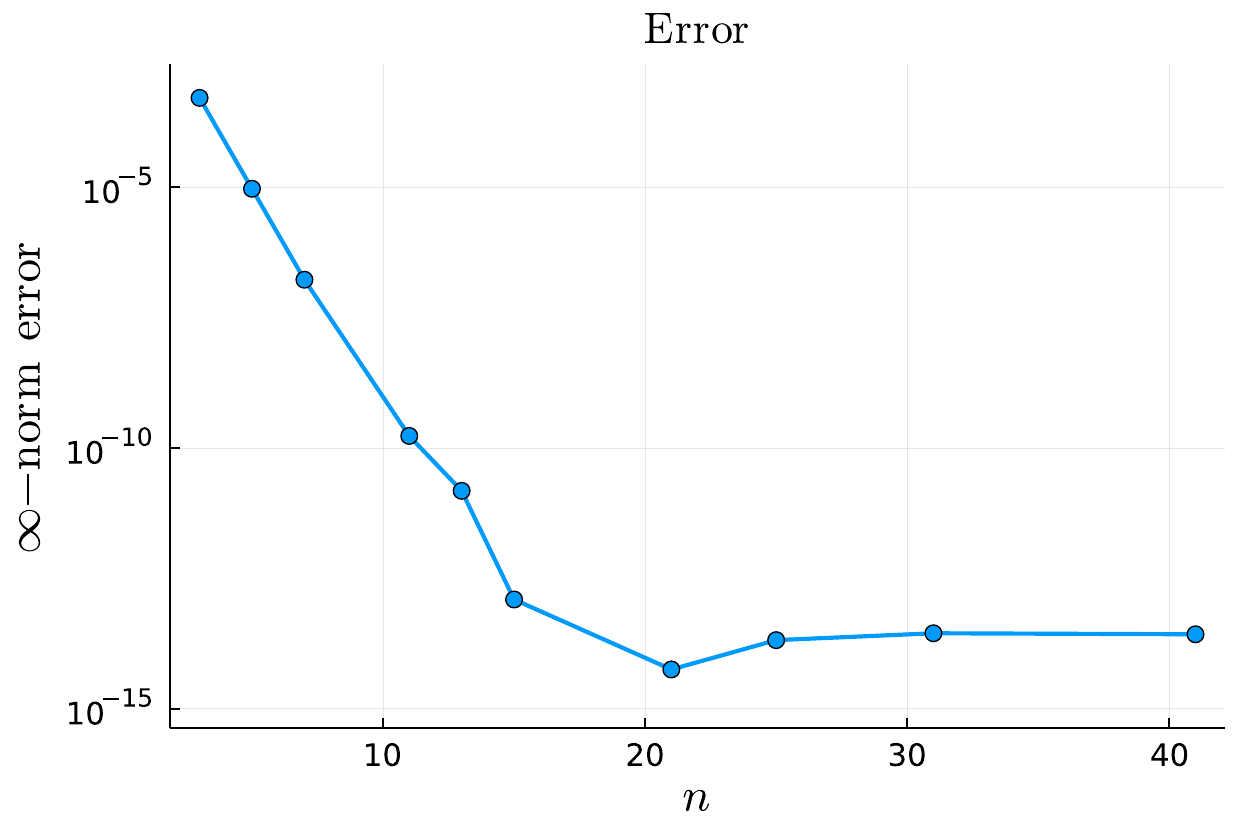} \label{fig:ex2-1a}}
\subfloat[Error for \cref{eq:ex2:case2}]{\includegraphics[width =0.49 \textwidth]{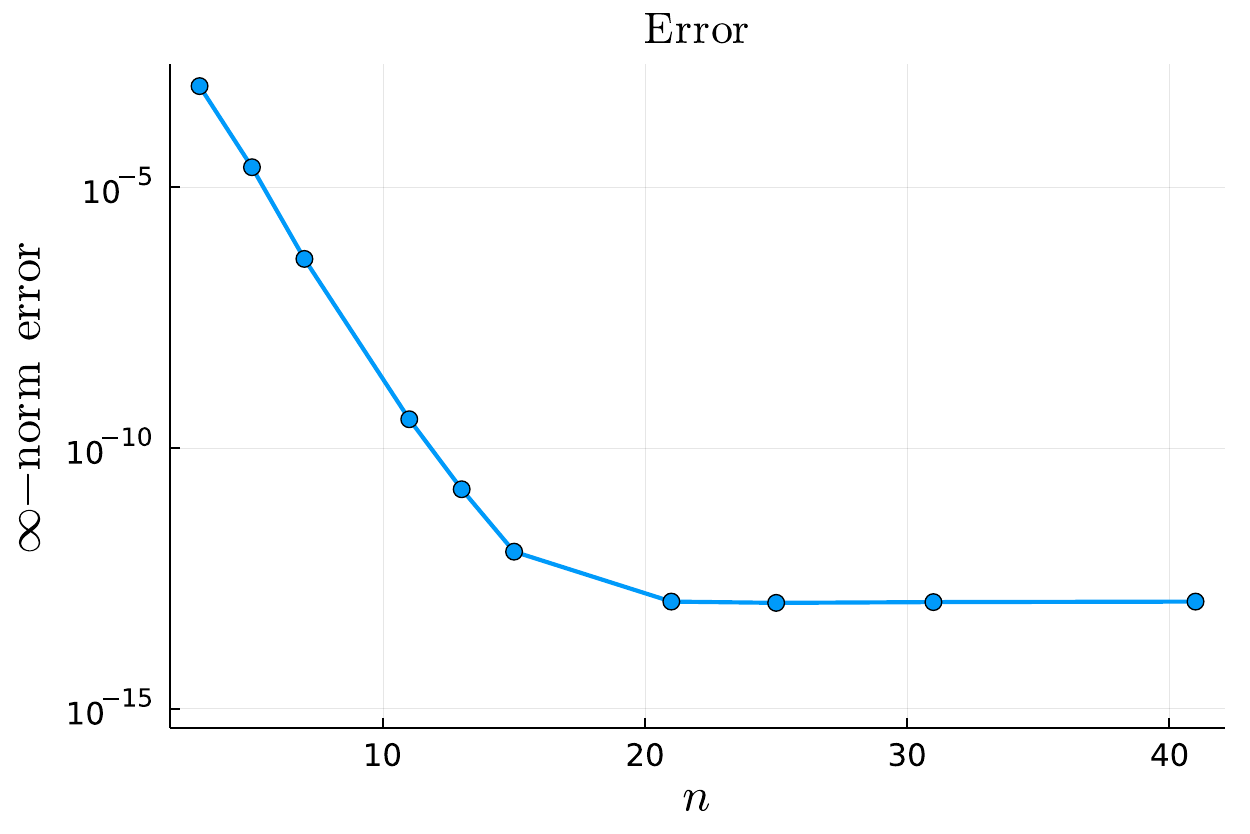} \label{fig:ex2-1b}}
\caption{Error in the numerical solution $u(x)$ as measured by \cref{eq:man-error} with the right-hand sides \cref{eq:ex2:case1} and \cref{eq:ex2:case2} for increasing truncation degree $n$.  The approximation space contains $10n +31$ functions for each value of $n$.  The convergence is spectral.}\label{fig:ex2-1}
\end{figure}

Since we are not guaranteed that the sum space is a frame on all $\mathbb{R}$ (\cref{prop:mul-frame} only guarantees it is a frame on the user-chosen intervals), in \cref{fig:manfactured-solution-coeffs} we investigate the behaviour of the expansion coefficients of the right-hand side. In particular we plot the $l^\infty$-norm of the coefficient vector $\vectt{f}$ of the right-hand sides \cref{eq:ex2:case1} and \cref{eq:ex2:case2} for increasing truncation degree $n$. Despite the lack of a strict frame condition, we achieve bounded coefficients of magnitude $\mathcal{O}(1)$ for all values of $n$.  

\begin{figure}[h!]
\centering
\subfloat[Coefficient norm for \cref{eq:ex2:case1}]{\includegraphics[width =0.49 \textwidth]{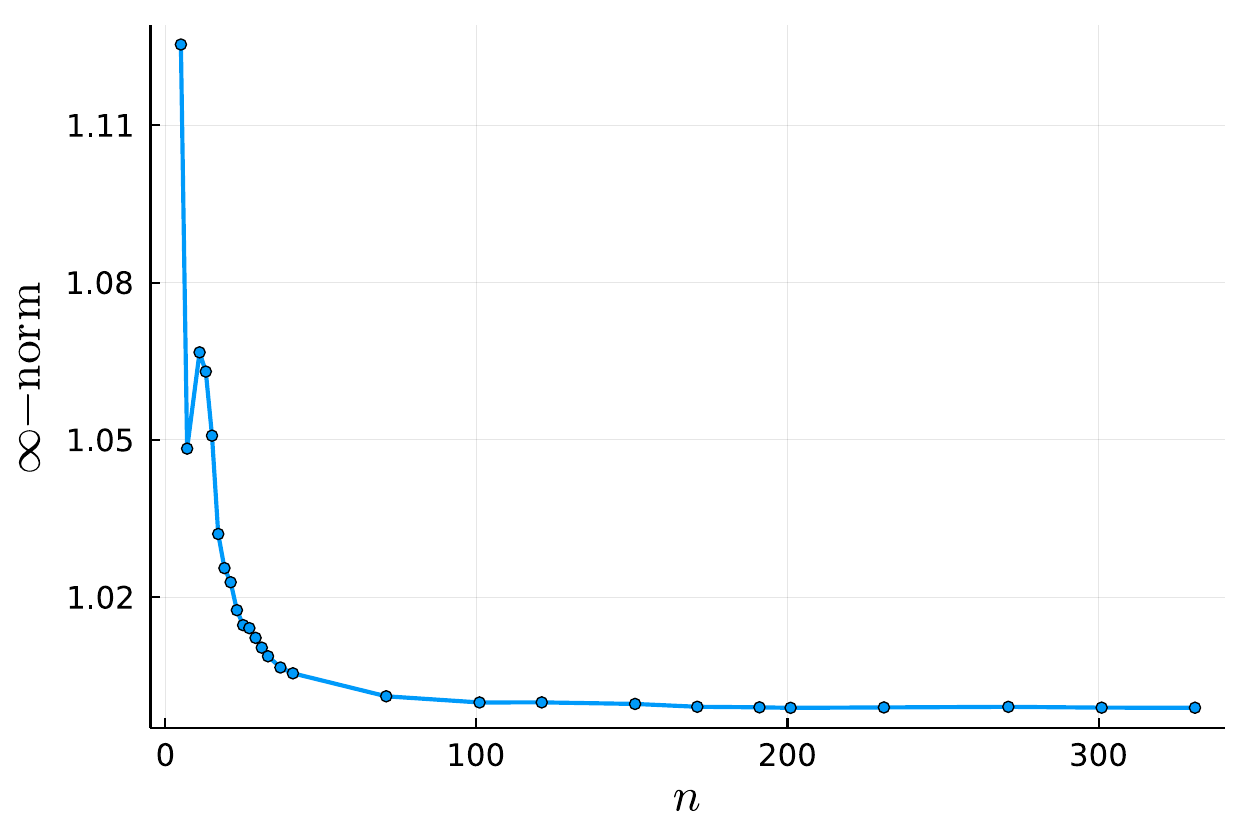}}
\subfloat[Coefficient norm for \cref{eq:ex2:case2}]{\includegraphics[width =0.49 \textwidth]{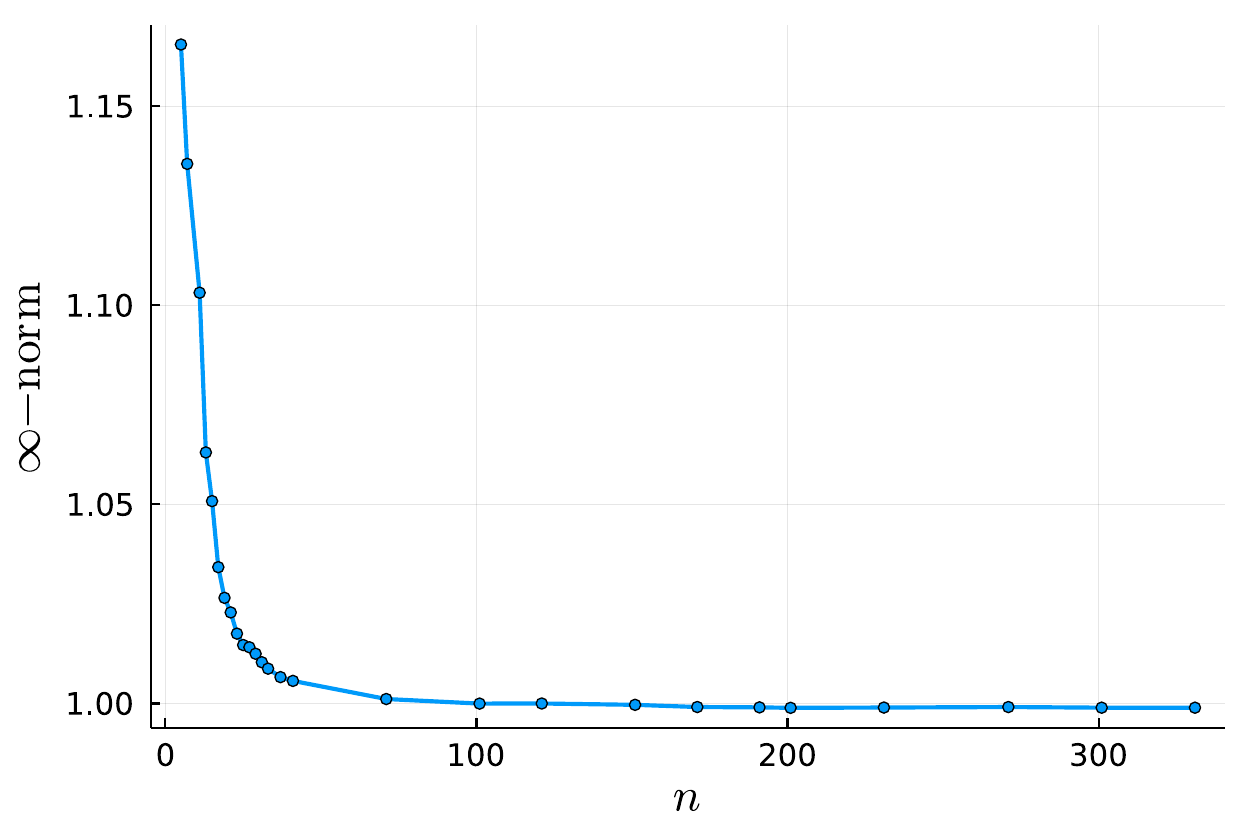}}
\caption{The $l^\infty$-norm of the coefficient vector of the sum space expansion of the right-hand side for increasing truncation degree $n$.  The approximation space contains $10n +31$ functions for each value of $n$. }\label{fig:manfactured-solution-coeffs}
\end{figure}

\subsection{Discontinuous right-hand side}

This example examines the behaviour of the (dual) sum space approximation of a problem with discontinuous data. We seek $u \in H^{1/2}(\mathbb{R})$ that satisfies
\begin{align}
(\mathcal{I} + (-\Delta)^{1/2})u(x) = f(x), \;\;\; \text{where} \;\;
f(x) \coloneqq 
\begin{cases}
1 & |x| <1,\\
0 & |x| \geq 1. 
\end{cases} \label{eq:ex3-1}
\end{align}
We use a multiple-interval (dual) sum space centred at the intervals $[-5,-3]$, $[-3,-1]$, $[-1,1]$, $[1,3]$, and $[3,5]$. In the sum space expansion we choose 6001 equally spaced points on each interval as well as 6001 equally spaced points between $[-10,-5]$ and $[5,10]$. This results in 42,001 unique collocation points. As the multiple-interval dual sum space contains functions that are undefined at $x= -5,-3,-1,1,3$ and 5, we instead choose 6001 equally spaced points in $[a+\epsilon, b-\epsilon]$, $\epsilon = 10^{-2}$, where $a, b$ represent the endpoints of each interval as well as 6001 equally spaced points in $[-10+\epsilon,-5-\epsilon]$ and $[5+\epsilon,10-\epsilon]$. This results in 42,007 collocation points. 

We first check the convergence of the approximation of the discontinuous right-hand side $f$. We measure the error in the sum space approximation via \cref{eq:man-error}. We exclude the points $x= -5,-3,-1,1,3$ and 5, when measuring the error in the dual sum space expansion, i.e. we compute:
\begin{align}
\max_{x \in \{-5,-4.99,\cdots,4.99,5\} \backslash \{-5,-3,-1,1,3,5\}} | u(x) - \Skd(x) \vectt{u}|.
\label{eq:dual-man-error}
\end{align}
In \cref{fig:ex3-1}, we plot the sum space approximated right-hand side as well as the error plot with increasing truncation degree of both the sum space and dual sum space approximations.  At truncation degree $n=101$, the sum space error is of the order $\mathcal{O}(10^{-6})$. Whereas the dual sum space approximation reaches an error of $\mathcal{O}(10^{-13})$. Hence, the best approximation of this particular discontinuous right-hand side is achieved via a direct dual sum space expansion rather than a sum space expansion coupled with the identity conversion operator $\vectt{E}$ (as defined in \cref{eq:element-E-H}) to re-expand in the dual sum space. 

The numerical solution is plotted in \cref{fig:ex3-2} as well as an approximation of the convergence. The appended sum space functions are $v_0(x)$, $\tilde{u}_{-1}(x)$, $v_{n+2}(x)$, and $\tilde{u}_{n+1}(x)$. Mathematica's \texttt{NIntegrate} routine is used to compute the necessary inverse Fourier transforms. As we do not have an explicit solution, we cannot measure the error directly. Instead, we truncate the right-hand side expansion at degrees $n_f =11, 15, 21$, and 27. We then measure the 2-norm difference in the sum space coefficient vectors at truncation degree $n$ and $n-2$ of the solution for $n = n_f+2, n_f+4, \dots, n_f + 16$. We observe spectral convergence in the coefficient vectors. 
\begin{figure}[h!]
\centering
\includegraphics[width =0.49 \textwidth]{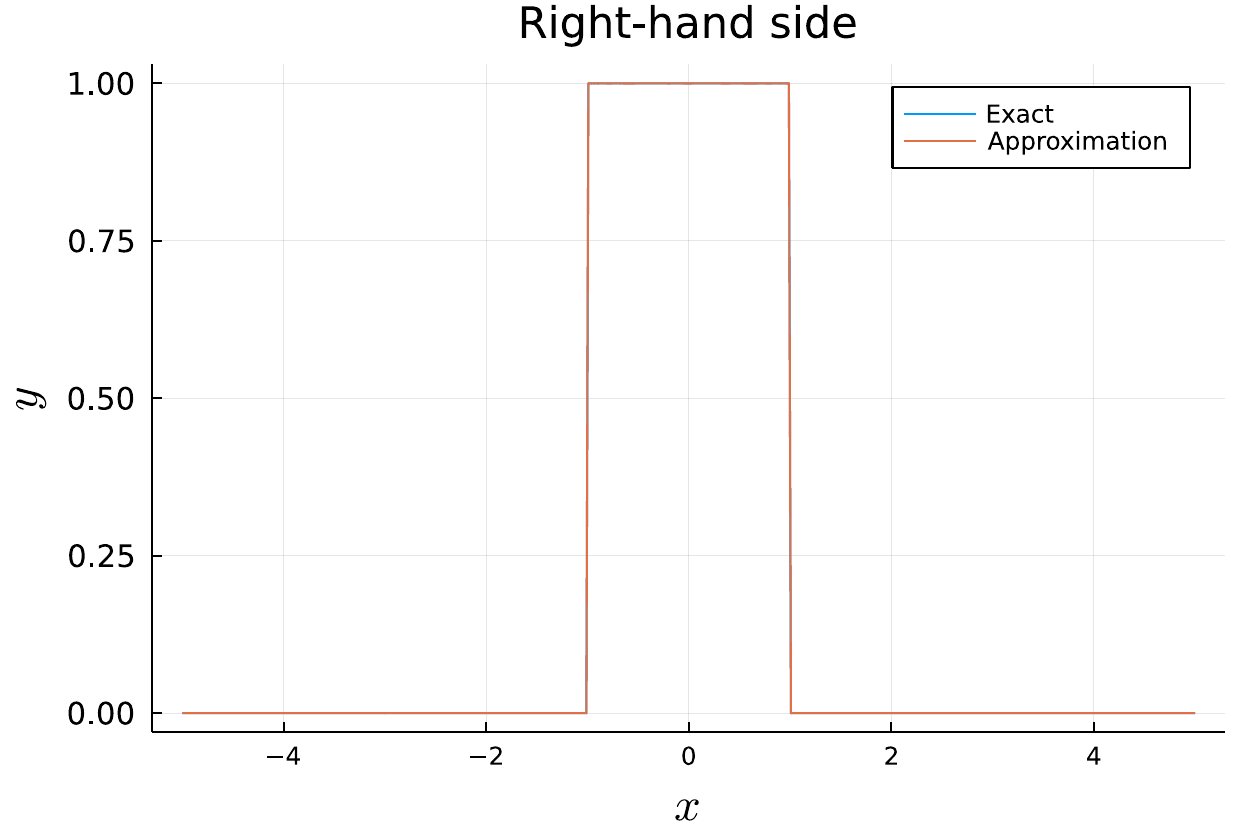}
\includegraphics[width =0.49 \textwidth]{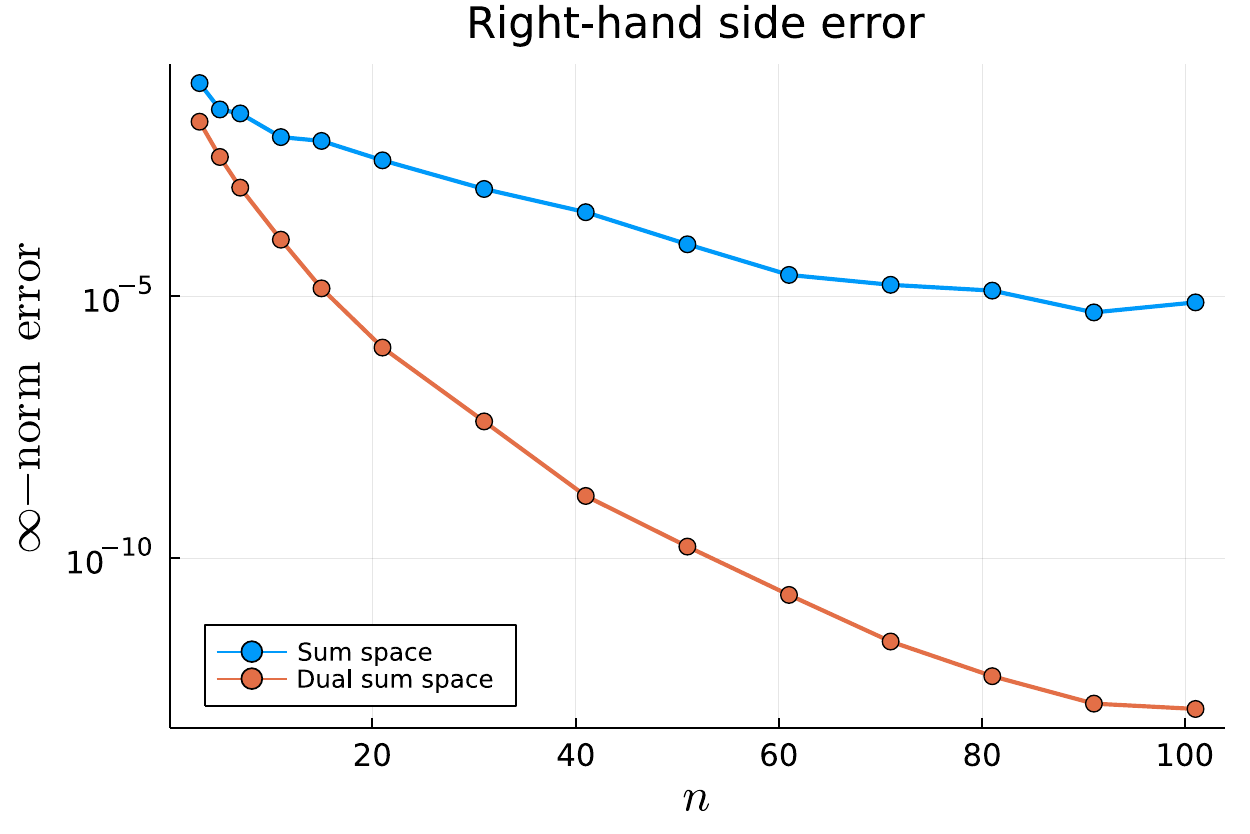}
\caption{The exact and sum space approximated $f(x)$ as defined in \cref{eq:ex3-1} $(n=41)$ (left) and the $l^\infty$-norm error semi-log plot for the approximation of $f(x)$ for increasing truncation degree $n$ (right). We use the error measure \cref{eq:man-error} for the sum space approximation and \cref{eq:dual-man-error} for the dual sum space approximation.  The approximation space contains $10n +31$ functions for each value of $n$. }\label{fig:ex3-1}
\end{figure}
\begin{figure}[h!]
\centering
\includegraphics[width =0.49 \textwidth]{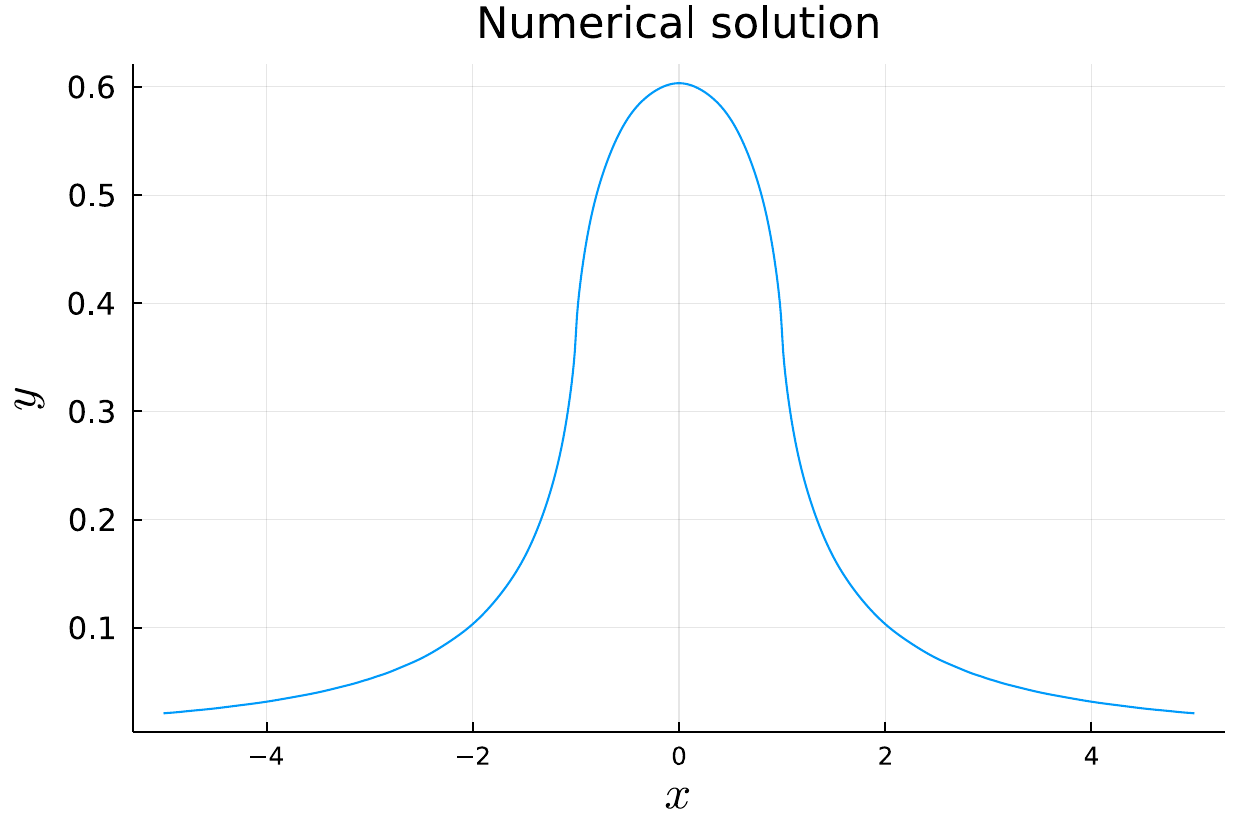}
\includegraphics[width =0.49 \textwidth]{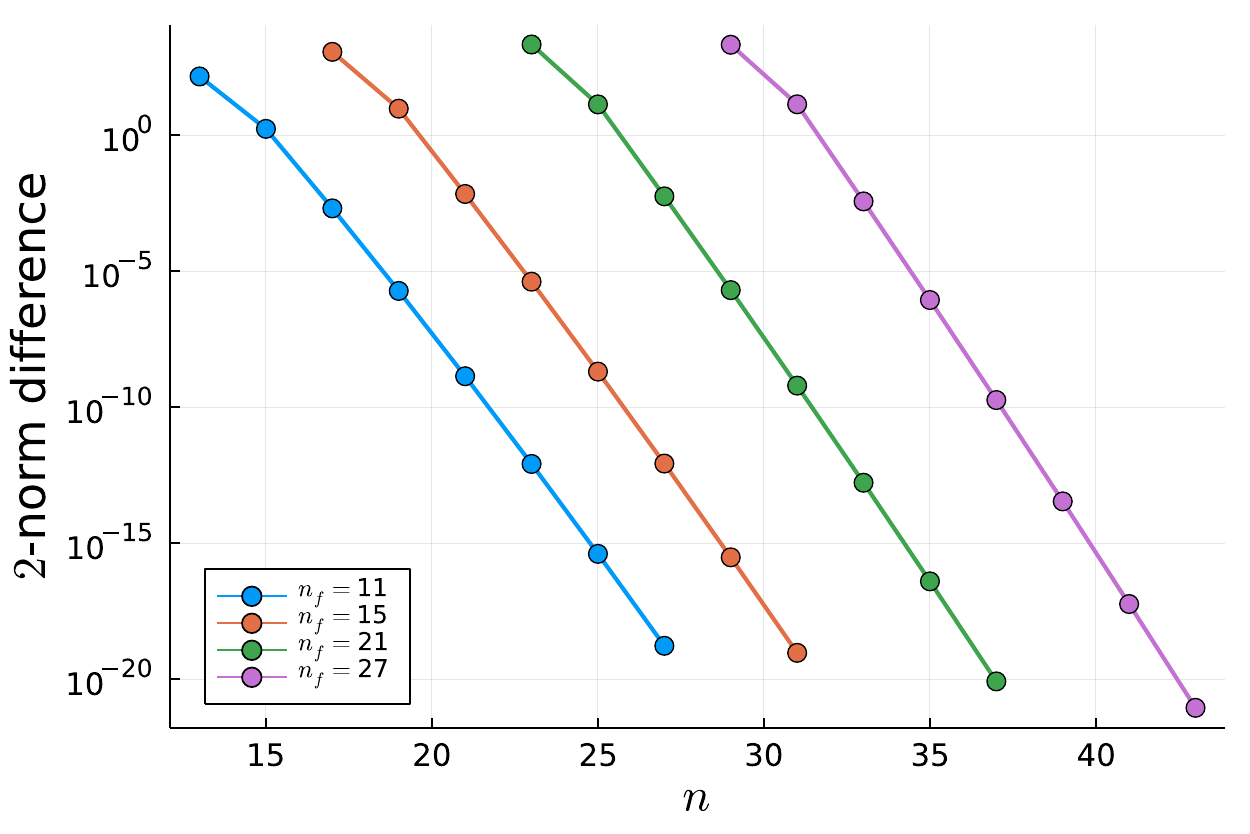}
\caption{Numerical approximation of the solution to \cref{eq:ex3-1}, $n=41$ (left). $l^2$-norm difference of the coefficients of the expansion of the solution, where we fix the right-hand side expansion at degree $n_f =11, 15, 21$, and 27 (right).}\label{fig:ex3-2}
\end{figure}

As in the previous example, we investigate the behaviour of the expansion coefficients in both the sum space and dual sum space for increasing degree $n$ in \cref{fig:rhs-jump-coeffs}. The result is interesting; the norm begins large and appears to blow-up, reaching a magnitude of $\mathcal{O}(10^{10})$ at $n=17$ for the sum space expansion and $\mathcal{O}(10^{6})$ at $n=13$ for the dual sum space expansion. Thereafter, the norm oscillates in the sum space expansion until $n=41$ where afterwards the norm drops at an exponential rate. For $n>300$ the norm has magnitude $\mathcal{O}(1)$. In the dual sum space expansion, the norm quickly decreases for $n>13$ and plateaus for $n\geq71$ at a value of magnitude of $\mathcal{O}(10^{-1})$. 

\begin{figure}[h!]
\centering
\includegraphics[width =0.49 \textwidth]{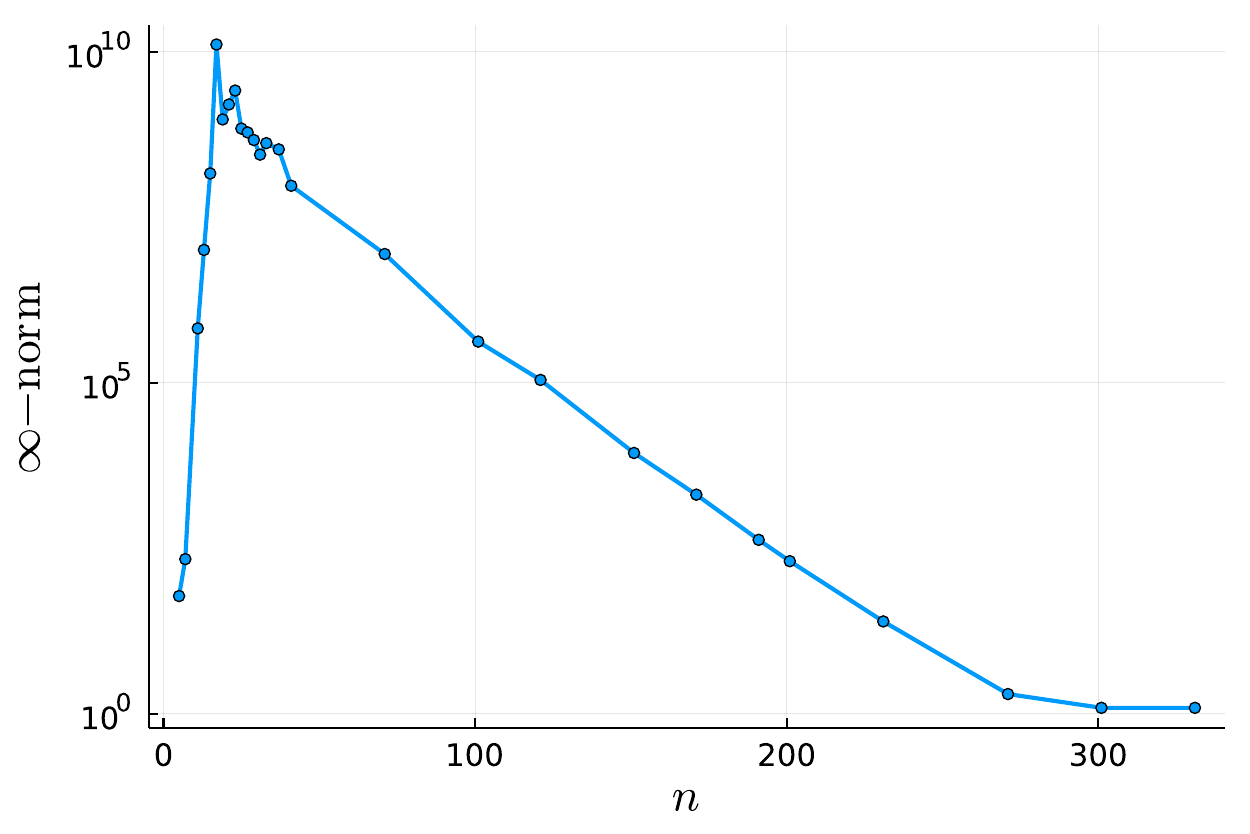}
\includegraphics[width =0.49 \textwidth]{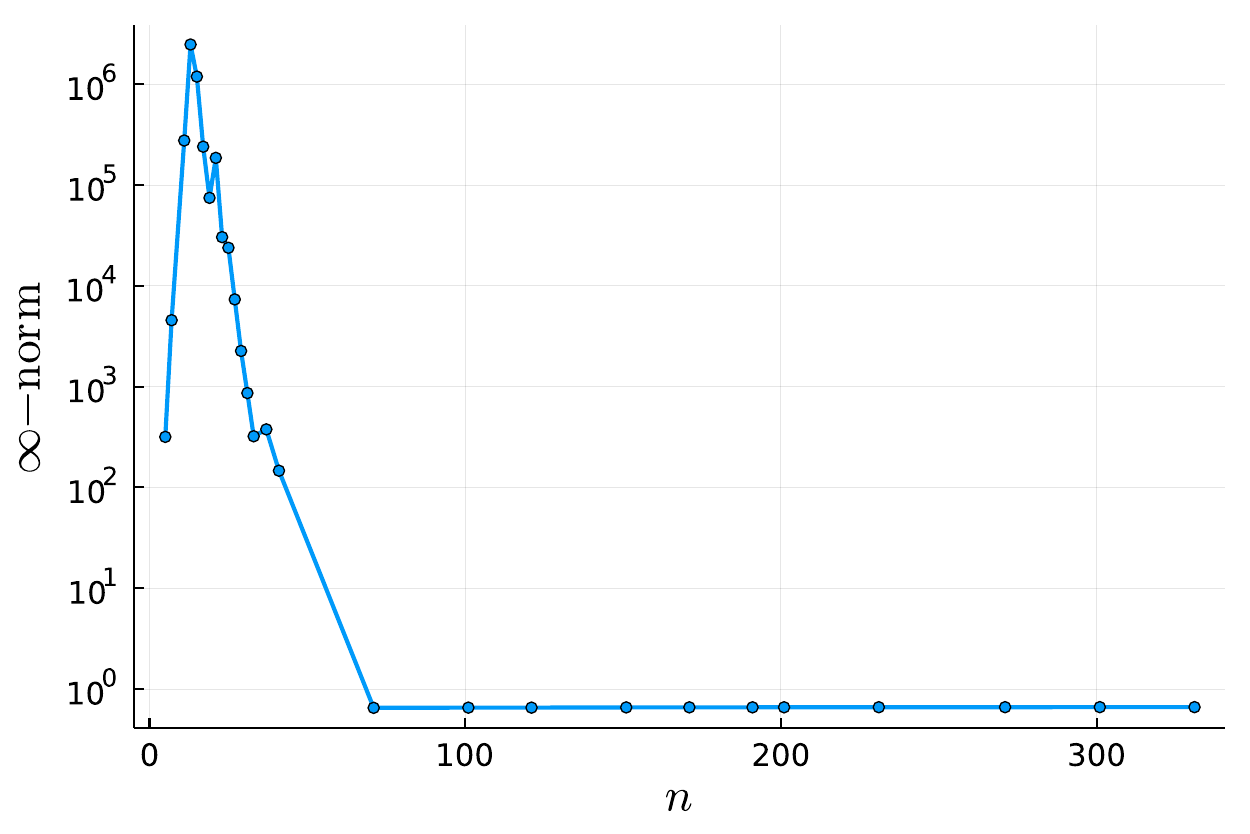}
\caption{The $l^\infty$-norm of the coefficient vectors of the sum space expansion (left) and the dual sum space expansion (right)  of \cref{eq:ex3-1} for increasing truncation degree $n$.  The approximation space contains $10n +31$ functions for each value of $n$. }\label{fig:rhs-jump-coeffs}
\end{figure}

\subsection{Nonsmooth right-hand side}
\label{sec:examples:nonsmooth}
Here we conduct the same investigation as in the previous example except now we choose the right-hand side
\begin{align}
f(x) = 
\begin{cases}
\arcsin(x) & \text{if} \; |x| \leq 1,\\
\arcsin(1) \mathrm{sgn}(x) e^{1-|x|} & \; \text{otherwise}.
\end{cases}
\label{eq:f-asin}
\end{align}
We use the same multiple-interval sum space as in the previous example as well as the same collocation points for the (dual) sum space expansion. 

In \cref{fig:rhs-asin} we plot the approximation of $f(x)$ expanded in the sum space at truncation degree $n=41$ together with a plot of the exact right-hand side. Moreover, we plot the convergence of the expansions of $f(x)$. The error norms are measured with the same metrics as in the previous example. Both expansions appear to stagnate at an error of $\mathcal{O}(10^{-8})$.  For smaller truncation degree values of $n$ the error in the dual sum space expansions is smaller, however, at $n=101$, the sum space expansion has the smallest error.

The appended sum space functions are $v_0(x)$, $\tilde{u}_{-1}(x)$, $v_{n+2}(x)$, and $\tilde{u}_{n+1}(x)$. Mathematica's \texttt{NIntegrate} routine is used to compute the necessary inverse Fourier transforms. In \cref{fig:sols-asin} we plot the numerical solution $u(x)$ of \cref{eq:fpde} (where $\lambda = 1$, $\mu = \eta = 0$). Again as we do not have an explicit solution, we truncate the right-hand side expansion at degree $n_f =11, 15, 21$, and 27. We then measure the 2-norm difference in the sum space coefficient vectors at truncation degree $n$ and $n-2$ of the solution for $n = n_f+2, n_f+4, \dots, n_f + 16$. We observe spectral convergence in the coefficient vectors. 
\begin{figure}[h!]
\centering
\includegraphics[width =0.49 \textwidth]{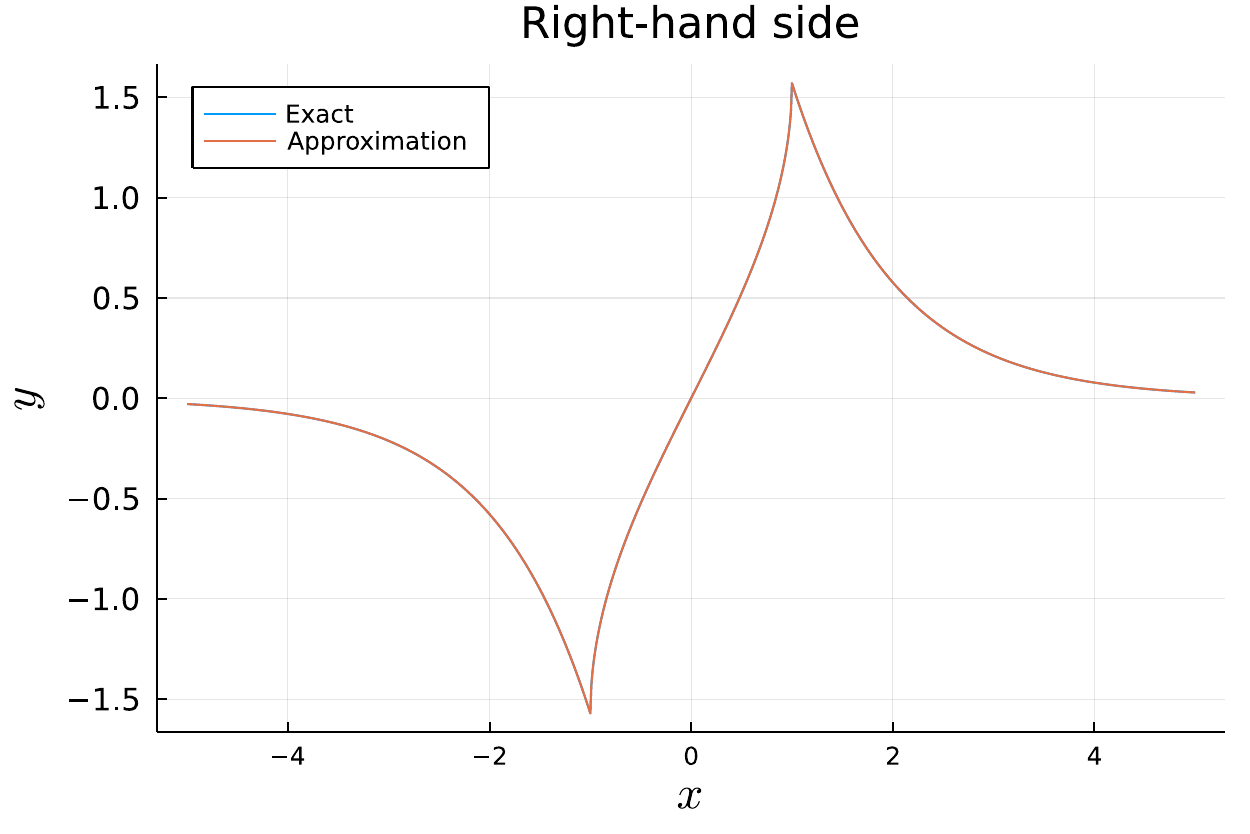}
\includegraphics[width =0.49 \textwidth]{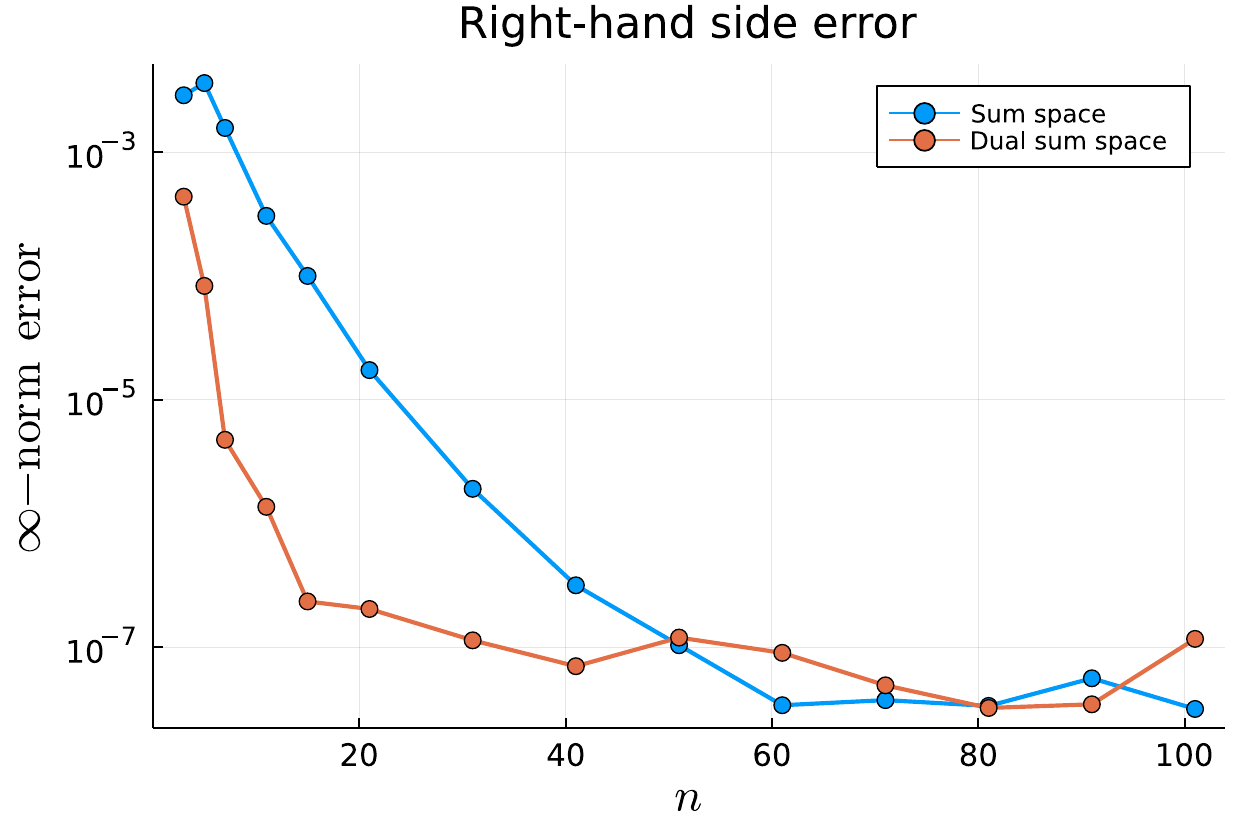}
\caption{The exact and sum space approximated $f(x)$ as defined in \cref{eq:f-asin} $(n=41)$ (left) and the $l^\infty$-norm error semi-log plot for the approximation of $f(x)$ for increasing truncation degree $n$ (right).  The approximation space contains $10n +31$ functions for each value of $n$.  We use the measure \cref{eq:man-error} for the sum space approximation and \cref{eq:dual-man-error} for the dual sum space approximation.}\label{fig:rhs-asin}
\end{figure}
\begin{figure}[h!]
\centering
\includegraphics[width =0.49 \textwidth]{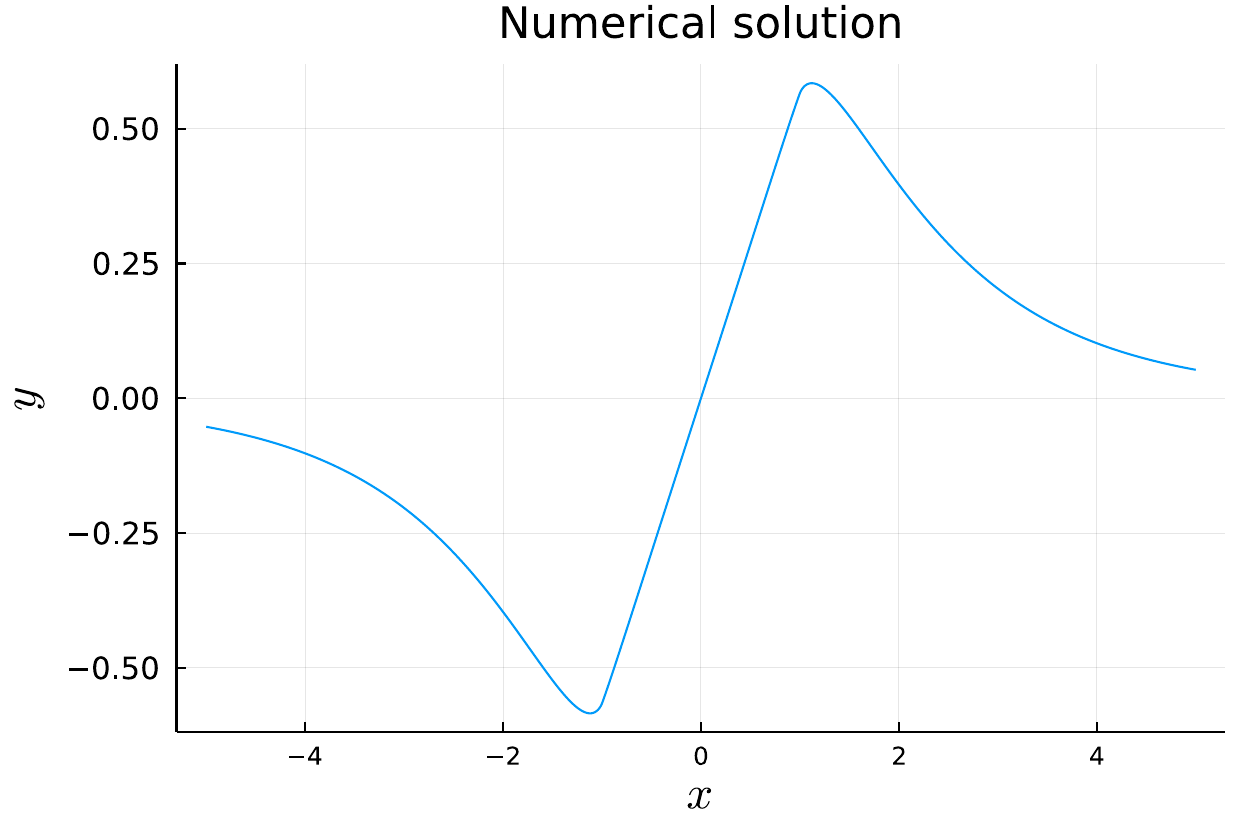}
\includegraphics[width =0.49 \textwidth]{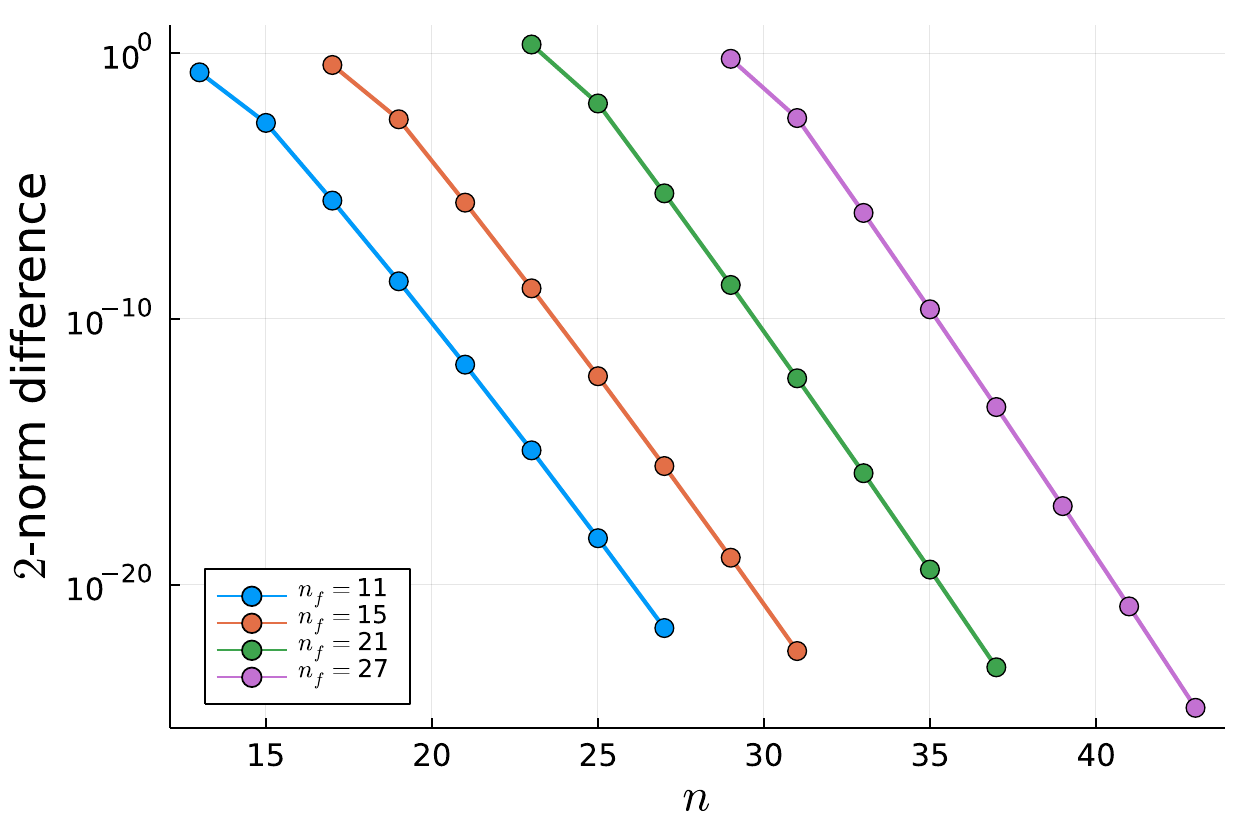}
\caption{Numerical approximation of the solution to the nonsmooth right-hand side problem, $n=41$ (left). $l^2$-norm difference of the coefficients of the expansion of the solution, where we fix the right-hand side expansion at degree $n_f =11, 15, 21$ and 27 (right).}\label{fig:sols-asin}
\end{figure}

As in the previous two examples, we investigate the behaviour of the expansion coefficients in both the sum space and dual sum space for increasing degree $n$ in \cref{fig:rhs-asin-coeffs}. As in the discontinuous right-hand side example, the $l^\infty$-norm of the coefficient vector begins large and appears to blow-up, reaching a magnitude of approximately $\mathcal{O}(10^{6})$ at $n=17$ for the sum space expansion and $\mathcal{O}(10^{5})$ at $n=13$ for the dual sum space expansion. Thereafter, the norm decreases and for $n\geq 71$ plateaus around $\mathcal{O}(10^4)$ in both expansions.  The convergence rate may be surprising at first glance. The sum space and the dual sum space contain functions that themselves are nonsmooth. Since the interval $[-1,1]$ is included in $\vect{I}$, then the nonsmoothness in the sum space functions aligns with the nonsmoothness in the right-hand side. Hence, the basis functions are able to approximate the right-hand side at a surprising rate of accuracy.

A stagnation higher than machine precision (in this case at $10^{-8}$) is typical for approximation spaces that are frames and is at least partially explained by \cref{th:convergence}. Essentially, the choice of $\epsilon$ as the cutoff in an $\epsilon$-truncated SVD factorization dictates a minimum error in the expansion of the right-hand side. However, making $\epsilon$ too small leads to numerical instability.  It is possible that with more fine-tuning, we may be able to improve the accuracy before the stagnation in the error.

\begin{figure}[h!]
\centering
\includegraphics[width =0.49 \textwidth]{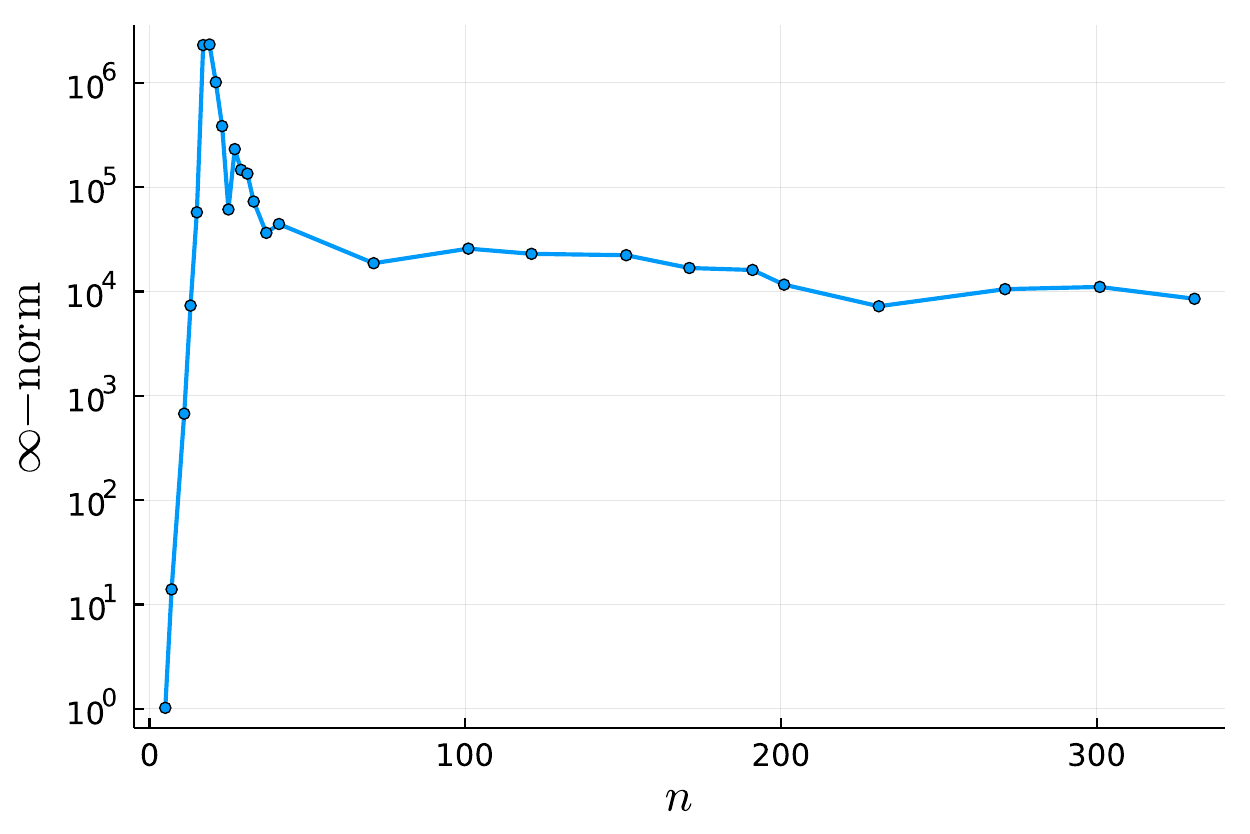}
\includegraphics[width =0.49 \textwidth]{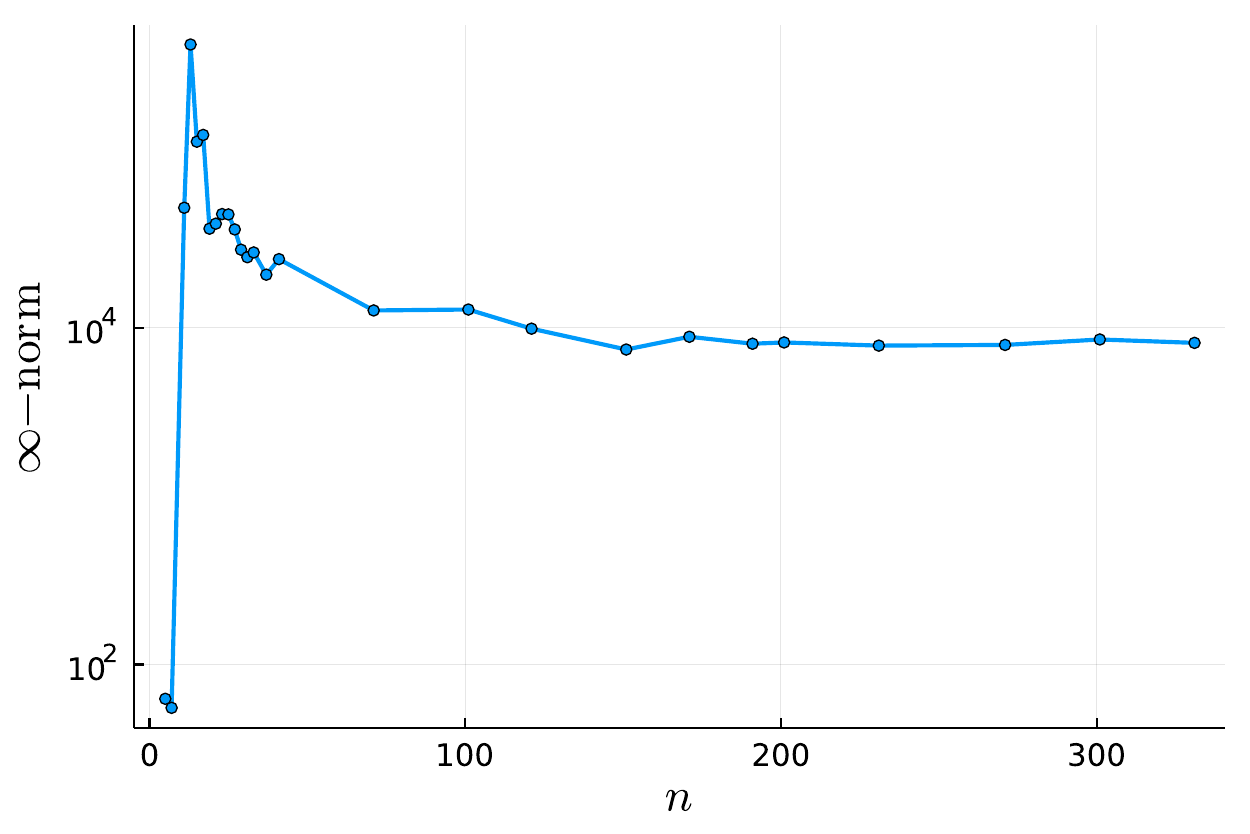}
\caption{The $l^\infty$-norm of the coefficient vector of the sum space (left) and the dual sum space (right) expansions of \cref{eq:f-asin} for increasing truncation degree $n$.  The approximation space contains $10n +31$ functions for each value of $n$. }\label{fig:rhs-asin-coeffs}
\end{figure}

\subsection{Fractional heat equation}
\label{sec:frac-heat-example}
Consider the fractional heat equation, i.e.~the following time-dependent fractional PDE:
\begin{align}
\begin{split}
\partial_t u(x,t) + (-\Delta)^{1/2} u(x,t) = 0, \;\;
 u(x,t) \to 0 \; \text{as} \; |x| \to \infty.
 \end{split}
 \label{eq:ex1}
\end{align}
We pick the following two choices for the initial condition:
\begin{align}
 u(x,0) &=(x^2+1)^{-1}, \label{eq:ic1}\\
 u(x,0) &= W_0(x).
  \label{eq:ic2}
\end{align}
The initial condition \cref{eq:ic1} coincides with a constant scaling of the fundamental solution, $t(2\pi(x^2+t^2))^{-1}$, to the fractional heat equation at $t=1$ \cite[Sec.~2]{Vazquez2018}. Hence, \cref{eq:ex1} with the initial condition \cref{eq:ic1} has the exact solution:
\begin{align}
u(x,t) = \frac{1+t}{x^2 + (1+t)^2}.
\label{eq:ex1-exact}
\end{align}
We discretize \cref{eq:ex1} in time with a backward Euler discretization to yield, for $k \geq 0$,
\begin{align}
\begin{split}
\lambda u_{k+1}(x) + (-\Delta)^{1/2} u_{k+1}(x) &= \lambda u_k(x), \\
 u_{k+1}(x) &\to 0 \; \text{as} \; |x| \to \infty,
 \end{split}
 \label{eq:ex1-2}
\end{align}
where $u_0(x) =   (x^2+1)^{-1}$ or $W_0(x)$ and $\lambda = (\Delta t)^{-1}$ where $\Delta t$ is the time step. By choosing $\mu = \eta = 0$, we recover the equation \cref{eq:fpde} which is discretized in space using our spectral method. We choose the appended sum space functions $v_0(x)$, $\tilde{u}_{-1}(x)$, $v_1(x)$, and $\tilde{u}_{0}(x)$. We use an FFT (as outlined in \cref{sec:FFT}) to approximate the necessary inverse Fourier transforms. 

Given the tuple of intervals $\vect{I}$ and the coefficient vector solution $\vectt{u}_k$ at time step iterate $k$, we solve the following for the coefficient vector solution $\vectt{u}_{k+1}$:
\begin{align}
(\lambda \vectt{E} + \vectt{A}^{\vect{I}})\vectt{u}_{k+1} = \lambda (\vectt{E} \vectt{R})\vectt{u}_k. \label{eq:fractional-heat1}
\end{align}
$\vectt{E}$ and $\vectt{A}^{\vect{I}}$ are defined in \cref{prop:element-mappings} and $\vectt{R}$ is defined in \cref{sec:Id:Sp:Sd}. We reiterate that $(\lambda \vectt{E} + \vectt{A}^{\vect{I}})$ is block diagonal. Hence, after the right-hand side $\lambda (\vectt{E} \vectt{R})\vectt{u}_k$ is computed, we decompose the solve into $K$ smaller linear solves, where $K$ is the number of intervals, as in \cref{eq:block-decompose}. Our approximate solution is given by
\begin{align}
u(x,k \Delta t ) \approx \Skp(x) \vectt{u}_k. 
\end{align}

We pick the multiple-interval sum space centred at the intervals $[-5,-3]$, $[-3,-1]$, $[-1,1]$, $[1,3]$, and $[3,5]$. We fix the truncation degree $n=5$. The initial condition \cref{eq:ic1} is expanded in the sum space via the least-squares matrix as discussed in \cref{sec:expansion}. We use 5001 equally spaced points in the interval $[-5,5]$ and 501 equally spaced points in $[-20,-5]$ and $[5,20]$ each. This results in 26,001 collocation points. The initial condition \cref{eq:ic2} is represented exactly as one of the intervals is $[-1,1]$, resulting in zero initial error. We choose a time step of $\Delta t = 10^{-2}$. Snapshots of the solution at $t = 0, 1/2$ and 1 are found in \cref{fig:ex1-3b} and contour plots of $t$ against $x$ are given in \cref{fig:ex1-4}. 

In \cref{fig:ex1-2} we estimate the errors of our numerical solutions. In \cref{fig:ex1-2a} we compare our numerical solution with the exact solution \cref{eq:ex1-exact}. An exact solution for the initial condition \cref{eq:ic2} is not available. Hence, in \cref{fig:ex1-2b}, we compare our approximate solution with one computed via an approximate inverse Fourier transform at each time step. By using that $\mathcal{F}[W_0] = \pi J_1(|\omega|)/|\omega|$ and an induction argument, one can show that \cref{eq:ex1-2} has the solution
\begin{align}
u_{k+1}(x) = \mathcal{F}^{-1}\left[\frac{\pi J_1(|\omega|)}{|\omega|(1+\lambda^{-1} |\omega|)^{k+1}}\right].
\end{align}
For each iteration $k$, we approximate the solution $u_k(x)$ with an FFT as outlined in \cref{sec:FFT}. For 100 iterations, this took 92.2 seconds. By contrast, 100 solves of \cref{eq:fractional-heat1} took 0.0162 seconds. The setup took 6.95 seconds to compute the four FFTs to approximate the solutions required to construct the appended sum space $\Skp(x)$ (as all the intervals have the same width) and another 1.06 seconds to compute their expansions in the sum space $\Sk(x)$. This totals to 8.01 seconds, approximately an 11 times speedup. We note that the setup is not tied to the initial condition or knowledge of its Fourier transform. Hence, after one initial setup, the time evolution of any initial condition can be computed in fractions of a second via the repeated solve of \cref{eq:fractional-heat1}. The $l^\infty$-norm error is measured on a 1001-point equally spaced grid at $[-20,20]$, i.e.
\begin{align}
\max_{x \in \{-20,-19.99,\cdots,19.99,20\}} | u(x) - \Sk(x) \vectt{u}|.
\label{eq:man-error-heat}
\end{align}

\begin{figure}[h!]
\centering
\includegraphics[width =0.49 \textwidth]{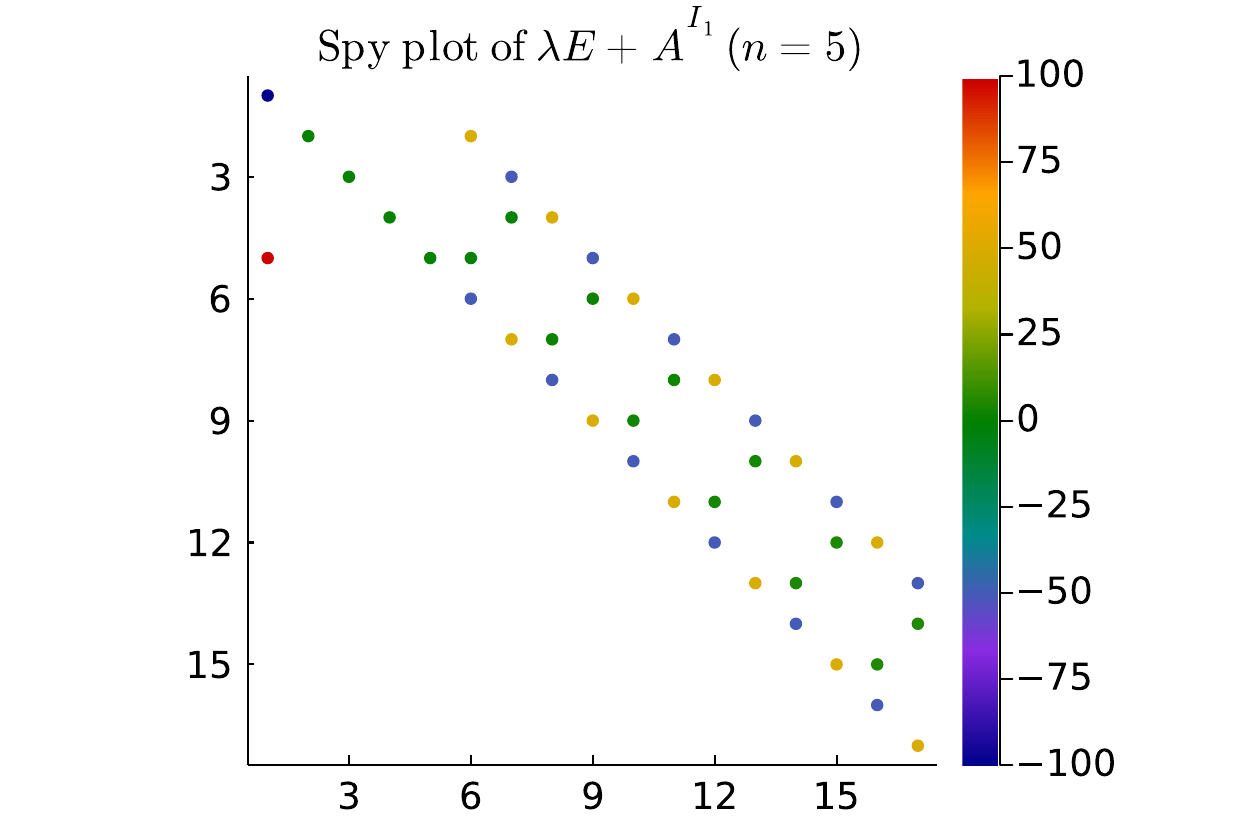}
\includegraphics[width =0.49 \textwidth]{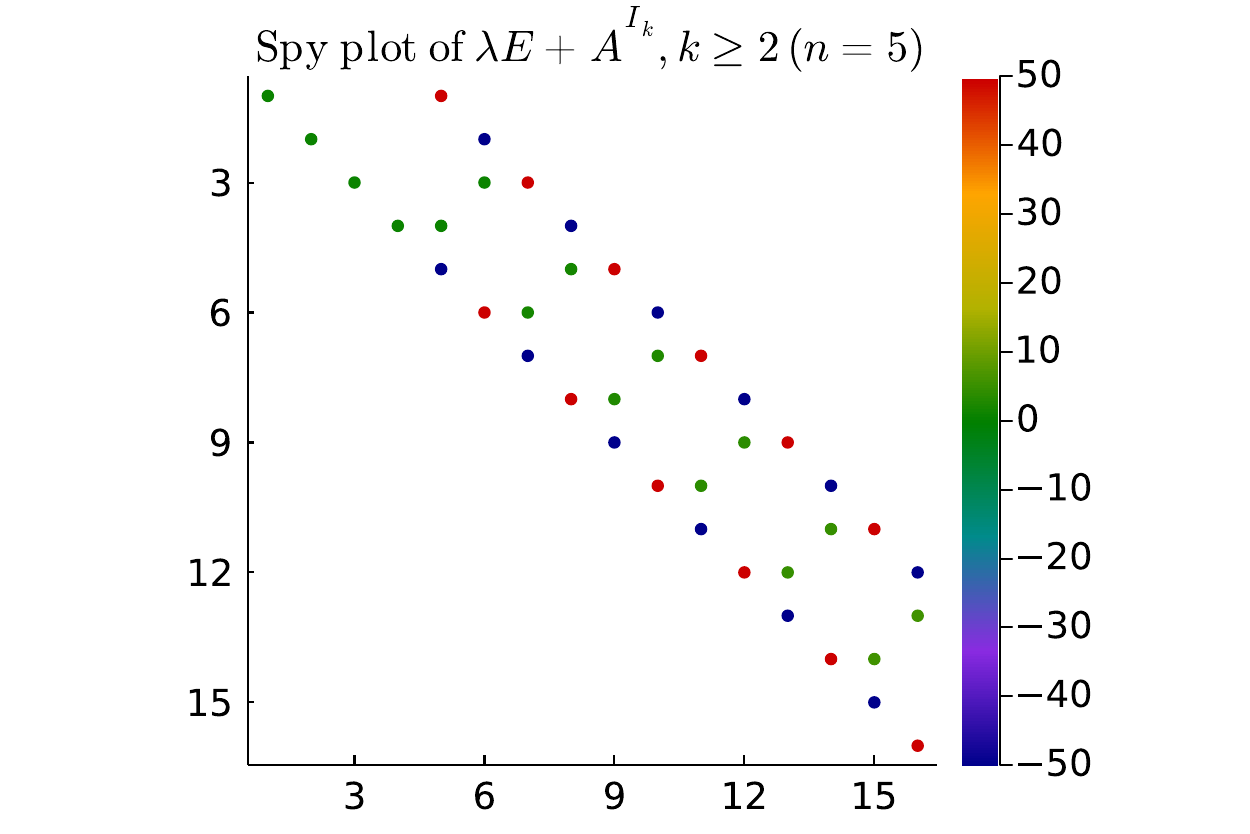}
\caption{\texttt{Spy} plots of the matrices in the linear systems after decomposing \cref{eq:fractional-heat1} block-wise with $\lambda = 100$  and using the additional functions $v_0$, $\tilde{u}_{-1}$,  $v_1$, $\tilde{u}_{0}$. The matrices are banded and sparse.}\label{fig:ex1-1}
\end{figure}
\begin{figure}[h!]
\centering
\subfloat[Initial condition \cref{eq:ic1}]{\includegraphics[width =0.49 \textwidth]{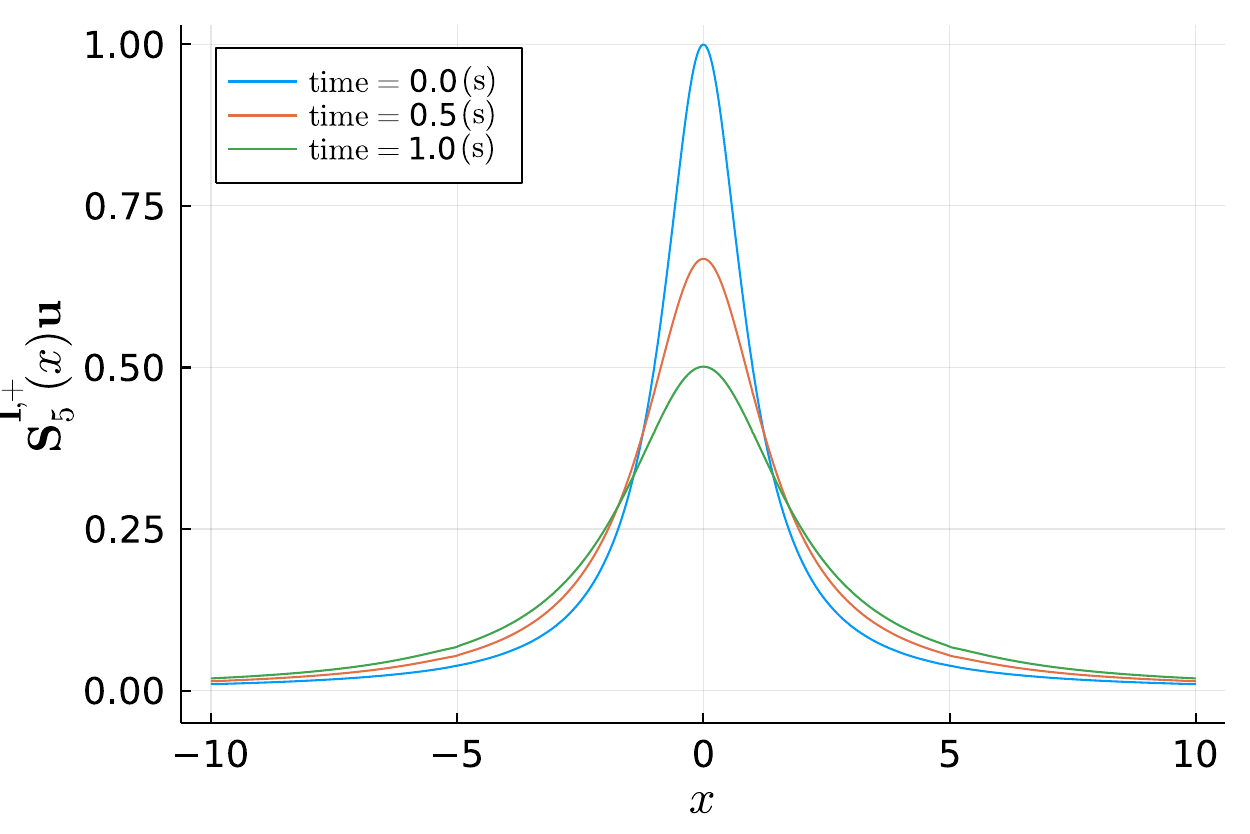}}
\subfloat[Initial condition \cref{eq:ic2}]{\includegraphics[width =0.49 \textwidth]{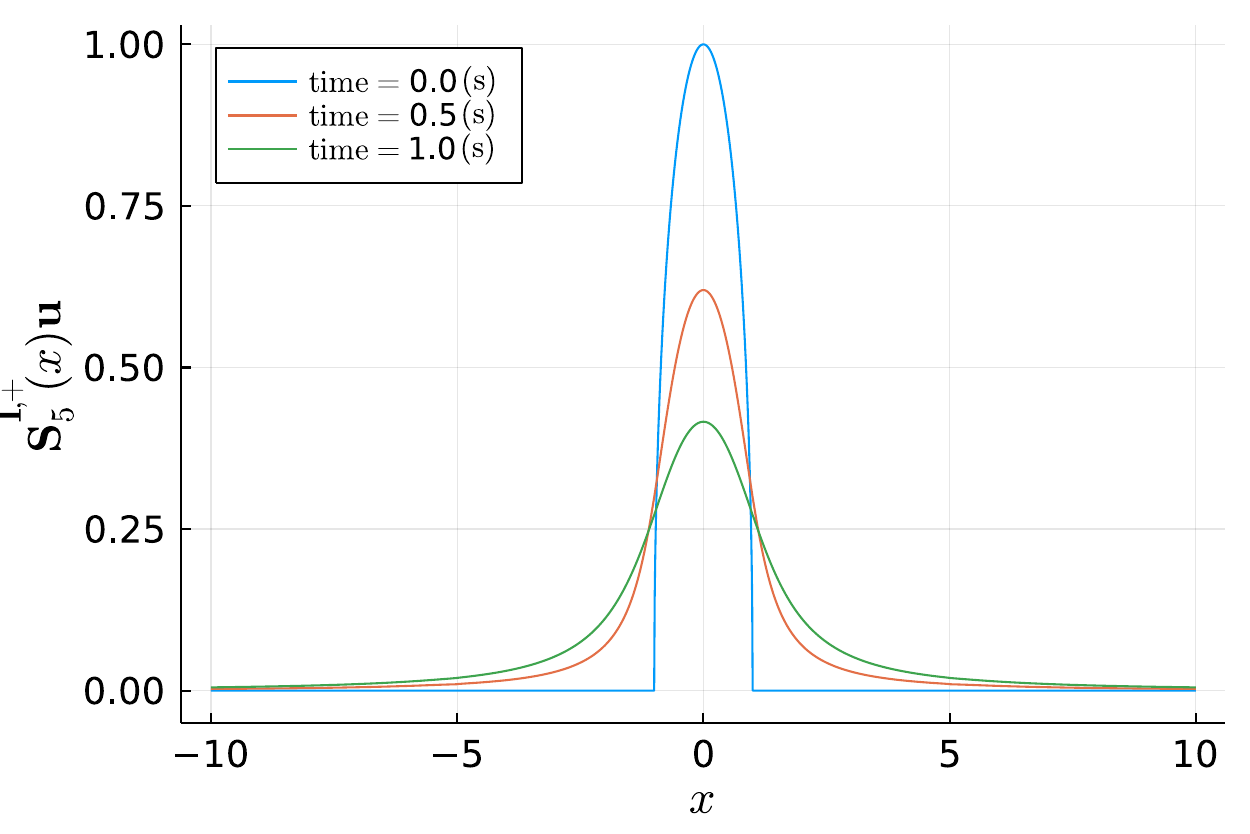} \label{fig:ex1-3b}}
\caption{Snapshots of the numerical solutions of the fractional heat equation at $t=0$, $0.5$, and 1 using the truncated space $\vectt{S}^{\vect{I},+}_5(x)$, $\Delta t = 10^{-2}$.}\label{fig:ex1-3}
\end{figure}
\begin{figure}[h!]
\centering
\subfloat[Initial condition \cref{eq:ic1}]{\includegraphics[width =0.49 \textwidth]{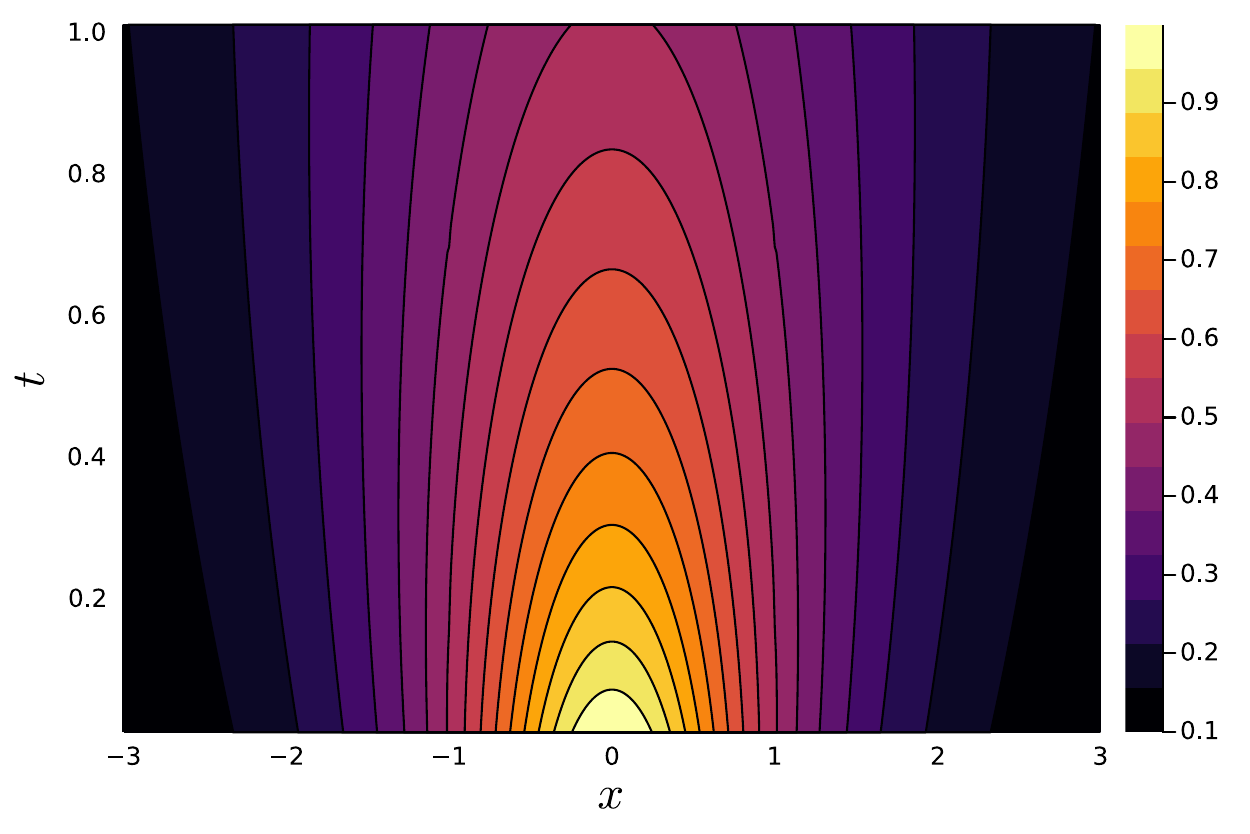}}
\subfloat[Initial condition \cref{eq:ic2}]{\includegraphics[width =0.49 \textwidth]{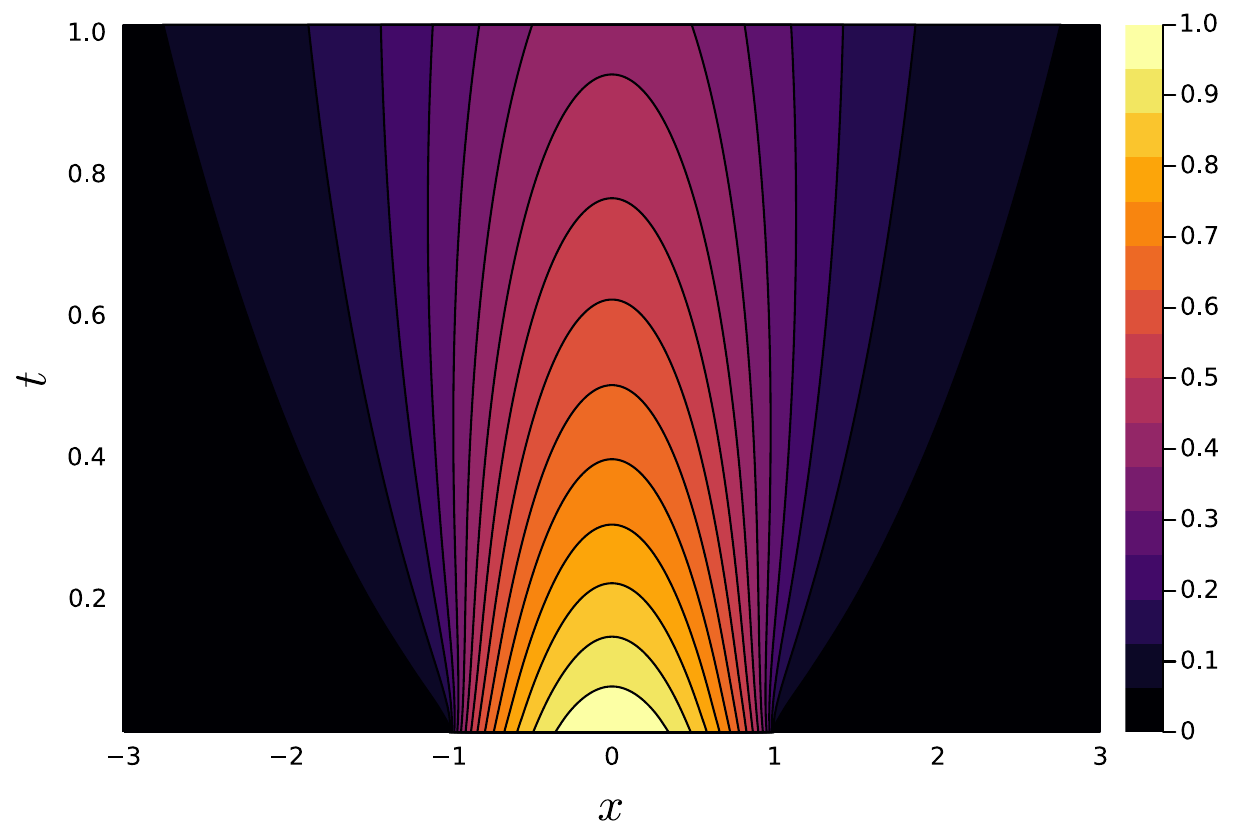}}
\caption{Contour plots of the numerical solutions to the fractional heat equation.}\label{fig:ex1-4}
\end{figure}
\begin{figure}[h!]
\centering
\subfloat[Initial condition \cref{eq:ic1}]{\includegraphics[width =0.49 \textwidth]{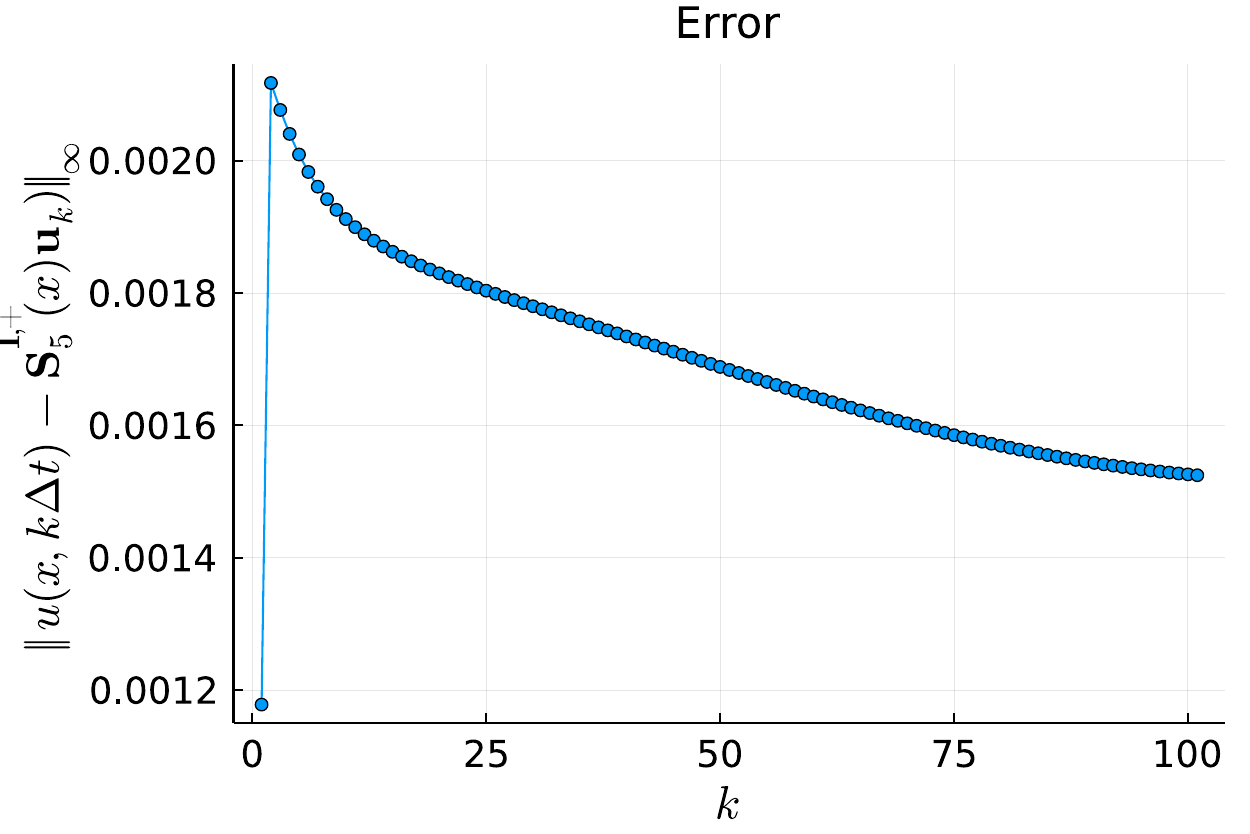} \label{fig:ex1-2a}}
\subfloat[Initial condition \cref{eq:ic2}]{\includegraphics[width =0.49 \textwidth]{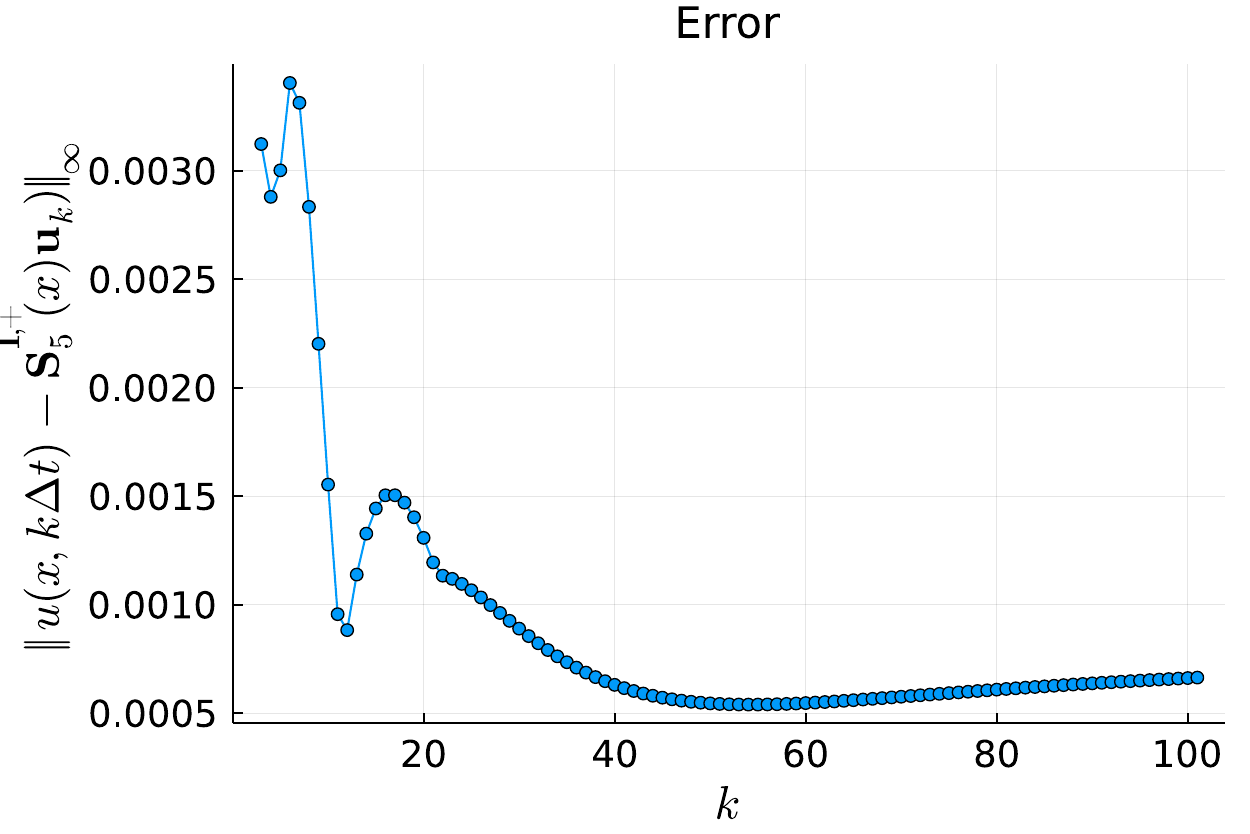} \label{fig:ex1-2b}}
\caption{Error as measured by \cref{eq:man-error-heat} of the numerical solutions to the fractional heat equation for iterate $k$ as measured against (a) the exact solution and (b) the approximate FFT solution.}\label{fig:ex1-2}
\end{figure}

\begin{remark}
The choice of time discretization is independent to the spectral method. In particular coupling the spectral method to a Runge--Kutta time discretization is no more difficult than for a standard finite difference of finite element discretization.
\end{remark}

\subsection{Wave propagation}

Consider the fractional Hilbert wave equation:
\begin{align}
[(-\Delta)^{1/2} + \Hil + \frac{\partial^2}{\partial t^2}] u(x,t) = f(x,t). 
\end{align}
After a Fourier transform with respect to $t$, we recover the equation
\begin{align}
[(-\Delta)^{1/2}+ \Hil - \omega^2] \hat{u}(x,\omega) = \hat{f}(x,\omega). \label{eq:wave1}
\end{align}
Equation \cref{eq:wave1} is of the form \cref{eq:fpde} with $\lambda = -\omega^2$, $\mu = 1$, and $\eta=0$. Thus the idea is to solve \cref{eq:wave1} to find $\hat{u}(x,\omega)$, for a range of values of $\omega$, and then take the inverse Fourier transform to recover the approximation of the physical solution $u(x,t)$. We choose the datum 
\begin{align}
f(x,t) = W_4(x) \me^{-t^2} \implies \hat{f}(x,\omega) = \sqrt{\pi} W_4(x) \me^{-\omega^2/4}.
\end{align}
This corresponds to a forcing term supported on $x \in [-1,1]$ that exponentially decays in time.

To fix a unique solution, we enforce that $\lim_{|x| \to \infty} \hat{u}(x,\omega) \to 0$. For this example, we approximate the additional functions $\tilde{u}_{-1}(x)$, $\tilde{u}_{n+1}(x)$, $v_0(x)$ and $v_{n+2}(x)$ via an FFT (as described in \cref{sec:FFT}). We pick $n=7$ and a uniform distribution of $\omega$ in the range $[0, 20]$ in increments of 1/10. The inverse Fourier transform from $\hat{u}(x,\omega)$ to $u(x,t)$ is approximated via an FFT. We provide the contour plots of the approximate solutions of $\hat{u}(x,\omega)$ and $u(x,t)$ in \cref{fig:ex4-1}. Even though the right-hand side is even, there is a drift in the solution $u(x,t)$ caused by the Hilbert transform operator. Note that the Hilbert transform maps even functions to odd functions and vice versa.

\begin{figure}[h!]
\centering
\includegraphics[width =0.49 \textwidth]{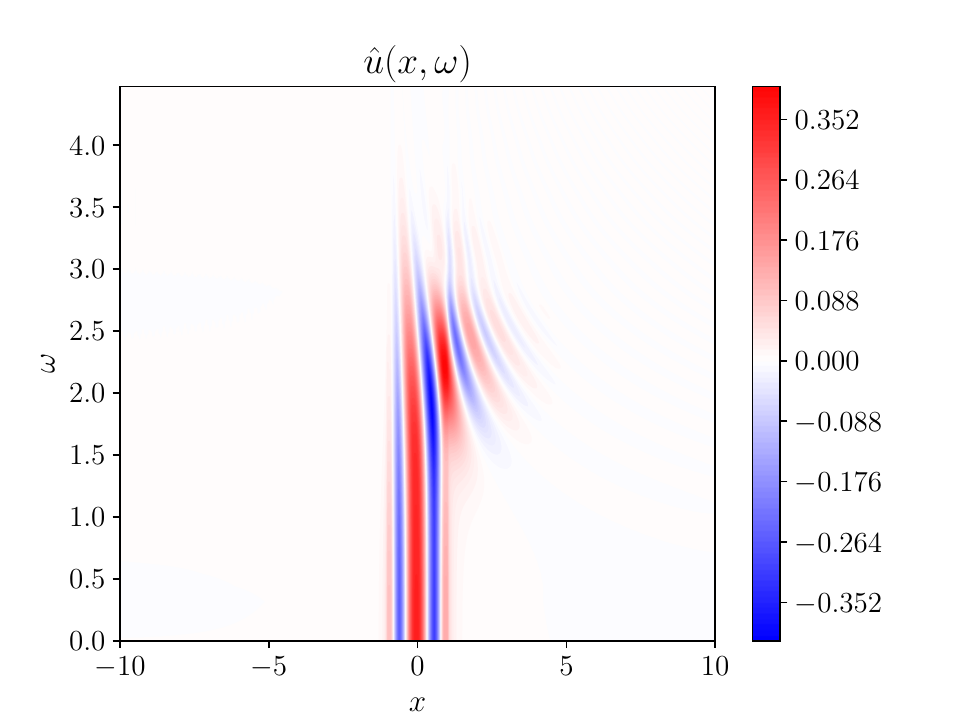}
\includegraphics[width =0.49 \textwidth]{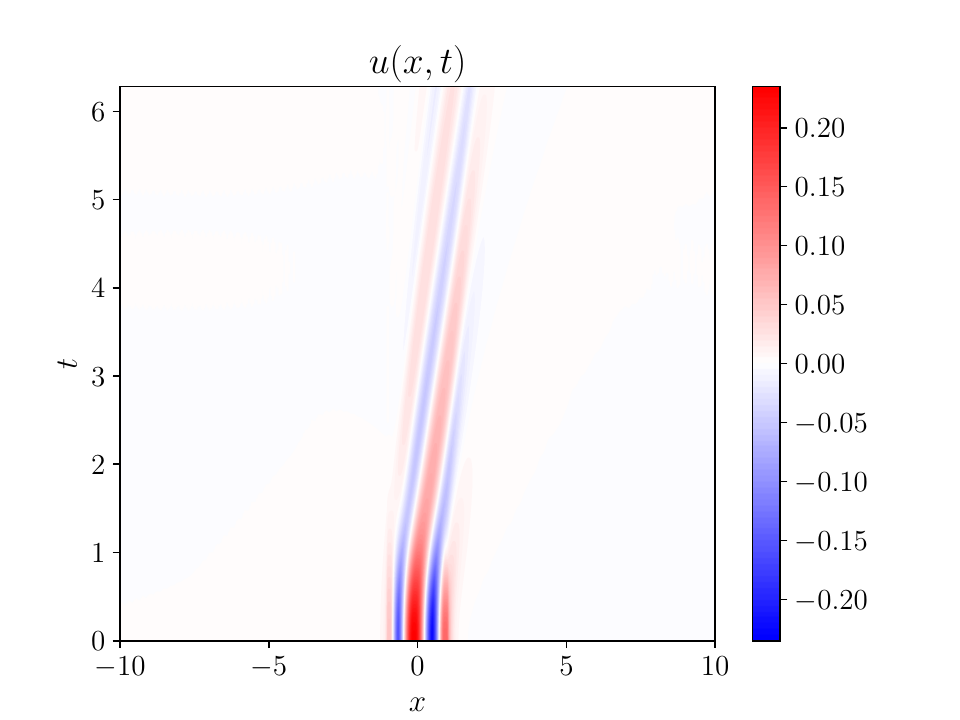}
\caption{Contour plot of the approximation of $\hat{u}(x,\omega)$ (left) and $u(x,t)$ (right) with $f(x,t) = W_4(x) \me^{-t^2}$ to the fractional Hilbert wave equation.}\label{fig:ex4-1}
\end{figure}

\section{Conclusions}
\label{sec:conclusions}
In this work we introduced a sparse spectral method for a one-dimensional fractional PDE posed on an unbounded domain that may contain the identity, Hilbert, derivative, and sqrt-Laplacian operators. The method is constructed by forming a sum space of weighted Chebyshev polynomials of the second kind extended to the whole of $\mathbb{R}$ by zero and their Hilbert transforms. We derive explicit identities for the actions of the identity, Hilbert, derivative, and sqrt-Laplacian operators which allows us to build our method. The operator applied to different affine transformations of the sum spaces decouples during the solve. Hence, the solve can be performed in parallel over each interval separately. Moreover, the induced matrices are sparse leading to fast sparse solves. Numerically, we observe spectral convergence when the data is smooth. We emphasize that the coupling of the different interval sum spaces occurs during the expansion of the right-hand side.

In future work we plan to orthogonalize the basis as well as extend the method to general fractional Laplacians, $(-\Delta)^s$, $s \in (0,1)$ by considering weighted Jacobi polynomials and their fractional Laplacian analogues \cite[Tab.~5]{Gutleb2023}. This will be achieved by combining the techniques introduced in \cite{Papadopoulos2023} for expanding the right-hand side with the methods developed in \cite{Hale2018} for the construction of linear systems. Moreover, by considering Zernike polynomials, the spectral method will be extended to two-dimensional problems posed on a disk \cite[Tab.~7]{Gutleb2023}.

\begin{appendices}
\crefalias{section}{appendix}
\label{sec:appendix}

\section[IFT via FFT]{Approximate inverse Fourier transform via the FFT}
\label{sec:FFT}
The setup of the spectral method requires four inverse Fourier transforms per interval. Although this can be reduced if the intervals are translations on $\mathbb{R}$ and are of the same size as resulting solutions are simply the equivalent translations. In this section we describe how to use the FFT to approximate the inverse Fourier transform. The FFT implements the discrete Fourier transform (DFT) or its inverse (IDFT) with $\mathcal{O}(N \log N)$ complexity where $N$ is the size of the input vector. Consider a vector $\vectt{u} \in \mathbb{R}^N$. We denote the IDFT as $\mathcal{F}_D^{-1}$ and define it as (as implemented in \texttt{ifft} of \texttt{FFTW.jl}):
\begin{align}
\mathcal{F}_D^{-1}[\vectt{u}]_j \coloneqq \frac{1}{N} \sum_{n=1}^N \exp \left(\mi \frac{2\pi(j-1)(n-1)}{N} \right) \vectt{u}_n, \;\; j = 1,\dots, N.
\end{align}
Suppose that $N$ is even. Then, we also define the shifted IDFT, $\hat{\mathcal{F}}_D^{-1}$ such that the components of $\mathcal{F}_D^{-1}[\vectt{u}]$ are reordered from $j=N/2+1,\dots,N,1,\dots, N/2 $, 
\begin{align}
\hat{\mathcal{F}}_D^{-1}[\vectt{u}]_j
= \begin{cases}
\mathcal{F}_D^{-1}[\vectt{u}]_{j+N/2} & 1 \leq j \leq N/2,\\
\mathcal{F}_D^{-1}[\vectt{u}]_{j-N/2} & N/2+1 \leq j \leq N.
\end{cases}
\label{eq:fft5}
\end{align}
This is implemented as \texttt{iffshift(ifft(u))} in \texttt{FFTW.jl}. Consider the approximation, $j=1,\dots,N$:
\begin{align}
\mathcal{F}^{-1}[u](x_j) \approx 
\frac{1}{2\pi} \int_{-W}^W u(\omega) \me^{\mi \omega x_j}\, \mathrm{d}\omega \approx \frac{\delta}{2\pi} \sum_{n=0}^{N-1} u(\omega_n) \me^{\mi \omega_n x_j}. 
\end{align}
Here we choose the parameters $W \gg 1$ and $N \gg 1$, $N$ even. The points $\omega_n \coloneqq n \delta -W$, $n=0,\dots,N-1$, where $\delta  \coloneqq 2W/N$. Substituting in the definition of $\omega_n$, we see that
\begin{align}
\frac{\delta}{2\pi} \sum_{n=0}^{N-1} u(\omega_n) \me^{\mi \omega_n x_j} 
= \frac{\delta N  \me^{-\mi Wx_j}}{2 \pi} \left [\frac{1}{N} \sum_{n=0}^{N-1} u(\omega_n) \me^{2 \mi x_j Wn/N } \right]. \label{eq:fft1}
\end{align}
Define the vector $\vectt{u}_n = u(\omega_{n-1})$, $n=1,\dots,N$. Moreover, we fix $x_j$ as 
\begin{align}
x_j \coloneqq (-N/2 + j-1) \frac{\pi}{W}, \;\; j = 1,\dots,N.
\end{align}
Then, the right-hand side of \cref{eq:fft1} is equal to
\begin{align}
\frac{\delta N  \me^{-\mi Wx_j}}{2 \pi} \left [\frac{1}{N} \sum_{n=1}^{N} \vectt{u}_n \me^{2\pi \mi   (j-1) (n-1)/N} \me^{-\mi \pi (n-1) } \right]. \label{eq:fft2}
\end{align}
By noting that $\me^{-\mi \pi (n-1) } = \me^{\mi \pi (n-1) }$  for $n \in \mathbb{Z}$, then for $1 \leq j \leq N/2$, \cref{eq:fft2} is equal to
\begin{align}
\frac{\delta N  \me^{-\mi Wx_j}}{2 \pi} \left [\frac{1}{N} \sum_{n=1}^{N} \vectt{u}_n \me^{2\pi\mi   (j+N/2-1) (n-1)/N } \right], \label{eq:fft3}
\end{align}
and for $N/2+1 \leq j \leq N$, \cref{eq:fft2} is equal to
\begin{align}
\frac{\delta N  \me^{-\mi Wx_j}}{2 \pi} \left [\frac{1}{N} \sum_{n=1}^{N} \vectt{u}_n \me^{2\pi \mi   (j-N/2-1) (n-1)/N } \right]. \label{eq:fft4}
\end{align}
Hence, by the definition of the shifted IDFT in \cref{eq:fft5}, we see that, for $j,n=1,\dots,N$,
\begin{align}
\mathcal{F}^{-1}[u](x_j) \approx \frac{\delta N  \me^{-\mi Wx_j}}{2 \pi} \hat{\mathcal{F}}_D^{-1}[\vectt{u}]_j,  
\end{align}
where $\vectt{u}_n = u(2W(n-1)/N - W)$, $n = 1,\dots,N$.

\section[Appended sum space]{Approximating the appended sum space functions}
\label{sec:app:appended-sum-space}
With a na\"ive approach, it would take four integrals per interval to approximate the additional functions found in the appended sum space. This setup cost may become prohibitive for a large number of intervals. However, the next proposition reveals that the additional functions on intervals that are translated (and not scaled) are simply translations of the additional functions associated to one reference interval. 

\begin{proposition}[Translations of the reference interval]
Consider the interval $I = [a, b]$ and its associated reference interval $I_R = [-(b-a)/2, (b-a)/2]$. Then,
\begin{align}
\begin{split}
v^{I}_0(x) &= v^{I_R}_0(x-(a+b)/2), \;\; 
\tilde{u}^{I}_{-1}(x) = \tilde{u}^{I_R}_{-1}(x-(a+b)/2), \\
v^{I}_1(x) &=v^{I_R}_1(x-(a+b)/2),\;\;
\tilde{u}^{I}_{0}(x) = \tilde{u}^{I_R}_{0}(x-(a+b)/2).
\end{split}
\end{align}
\end{proposition}
\begin{proof}
The result follows from the identity $\mathcal{F}[f(\diamond + \alpha)](\omega) = \me^{\mi \alpha \omega} \mathcal{F}[f(\diamond)](\omega)$ and a routine calculation. 
\end{proof}

The additional functions on the reference interval are approximated with one integral each. Moreover, all intervals of the same width map to the same reference interval. Hence, the setup of the method only requires four integrals per interval of different width. In particular, if all the intervals have the same width, we only require four integrals in total for the setup, irrespective of the number of intervals used. 

\section{Special cases when $\lambda = 0$}
\label{sec:special-cases}

In this subsection we discuss the conditioning of $L^+$ for choices of $\mu$ and $\eta$ when $\lambda = 0$. The ill-conditioning is alleviated by removing the rows of $L^+$ associated with functions that are no longer in the range of $\mathcal{L}_{0,\mu,\eta} S_n^+$. 

\subsection{$\lambda=\mu=\eta=0$}
The first special case that we consider is $\lambda=\mu=\eta=0$. Here, \cref{eq:fpde} reduces to finding $u \in H^{1/2}(\mathbb{R})$ that satisfies
\begin{align}
(-\Delta)^{1/2}[u] = f. \label{eq:special2}
\end{align}
The first issue is that the matrix $L^+$, as constructed in \cref{eq:Ap}, is singular. This is due to three rows and one column whose entries are all zero. The three rows correspond to the dual sum space functions $\eU_{-2}(x)$, $\eU_{n+1}(x)$, and $V_{n+2}(x)$ which do not lie in the range of $(-\Delta)^{1/2} S_n^+(x)$. The column corresponds to the sum space function $\eT_0(x)$. Hence, we remove these rows and column from $L^+$ resulting in an $(2n+4) \times (2n+6)$ matrix. The second issue is that the additional functions computed to form the square matrix $L^+$ become problematic or redundant. By \cref{prop:halfLaplacian}, we have that $\tilde{u}_{j}(x) = W_j(x)$, $j \geq 0$ and $v_j(x) = \eT_j(x)$, $j \geq 1$. Hence, two of the additional functions are already included in the sum space and correspond to two columns that should be removed. At first glance, this is advantageous as this reduces $L^+$ to a  $(2n+4) \times (2n+4)$ matrix. However, computing the remaining additional functions $\tilde{u}_{-1}(x)$ and $v_0(x)$ poses a numerical difficulty. Although $\eU_{-1}, V_0 \in H^{-1/2}(\mathbb{R})$, since $\lambda = 0$, the Lax--Milgram theorem does not guarantee existence of solutions. Attempting to compute these two solutions by a forward and inverse Fourier transform fails. Hence, as discussed in \cref{sec:introduction}, there exist solutions that satisfy the Fourier multiplier reformulation \cref{eq:strong-fourier} but not \cref{eq:weak-form}. 
Ignoring these technical issues and na\"ively computing the inverse Fourier transform via \cref{def:fourier} does recover the non-decaying solutions:
\begin{align}
\tilde{u}_{-1}(x) &= 
\begin{cases}
-\arcsin(x) & |x| < 1,\\
-\mathrm{sgn}(x) \pi/2 & |x| \geq 1,
\end{cases}\\
v_0(x) &= 
\begin{cases}
\log(2) - \gamma & |x| < 1,\\
\log(2) - \gamma - \mathrm{arcsinh}(\sqrt{x^2-1}) & |x| \geq 1.
\end{cases}
\end{align}
However, as neither of these tend to zero as $|x| \to \infty$, they cannot live in $H^{1/2}(\mathbb{R})$.  Therefore, we choose to remove the columns associated with $\tilde{u}_{-1}(x)$ and $v_0(x)$ as well. Since all the additional functions have been removed, we are only required to consider the range of $(-\Delta)^{1/2} S_n(x)$. The range does not include $\eU_{-1}(x)$ and $V_0(x)$. Hence, these functions can be removed from the dual sum space (whilst still preserving the exact map $(-\Delta)^{1/2} : S_n \to {{\bf S}}^*_n$) and their corresponding rows from $L^+$. In summary, we remove the columns associated with $\eT_0(x)$, $v_0(x)$, $\tilde{u}_{-1}(x)$, $v_1(x)$, and $\tilde{u}_{0}(x)$, and the rows associated with $\eU_{-2}(x)$, $V_0(x)$, $\eU_{-1}(x)$, $V_{n+2}(x)$, and $\eU_{n+1}(x)$ reducing $L^+$ from a $(2n+7)\times(2n+7)$ matrix to a $(2n+2)\times(2n+2)$ matrix. We emphasize that the action of $(-\Delta)^{1/2}$ on $S_n(x)$ is still represented exactly by the reduced $L^+$. 

\subsection{$\lambda = \mu = 0$, $|\eta| > 0$}
We now consider where $\lambda = \mu = 0$ but $|\eta|>0$. This corresponds to finding $u \in  H^1(\mathbb{R})$ that satisfies
\begin{align*}
\left( \eta \fdx + (-\Delta)^{1/2} \right)[u]= f.
\end{align*}
This special case is similar to the previous one. Again $L^+$ is singular due to three rows and a column which contain only zeroes. The rows are associated with the dual sum space functions $\eU_{-2}(x)$, $\eU_{n+1}(x)$, and $V_{n+2}(x)$ which do not lie in the range of $(\eta \fdx + (-\Delta)^{1/2}) S_n^+(x)$ and the column corresponds to $\eT_0(x)$. Hence, we remove those three rows and column. Moreover, $\tilde{u}_{1}(x)$ and $v_0(x)$ cannot be computed via a forward and inverse Fourier transform due to the same issues as in the previous case. Hence, the columns associated with  $\tilde{u}_{1}(x)$ and $v_0(x)$ must also be removed. This implies that the rows associated with $V_0(x)$ and $\eU_{-1}(x)$ are now all zero and must also be removed. Hence, as in the previous case, we remove the columns associated with $\eT_0(x)$, $v_0(x)$, $\tilde{u}_{-1}(x)$, $v_1(x)$, and $\tilde{u}_{0}(x)$, and the rows associated with $\eU_{-2}(x)$, $V_0(x)$, $\eU_{-1}(x)$, $V_{n+2}(x)$, and $\eU_{n+1}(x)$ reducing $L^+$ from a $(2n+7)\times(2n+7)$ matrix to a $(2n+2)\times(2n+2)$ matrix. Moreover, the action of $(\fdx + (-\Delta)^{-1/2})$ on $S_n(x)$ is still represented exactly by the reduced $L^+$. 

\begin{remark}
Since we have removed the additional functions from the appended sum space, the setup cost is minimal when $\lambda=\mu=0$ and $\eta \in \mathbb{R}$ as no integrals are required.
\end{remark}

\subsection{$\lambda = 0$, $|\mu|>0$, $\eta \in \mathbb{R}$}
The final case to consider is when $\lambda = 0$ but $\mu \neq 0$. In this case there is only one zero row and column associated with $\eU_{-2}(x)$ and $\eT_0(x)$, respectively. Removing that row and column results in an invertible matrix that still represents the mapping exactly. Moreover, unlike the previous two cases, we recover $v_0(x), \tilde{u}_{-1}(x) \in H^{1/2}(\mathbb{R})$. However, we note that with increasing $n$ (truncation degree) $L^+$ has two singular values that quickly decrease to zero which impact the conditioning of $L^+$. If the problem setup requires high truncation degree $n$, and small parameter $|\mu| \ll 1$, then $L^+$ may become numerically singular. In this case, we suggest using the additional functions $v_{n+2}(x)$ and $\tilde{u}_{n+1}(x)$ instead of $v_1(x)$ and $\tilde{u}_{0}(x)$ as suggested in \cref{rem:stabiliseL}. This results in a well-conditioned $L^+$ irrespective of the choices of $n$ and $\mu$. 

\end{appendices}

\section*{Acknowledgements}
The authors would like to thank the anonymous reviewers for their insightful comments.

\clearpage

\printbibliography
\end{document}